\documentclass{article}
\usepackage{pgfplots}
\pgfplotsset{compat=1.17}
\usepackage[utf8]{inputenc}
\usepackage{amsfonts,amssymb,amsmath,amsthm}
\usepackage[a4paper, total={6in,9in}]{geometry}

\usepackage[utf8]{inputenc}
\usepackage{graphicx}
\usepackage{mathtools}

\usepackage{subcaption}
\usepackage{tikz}
\usepackage{pgfplots}
\usepackage{enumerate}
\usepackage{cases}
\usepackage{multicol}
\usepackage{sidecap}
\usepackage{float}
\usepackage{xcolor}
\definecolor{blue2}{rgb}{0.67, 0.9, 0.93}
\usepackage[color=blue2]{todonotes}
\usepackage{cases}
\usepackage{extarrows}
\usepackage{soul}
\usepackage[pagewise]{lineno}%\linenumbers

\colorlet{siaminlinkcolor}{green!50!black}
\colorlet{siamexlinkcolor}{red!50!black}
\colorlet{siamreviewcolor}{black!50}

\usepackage[normalem]{ulem}
\usepackage[colorlinks=true,hyperfootnotes=false]{hyperref}
\hypersetup{
  colorlinks = true,
  allcolors = siaminlinkcolor,
  urlcolor = siamexlinkcolor,
}
\usepackage[nameinlink]{cleveref}
\usepackage{setspace}
\numberwithin{equation}{section}
\newtheorem{theorem}{Theorem}[section]
\newtheorem{lemma}[theorem]{Lemma}
\newtheorem{proposition}[theorem]{Proposition}
\newtheorem{corollary}[theorem]{Corollary}
\newtheorem*{remark}{Remark}

\newcommand{\N}{{\mathbb N}}
\newcommand{\R}{{\mathbb R}}

\newcommand{\Q}{\mathbb{Q}}

\renewcommand{\d}{{\mathrm d}}
\DeclareMathOperator{\sech}{sech}
\usepackage{xspace}

\usepackage{authblk}
\title{\sc Stability of Stochastically Forced Solitons in the Korteweg-de Vries Equation }
\author[1]{R.W.S. Westdorp \thanks{\tt \href{mailto:r.w.s.westdorp@math.leidenuniv.nl}{r.w.s.westdorp@math.leidenuniv.nl}}}
\author[1]{H.J. Hupkes \thanks{\tt \href{mailto:hhupkes@math.leidenuniv.nl}{hhupkes@math.leidenuniv.nl}}}
\affil[1]{\small Mathematisch Instituut, Universiteit Leiden, P.O. Box 9512, 2300 RA Leiden, The Netherlands}

\begin{document}

\maketitle
%\todo[color=yellow]{Consider \textbf{Diffusion-driven} propagation reversal...}
\begin{abstract}
We study the stability and dynamics of solitons in the Korteweg-de Vries (KdV) equation in the presence of noise and deterministic forcing. The noise is space-dependent and statistically translation-invariant. We show that, for small forcing, solitons remain close to the family of traveling waves in a weighted Sobolev norm, with high probability. We study the effective dynamics of the soliton amplitude and position via their variational phase, for which we derive explicit modulation equations. The stability result holds on a time scale where the deterministic forcing induces significant amplitude modulation.
\end{abstract}



\section{Introduction} 

This paper studies the stochastic KdV equation
\begin{align}\d u = -(\partial_x^3 u +2 u\partial_x u)\ \d t+\epsilon f(t)u \ \d t+\sigma u\ \d W^Q_t,\label{eqn:skdv}\end{align}
where $u$ is a real-valued process on $(t,x)\in\R^+\times\R$. The scalar parameters $\epsilon>0$ and $\sigma>0$ introduce deterministic and random multiplicative forcing to the KdV dynamics, respectively. The noise process $W_t^Q$ is a Wiener process taking values in $L^2(\R)$.  The noise is white in time and colored in space with translation-invariant statistics given by the formal identity
\[\mathbb{E}\big[W(t,x)W(t,y)\big]=q(x-y)(s\wedge t), \quad x,y\in \R, \quad s,t>0,\]
with $q\in H^1(\R)\cap L^1(\R)$. The deterministic forcing depends only on time, prescribed by the scalar-valued function $f:\R^+\to \R$. We study the effect of this forcing on the family of solitary waves associated to the unforced KdV equation (\eqref{eqn:skdv} with $\epsilon=\sigma=0$). In particular, we establish the stability of forced solitons on time scales where the forcing $\epsilon, \sigma>0$ causes drastic amplitude modulation. This paper complements our formal analysis in \cite{westdorp1} with a rigorous stability result, and extends our deterministic results in \cite{westdorp2} to space-dependent noise.

The Korteweg-de Vries equation is a well-known nonlinear dispersive PDE originally introduced as a model for shallow-water waves \cite{boussinesq, korteweg}. It famously possesses a family of right-traveling wave solutions
$u(t,x)=\phi_c(x-ct)$ with velocity dependent wave-profiles explicitly given by  
\begin{align}\label{eqn:solitons}
    \phi_c(x)=\frac{3c}{2}\sech^2(\sqrt{c}x/2), \quad c>0.
\end{align}
Note that the velocity dependence of the wave profiles satisfies the simple scaling $\phi_c(x)=c\phi_1(\sqrt{c}x)$. The traveling waves \eqref{eqn:solitons}, or solitons, arise due to a balance of dispersion and nonlinear effects. They are of key importance for the KdV dynamics: inverse scattering theory provides exact solutions to the KdV equation in terms of right-traveling solitons and a remaining dispersive component \cite{scattering}.

Many variations on \eqref{eqn:skdv} have been introduced in the literature to incorporate various (random) forcing mechanisms. In particular, we mention \cite{herman}, which first considered the stochastic KdV equation. We refer to our previous works \cite{westdorp1, westdorp2} and references therein for more background on stochastic forcing and deterministic forcing, respectively, in the context of the KdV equation. 

\paragraph{Deterministic stability} The stability of solitons \eqref{eqn:solitons} with respect to an initial perturbation has been extensively studied, for instance, by energy methods \cite{merle}. These rely on the conservative nature of the KdV dynamics: an infinite number of `integrals of motion', such as the $L^2$-norm, is conserved under the KdV flow \cite{miura}. The family \eqref{eqn:solitons} is, as a consequence of dispersion, only marginally linearly stable in $L^2$. Indeed, the linearized dynamics around the soliton $\phi_c$ are detailed by the operator
\begin{align}
    \mathcal{L}_c=-\partial_x^3+c\partial_x-2\partial_x(\phi_c\cdot),\label{eqn:Lc}
\end{align}
which, viewed as an operator on $L^2$, has spectrum $i\R$. In this work we rely heavily on the stability theory for the exponentially weighted spaces
\begin{align}
L^2_w(\R)=&\big\{g: e^{wx}g\in L^2(\R)\big\}\quad \text{with} \quad \|g\|_{L^2_w}=\|e^{wx}g\|_{L^2},\label{eqn:weighted}\\
H^1_w(\R)=&\big\{g: e^{wx}g\in H^1(\R)\big\}\quad \text{with} \quad \|g\|_{H^1_w}=\|e^{wx}g\|_{H^1}\label{eqn:weightedh1} \end{align}
that was developed by Pego and Weinstein in the classic work \cite{pegoweinstein}.
The exponential weight $e^{wx}$ ($w>0$) shifts the continuous spectrum of the operator $\mathcal{L}_c$ with $\sqrt{c}>w$ into the stable half-plane, leaving a spectral gap of size $w(c-w^2)$ and a double eigenvalue at $0$. Physically, the exponential weight dampens persisting disturbances outrun by the soliton. The linearized dynamics contain two neutral modes, associated with infinitesimal changes of $\phi_c$ with respect to the space variable $x$ and the amplitude parameter $c$. These spectral properties ensure that the linear flow $\{e^{\mathcal{L}_ct}\}_{t\geq 0}$ on $L^2_w$ generated by the operator \eqref{eqn:Lc} is exponentially stable on a subspace where the neutral eigenvalue is avoided. This subspace of $L^2_w$ consists of functions $v$ that satisfy the orthogonality conditions
\begin{align}\langle v,\phi_c\rangle=\langle v,\zeta_c\rangle=0,\label{eqn:introortho}\end{align}
where
\[\zeta_c(x)=\int_{-\infty}^x\partial_c\phi_c(y)\d y.\]
Based on these properties, Pego and Weinstein \cite{pegoweinstein} established the orbital stability of the traveling wave family \eqref{eqn:solitons} with respect to a small initial perturbation in $H^1\cap H^1_w$. In fact, such an initial perturbation only causes a small asymptotic phase-shift and amplitude change in the soliton.

\paragraph{Stochastic stability} Stability of the wave family \eqref{eqn:solitons} with respect to small stochastic forcing has been previously considered in \cite{debouardsoliton}. Based on energy methods, de Bouard and Debussche show that solutions to \eqref{eqn:skdv} with $\epsilon=0$ and small $\sigma>0$ stay close to the wave family for times small with respect to $\sigma^{-2}$. Their method relies on the fact that the soliton amplitude remains close to its starting value on this time scale. This is a significant restriction, as the result does not cover any (significant) modulation of the soliton. These stochastic modulations have, however, been explored to some degree in a formal sense. For example, %stochastic amplitude dynamics induced by noise ($\sigma>0$ in \eqref{eqn:skdv}) have previously been explored in a formal sense. 
the work \cite{cartwright} 
uses a collective coordinate approach to treat
the case where $W_t^Q$ is replaced by a scalar Brownian motion. The numerical and analytical results in \cite{westdorp1} cover both scalar and space-dependent noise
and provide explicit Taylor expansions for the behavior of the modulation parameters.

%treats the case of scalar noise -- i.e. \eqref{eqn:skdv} with $\sigma>0$ and $W_t^Q$ replaced by a scalar Brownian motion -- . A more general setting, covering both scalar and space-dependent noise, is explored in detail in \cite{westdorp1}, both numerically and analytically.
%Indeed, the statistics of a process driven by $\sigma W_t^Q$ grow as a function of $\sigma^2 t$. 

\paragraph{Present work}
In the present work, we consider deterministic and stochastic forcing of the soliton family \eqref{eqn:solitons}, leading to stochastic modulations of the amplitude and position over time.  We establish the orbital stability of \eqref{eqn:solitons} in a setting where the stochastic amplitude modulations can have arbitrary size. This is primarily due to the deterministic forcing mechanism $(\epsilon>0)$ present in \eqref{eqn:skdv}, which increases/decreases the energy present in the system after time $t$ by a factor $e^{\epsilon \int_0^t f(s)\d s}$ \cite{westdorp2}. In this setting, it is vital to explicitly account for the dynamics of the soliton amplitude. In particular, we will show that, to leading order in the small parameters $\epsilon$ and $\sigma$, this amplitude evolves according to the SDE\footnote{The process $c_{\mathrm{ap}}$ introduced in \eqref{eqn:cap} approximates the soliton amplitude only in distribution. As will be made clear later on, the soliton amplitude is driven by a translated version of $W_t^Q$.}
\begin{align}
    \d c_{\mathrm{ap}} =& \big[ \tfrac{4}{3}c_{\mathrm{ap}}\epsilon f(t)+\sigma^2 g_Q(c_{\mathrm{ap}}) \big] \ \d t+\tfrac{2}{9}c_{\mathrm{ap}}^{-1/2}\sigma \langle \phi_{c_{\mathrm{ap}}}^2, \d W_t^Q \rangle,\label{eqn:cap} 
    % \d \xi_{\mathrm{ap}} =& \big[c(t)+\tfrac{2}{3}c_{\mathrm{ap}}^{-1/2}\epsilon f(t)+\sigma^2 \Omega_d(c_{\mathrm{ap}},0)\big] \ \d t+\tfrac{2}{9}\sigma \langle -c_{\mathrm{ap}}^{-2}\phi_{c_{\mathrm{ap}}}^2+c_{\mathrm{ap}}^{-1/2}\phi_{c_{\mathrm{ap}}}\zeta_{c_{\mathrm{ap}}},  \d W_t^Q \rangle.\label{eqn:xiap}
\end{align}
in which the function $g_Q:\R\to\R$ %above 
will be made explicit in \Cref{sec:global}.  This allows us to demonstrate that stochastic stability is not limited to trivial changes in amplitude.

\paragraph{Main result}
We study solutions $u$ to \eqref{eqn:skdv} in the (relatively standard) setting \hyperlink{hyp:initial}{S1} -- see \Cref{sec:setup} ahead -- through the modulation Ansatz
\[u\big(t,x+\xi(t)\big)=\phi_{c(t)}(x)+v(t,x).\]
Here, $c(t)$ and $\xi(t)$ are stochastic processes that track the soliton amplitude and position, respectively, which we fully define later on. In particular, we supply \eqref{eqn:skdv} with the initial condition
\[u(0,x)=\phi_{c_*}(x),\]
for some $c_*>0$. The remainder $v(t,x)$ constitutes the error resulting from our modulation approach. The main contribution of this work concerns
the probabilistic behavior of the exit time
\[t_{\mathrm{st}}(\eta)=\sup\big\{t\geq 0:\|v(t)\|_{H^1_w} \leq\eta\big\},\]
which signals a deviation from the modulated soliton description.

The parameter $E$ in our result below provides a-priori (but unlimited) control over the total potential impact of the deterministic forcing. This is related to the fact that
the linear stability properties of the soliton family $\phi_c$ are limited by the soliton amplitude $c$. Indeed, the $w(c-w^2)$-wide spectral gap of the linear operator $\mathcal{L}_c$ on $L^2_w$ is at most $\frac{2}{3\sqrt{3}}c^{3/2}$ (take $w=\sqrt{c/3}$). The constants 
appearing in many of our estimates therefore require $c$ to be kept within a fixed range $[c_{\mathrm{min}}, c_{\mathrm{max}}]$ that is controlled by $E$.
% \[t_c=\inf\{t\geq 0:c(t) \notin [c_{\mathrm{min}},c_{\mathrm{max}}]\},\]

\begin{theorem}[See \Cref{sec:proofmain}]\label{thm:bigthm}
Pick $c_*,E>0$ and $w\in (0,\sqrt{c_*}/3]$. Assuming \hyperlink{hyp:initial}{S1}, there exist constants $\eta_0,C,\delta>0$ such that the following holds true. For all $\eta\in[0,\eta_0]$, $C\sigma,C\epsilon\in [0,\eta]$, each $T\geq 1$,
and each continuous function $f:\R^+\to \R$ for which
\begin{align}
\sup_{t\geq 0}|f(t)|\leq 1, \quad \text{and}\quad\epsilon \int_0^\infty |f(t)|\d t\leq E,\label{eqn:fassumption}\end{align}
% there exist modulation functions $c,\xi \in C^1(\R^+;\R)$ associated to the solution $u$ of \eqref{eqn:skdv} with \eqref{eqn:initial} such that:
the exit time $t_{\mathrm{st}}$ satisfies
    \begin{align}\mathbb{P}\big[t_{\mathrm{st}}(\eta)<T\big]\leq CT\sigma^2\log(1/\sigma)+ CTe^{-\delta\eta^2/\sigma^2}.\label{eqn:finalresult}\end{align}
\end{theorem}
% For $a\in[1,2)$, one may take $\epsilon=\sigma^{a}$ provided that $\sigma\leq (\eta_0/\delta)^{\frac{1}{1+a/2}}$. We note a counterintuitive aspect of \eqref{eqn:finalresult}: we lose control of the probability that instability occurs as the noise strength $\sigma$ tends to zero! This is due to the fact that we assert stability on an increasing time scale, proportional to $\epsilon^{-1}=\sigma^{-a}$. Even so, instability can be made arbitrarily improbable: given $\lambda \in (0,1]$, pick  $r$ large enough to ensure $r^{\frac{2a}{2-a}}e^{-r^2}\leq(\lambda/2)^{1-\frac{a}{2-a}}$ and $\sigma=(\lambda/2r^2)^{\frac{1}{2-a}}$ which gives
% \[\mathbb{P}[t_{\mathrm{st}}(\eta)<T]\leq \sigma^{2-a} r^2+\frac{e^{-r^2}}{\sigma^a}\leq \lambda.\]  
The probability bound \eqref{eqn:finalresult} contains two contributions. The exponential part $Te^{-\delta\eta^2/\sigma^2}$ stems from the linear stability properties of the wave family \eqref{eqn:solitons} on weighted spaces, and matches modern phase-tracking results valid on time scales that are exponentially long with respect to the ratio $\eta^2/ \sigma^{2}$ \cite{mark, marklocal, christianphase, maclaurin, joris}. The contribution $T\sigma^2\log(1/\sigma)$, however, limits the time-scale on which stability can be obtained to times $T$ that are \textit{small} with respect to $\sigma^{-2}\log(1/\sigma)^{-1}$, but independent of the size-parameter $\eta$.

The limiting factor is that we rely on a-priori control of the remainder $v$ in the unweighted space $L^2$. In this space, the solitons $\phi_c$ are not linearly stable, and the norm $\|v(t)\|^2_{L^2}$ grows linearly in time. The factor $\sigma^2\log(1/\sigma)$ stems from an It\^o drift term proportional to $\sigma^2$, combined with a factor $\log(1/\sigma)$ that we need to account for the potentially large fluctuations of the soliton amplitude.

%Further below, we discuss the possibility of reducing the contribution $T\sigma^{-2}\log(1/\sigma)^{-1}$ by resorting to a-priori control in different spaces.

%Aside from this technicality, the time-scale $\sigma^{-2}$ imposes more profound limitations in the KdV context. Indeed, the linear stability properties of the soliton family $\phi_c$ are limited by the soliton amplitude $c$. The $w(c-w^2)$-wide spectral gap of the linear operator $\mathcal{L}_c$ on $L^2_w$ is at most $\frac{2}{3\sqrt{3}}c^{3/2}$ (take $w=\sqrt{c/3}$). As we shall see, the effective soliton amplitude $c$ is driven by the noise $\sigma W_t^Q$, and can only be kept within limits $[c_{\mathrm{min}}, c_{\mathrm{max}}]$ on time scales proportional to $\sigma^{-2}$. 

We emphasize that within the timescale discussed above,
%Despite the limitations outlined above, 
the deterministic forcing $(\epsilon>0)$ can cause significant modulation of the soliton amplitude $c(t)$. Thus,
the main point of our results is that they demonstrate that the stability of the wave family \eqref{eqn:solitons} is not limited to small fluctuations. To showcase this, we obtain a result on the exit time
\[t_{\mathrm{ap}}(\lambda)=\sup\big\{t\geq 0:|c(t)-c_{\mathrm{ap}}(t)|\leq\lambda\big\},\]
which describes the validity of the approximation \eqref{eqn:cap}.

\begin{theorem}[See \Cref{sec:validity}]\label{thm:validity}
   Assuming the setting of \Cref{thm:bigthm}, for each $\lambda> 0$ that satisfies \[ \lambda\leq \min\{\tfrac{1}{2}c_*e^{-3E},c_*e^{3E}\},\] the exit times $t_{\mathrm{ap}}$ and $t_{\mathrm{st}}$ satisfy%\todo{[hjh: als $\lambda$ klein is kunnen we dit ook 1 krijgen door $C$ groot te kiezen; kun je de $\sigma^2 \le \lambda$ beperking hiermee vermijden? Kun je ook de $\lambda$ uit de log halen en alleen maar een $\sigma$ in de log hebben staan?]}
    \begin{align}\label{eqn:approximation}
    \mathbb{P}\big[t_{\mathrm{ap}}(\lambda)<T\big]\leq C\frac{\sigma^2}{\lambda}\log(1/\sigma).
\end{align}
\end{theorem}
% \begin{remark}
%     {\color{red}What changes in the bound on $|c(t)-c_{\mathrm{ap}}(t)|$ if we replace $c_{\mathrm{ap}}(t)$ by $e^{\frac{4}{3}\epsilon\int_0^tf(s)\d s}?$ In other words, to what extend is the stochastic modulation necessary? See \Cref{lem:logc}}
% \end{remark}
\paragraph{Approach} The field of stochastic traveling waves has witnessed rapid development in recent years. Several approaches have emerged to study stochastic traveling waves in various PDE settings. These typically feature a decomposition of the solution in terms of a traveling wave modulated by a stochastic phase-shift. The exact definition of this phase-shift is where the various methods differ. Let us mention the phase-lag method \cite{kruger, manuel}, variational-phase \cite{mark, marklocal, christianphase, maclaurin} and isochronal-phase \cite{zachary}, all applied in the context of reaction-diffusion equations. In dispersive settings, (adaptations of) the variational phase have been applied in \cite{debouardsoliton, debouardadditive, debouardnls, westdorp2, joris}. Let us also mention the recent work \cite{jorisbanach}, which presents a phase-tracking mechanism for general symmetry groups.

In the present context, where we study the wave family \eqref{eqn:solitons}, we follow our approach in \cite{westdorp2} and decompose the solution $u$ to \eqref{eqn:skdv} as
\[u\big(t,x+\xi(t)\big)=\phi_{c(t)}(x)+v(t,x).\]
Here, $c(t)$ and $\xi(t)$ are processes that track the soliton amplitude and position, and ensure the orthogonality conditions
\begin{align} \big\langle v(t,\cdot),\phi_{c(t)}\big\rangle=\big\langle v(t,\cdot),\zeta_{c(t)}\big\rangle=0, \quad t\geq 0.\label{eqn:variational}\end{align}
We point out a technical but essential difference with \cite{westdorp2}, where we decomposed $u$ in terms of a fixed soliton $\phi_{c_*}$ in a rescaled frame. We do not employ this coordinate-transformation here, to avoid It\^{o} correction terms in the evolution equation of the remainder $v$. This leaves the challenge of harnessing the linear stability properties of a time-varying soliton $\phi_{c(t)}$. In spirit of \cite{maclaurin}, we employ the linear stability properties at some time $T>0$ of $\mathcal{L}_{c(T)}$ on an interval $[T,T+\Delta T]$. On such intermediate intervals, the soliton amplitude $c(t)$ does not deviate much from $c(T)$. We furthermore introduce local modulation parameters on the interval $[T,T+\Delta T]$ resulting in a local remainder $v^T$ which enjoys exponential damping by the linear flow $\{e^{\mathcal{L}_{c(T)}t}\}_{t\geq0}$. For $t\in[T,T+\Delta T]$, the local remainder can be used to control the global remainder $v(t)$.

\paragraph{Challenges}
 Classically, analysis of the $L^2_w$-norm of $v(t)$ requires complementary control on the unweighted $L^2$-norm of $v(t)$. This is due to the nonlinearity $u\partial_x u$ present in \eqref{eqn:skdv}, which is typically estimated in the weighted $L^2$-norm as
\begin{align}\int_{\R} e^{2wx}u^2(x)u^2_x(x)\d x\leq \|u\|^2_{\infty}\int_{\R} e^{2wx}u^2_x(x)\d x\leq 2\|u\|^2_{H^1}\|u_x\|^2_{L^2_w}.\label{eqn:infty}\end{align}
The unweighted $H^1$-norm of $v(t)$ can in turn be analyzed via energy methods, see \cite[Section 6]{westdorp2}. A formal application of the It\^o lemma shows that, in the presence of stochastic forcing $(\sigma>0)$, the energy of solutions to \eqref{eqn:skdv} evolves as
\[\d \|u\|_{L^2}^2=\sigma^2 \|u\|^2_{L^2} \ \d t+2\epsilon f(t)\|u\|_{L^2}^2\ \d t+2 \sigma \langle u,u \ \d W_t^Q\rangle.\] 
As we will see later on, due to an It\^o drift correction, the energy in $v(t)$ grows proportionally to $\sigma^2 t$ at leading order. We incur a further $\log( 1/\sigma)$ penalty as a consequence of the $\mathcal{O}(v^2)$ terms,
leading to the first term in \eqref{eqn:finalresult}.

At present, we can hence not carry out our arguments on a time-interval $[0,T_{\mathrm{max}}\sigma^{-2}]$ for \textit{any} $T_{\mathrm{max}}>0$, which we believe should be attainable with more refined arguments in the special case $\epsilon =0 $. Indeed, inspecting the approximation
\eqref{eqn:cap} with $\epsilon=0$,
we see that  the effective soliton amplitude $c$ is driven by the noise $\sigma W_t^Q$. We can hence expect $c$ to be kept within the range $[c_{\mathrm{min}}, c_{\mathrm{max}}]$ on time scales proportional to $\sigma^{-2}$ and the (arbitrary) size of this range. 

%This presents a difficulty for obtaining stability on a time-scale where the \textit{noise} causes significant modulation: the unweighted $L^2$-norm of $v(t)$ can only be kept small for times small with respect to $\sigma^{-2}$. Hence, we can not carry out our arguments on a time-interval $[0,T_{\mathrm{max}}\sigma^{-2}]$ for \textit{any} $T_{\mathrm{max}}>0$. 

Furthermore, the bound \eqref{eqn:infty} shows that it is in fact sufficient to control the \textit{supremum} norm of $v(t)$, instead of an energy norm. Preliminary numerical investigations suggest that this norm remains under control for timescales longer than $\sigma^{-2}$, showing that the dynamics of $c$ are indeed dominant. However, bounds of this nature will require us to depart from energy-based methods, which are only available in $L^2$-based spaces. We envision that future work in this direction will center on a careful (and challenging) direct pointwise analysis of the dispersive dynamics that drive the remainder $v(t)$ in the wake of the soliton. We emphasize however that the framework developed here to control the weighted norm of $v$ can readily accommodate such a refinement.

\paragraph{Outline}
This paper is organized as follows. In \Cref{sec:setup}, we outline the setting of \eqref{eqn:skdv} in more detail and recall classical linear stability results regarding the operator $\mathcal{L}_c$ defined in \eqref{eqn:Lc}. Then, in \Cref{sec:global}, we analyze the modulation system brought forth by the conditions \eqref{eqn:variational}. Following this, in \Cref{sec:local}, we introduce local approximations to this system. In \Cref{sec:weighted}, we carry out time-discretized stability arguments for the local remainder $v^T$ in the weighted spaces $H^1_w$. We supplement this with an analysis of the local modulation parameters in \Cref{sec:parameters}. In \Cref{sec:energy}, we develop global control of the unweighted $L^2$-norm of the remainder $v$, as well as the global amplitude parameter $c(t)$. Finally, we combine our results in \Cref{sec:proofmain} and \Cref{sec:validity} for the proofs of \Cref{thm:bigthm} and \Cref{thm:validity}.

\paragraph{Acknowledgments} Both authors acknowledge support from the Dutch Research Council (NWO) (grant 613.009.137).

\section{Preliminaries}\label{sec:setup}
Below, we describe the setting of \eqref{eqn:skdv} in more detail and collect several results regarding the linear stability of the soliton family \eqref{eqn:solitons} as developed in \cite{pegoweinstein}.

\paragraph{Stochastic set-up}
We work in the following setting, which we recall from \cite[Section 2.1.2]{westdorp1}.
\begin{itemize}
\item[\textbf{S1}]{
  \phantomsection\hypertarget{hyp:initial}{}\textit{
  We have $c_*>0$, and supply \eqref{eqn:skdv} with the initial condition
\begin{align}u(0,x)=\phi_{c_*}(x).\label{eqn:initial}\end{align}
The operator $Q\in\mathcal{L}(L^2)$ is defined as the convolution with an even function $q\in H^1\cap L^1$ through
\[Qf=\int_{\R}q(x-y)f(y)\d x.\]
Moreover, the Fourier transform $\hat{q}$ of the kernel $q$ is non-negative. Lastly, the forcing term $f:\R^+\to \R$ is continuous and bounded as
\[|f(t)|\leq 1, \quad t\geq 0.\]}} 
\end{itemize}

By Young's convolution inequality, the integrability assumption $q\in L^1$ ensures that $Q$ is a bounded operator on $L^2(\R)$. The non-negativity of the Fourier transform $\hat{q}$ furthermore ensures that $Q$ is a symmetric and non-negative operator. Thus, $Q$ can be used to construct an $L^2$-valued cylindrical Brownian motion $W_t^Q$ on a filtered probability space $(\Omega,\mathcal{F},\mathbb{F},\mathbb{P})$ cf. \cite{daprato,liurockner}, which satisfies the formal covariance identity
\begin{align*}
    \mathbb{E}\bigl[\d W^Q(x,t)\d W^Q(y,s)\bigr]=\delta(t-s)q(x-y),\quad x,y\in \R, \quad s,t>0.
\end{align*}
Since $q$ is an even function, the spatial correlation of $W^Q_t$ depends only on the distance $|x-y|$ between two points $x,y\in \R$ and preserves the translation-invariance of \eqref{eqn:skdv}. We recall furthermore from \cite[Section 2.1.2]{westdorp1} that the non-negative operator $Q$ has a square-root $Q^{1/2}$ that acts as the convolution
\begin{align*}
    Q^{1/2}f(x)=\int_{\R}q_{1/2}(x-y)f(y)\ \d y,
\end{align*}
where $q_{1/2}$ is the inverse Fourier transform of $\sqrt{\hat{q}}$. 

Let us now introduce some notation related to stochastic integration with respect to the noise $W_t^Q$. Letting $\{e_k\}_{k=0}^{\infty}$ be an orthonormal basis of $L^2(\R)$, and introducing the space
$L^2_Q:=Q^{1/2}(L^2)$ equipped with the inner product
\[\langle v, w\rangle_{L^2_Q}=\langle Q^{-1/2}v,Q^{-1/2}w\rangle_{L^2},\]
we note that $L^2_Q$ is a separable Hilbert space for which $\{Q^{1/2}e_k\}_{k=0}^\infty$ is an orthonormal basis. We furthermore denote by $\mathrm{HS}(L^2_Q,\mathcal{H})$ the space of Hilbert-Schmidt operators between $L^2_Q$ and a Hilbert space $\mathcal{H}$, equipped with the inner-product 
\[\langle A,B\rangle_{\mathrm{HS}(L^2_Q,\mathcal{H})}=\sum_{k=0}^\infty \bigl\langle A[Q^{1/2}e_k],B[Q^{1/2}e_k]\bigr\rangle_{\mathcal{H}}.\]

In the setting \hyperlink{hyp:initial}{S1},  de Bouard and Debussche \cite{debouardwellposedness} showed that  \eqref{eqn:skdv} admits unique mild solutions for purely stochastic forcing $(\epsilon=0)$. For deterministic forcing $(\sigma=0)$, the well-posedness of \eqref{eqn:skdv} has been established in \cite[lemma 2.1]{westdorp2}. Combining these results yields the following well-posedness property, in which $\{e^{-\partial_x^3t}\}_{t\in\R}$ denotes the $C_0$-group of isometries on $L^2$ associated to the Airy equation.
\begin{lemma}\label{lem:wellposed}[\cite{debouardadditive,westdorp2}]
    Assuming \hyperlink{hyp:initial}{S1}, \eqref{eqn:skdv} has a unique mild solution $u\in L^2(\Omega;C(\R^+;L^2(\R))$, that satisfies
\begin{align*}
    u(t)=&e^{-\partial_x^3 t}u_0-\int_0^t e^{-\partial_x^3 (t-t')}\partial_x\big(u^2(t')\big)\d t'+\epsilon\int_0^t f(t')e^{-\partial_x^3 (t-t')}u(t')\d t'+\sigma\int_0^t e^{-\partial_x^3 (t-t')}u(t')\d W_{t'}^Q,
\end{align*}
and has paths that are almost surely in $C(\R^+;H^1(\R))$.
\end{lemma}

\paragraph{Linear stability tools}

Our arguments rely heavily on the linear stability properties of the operator $\mathcal{L}_c$ defined in \eqref{eqn:Lc}, that were developed in \cite{pegoweinstein}. We recall from this work that $\mathcal{L}_c$ satisfies the eigenvalue relations $\mathcal{L}_{c} \partial_x\phi_{c}=0$ and $\mathcal{L}_{c} \partial_c \phi_{c}=\partial_x\phi_{c}$, giving rise to a two-dimensional generalized kernel. The (formal) adjoint $\mathcal{L}^*_c$ also has a two-dimensional kernel, spanned by $\phi_c$ and the primitive
\[\zeta_c (x)=\int_{-\infty}^x \partial_c \phi_c(y)\d y.\]
This function satisfies $\zeta_c \in L^2_{-w}$ (but not $\zeta_c\in L^2$). With this notation in place, we note that the projection onto the generalized kernel of $\mathcal{L}_c$ is given by
\begin{align}P_{c} f=\langle f, \tfrac{2}{9}c^{-1/2}\zeta_c+\tfrac{2}{9}c^{-2}\phi_c  \rangle \partial_x \phi_{c} +\langle f, \tfrac{2}{9}c^{-1/2}\phi_c\rangle \partial_c \phi_c.\label{eqn:spectral}\end{align}
% where 
% \begin{align}\eta_c^1=\tfrac{2}{9}c^{-1/2}\zeta_c+\tfrac{2}{9}c^{-2}\phi_c \quad \text{and} \quad \eta_c^2=\tfrac{2}{9}c^{-1/2}\phi_c, \label{eqn:linearcomb}\end{align}
We write  $Q_c=I-P_c$ for the complementary projection, and note that $\operatorname{Ker}(P_c)$ coincides with the subspace where the conditions \eqref{eqn:introortho} hold. We collect the following properties of the flow generated by $\mathcal{L}_c$, which demonstrate the parabolic nature of $-\partial_x^3$ on the weighted spaces $L^2_w$, after projecting out the neutral modes. First, let us fix limits $c_{\mathrm{min}}, c_{\mathrm{max}}$ and a weight $w$ as follows.
\begin{itemize}
\item[\textbf{S2}]{
  \phantomsection\hypertarget{hyp:weight}{}\textit{
  The constants $c_{\mathrm{min}}, c_{\mathrm{max}},w\in \R$ satisfy \[0<c_{\mathrm{min}}<c_* < c_{\mathrm{max}}\quad \text{and} \quad 0<w<\sqrt{c_{\mathrm{min}}}/3.\]
  }}
\end{itemize}
%Based on these spectral properties, the following can be concluded about the flow generated by this operator.

    \begin{theorem}[\cite{pegoweinstein, mizumachi}]\label{thm:linearstability}
    Let $c\in[c_{\mathrm{min}}, c_{\mathrm{max}}]$. Then, $\mathcal{L}_{c}$ is the generator of a $C_0$-semigroup $\{e^{\mathcal{L}_ct}\}_{t\geq 0}$ on $H^s_w$ for any real $s$. For all $b>0$ which satisfy $b<w(c-w^2)$, there exists a constant $M>0$ such that, for all $g\in L_{w}^2$, $t>0$ and $k\in\{0,1\}$, we have
    \begin{align}\label{eqn:linear}
        \|\partial_x^k  e^{\mathcal{L}_{c}t}Q_c g\|_{L^2_{w}}\leq Mt^{-k/2}e^{-b t}\|g\|_{L^2_{w}},
    \end{align}
    while for all $g\in L_{w}^1$ we have
    \begin{align}\label{eqn:linearL1}
        \|\partial_x^k  e^{\mathcal{L}_{c}t}Q_c g\|_{L^2_{w}}\leq Mt^{-(2k+1)/4}e^{-b t}\|g\|_{L^1_{w}}.
    \end{align}
\end{theorem}

\section{Global modulation system}
\label{sec:global}
With these preliminaries in place, we are in shape to introduce the decomposition
\begin{align}u\big(t,x+\xi(t)\big)=\phi_{c(t)}(x)+v(t,x)\label{eqn:decomposition}\end{align}
characterized by the orthogonality conditions
\begin{align}\big\langle v(t,\cdot),\phi_{c(t)}\big\rangle =\big\langle v(t,\cdot),\zeta_{c(t)}\big\rangle =0.\label{eqn:ortho}\end{align}
As a consequence of \Cref{lem:wellposed}, the remainder $v$ introduced through \eqref{eqn:decomposition} has paths that are almost surely in $C(\R^+;H^1(\R))$. The decomposition \eqref{eqn:decomposition} is equivalent to that in \cite{westdorp1, westdorp2}, albeit phrased in a different frame of reference. The unique existence of decomposition \eqref{eqn:decomposition} is guaranteed as long as $\|v(t)\|_{L^2_w}$ remains sufficiently small, essentially as a result of the implicit function theorem. 

\begin{lemma}[\cite{mizumachi}]\label{lem:implicit}
    Assuming \hyperlink{hyp:weight}{S2}, there exist constants $\delta_1, C_1>0$ such for each $v_* \in H^1_w\cap H^1$ with $\|v_*\|_{L^2_w}\leq \delta_1$ and $c_*\in [c_{\mathrm{min}},c_{\mathrm{max}}]$, there exist unique parameters $c>0, \xi\in \R$ and a unique function $v\in H^1_w\cap H^1$ that together satisfy the identities
    \[\phi_{c_*}(x)+v_*(x)=\phi_{c}(x-\xi)+v(x-\xi) \quad \text{with} \quad \langle v,\phi_{c}\rangle =\langle v,\zeta_{c}\rangle=0\]
    and the bounds
    \begin{align*}\|v\|_{H^1_w}+|\xi|+|c_*-c|\leq &C_1\|v_*\|_{H^1_w},\\
    \|v\|_{H^1}\leq& C_1\big(\|v_*\|_{H^1}+\|v_*\|_{H^1_w}\big).
    \end{align*}
\end{lemma}
For convenience, we introduce a phase-shift function $\Omega(t)$ through $\xi(t)=\int_0^t c(t') \d t'+\Omega(t)$. Our goal in this section is to describe the evolution of the modulation parameters $c(t)$ and $\Omega(t)$ via SDEs of the form
% c^\sigma_s[h]=&\sigma \langle T_{\xi} h,c_s(v,c)\rangle \\
%     \Omega^\sigma_s[h]=& \sigma \langle T_{\xi} h,\Omega_s(v,c)\rangle \\
\begin{align}
    \d c =& c_d^{\sigma,\epsilon}(v,c,t) \ \d t+\sigma \big\langle c_s(v,c), T_{\xi} \d W_t^Q \big\rangle, \label{eqn:postulate1} \\
    \d \Omega =& \Omega_d^{\sigma,\epsilon}(v,c,t)\ \d t+ \sigma \big\langle \Omega_s(v,c), T_{\xi} \d W_t^Q \big\rangle.\label{eqn:postulate2}
\end{align}
For $\xi\in\R$, the operator $T_{\xi}$ above denotes the translation by $\xi$, i.e. $(T_{\xi} f)(x)=f(x+\xi)$. We thus set out to find mappings $c_d^{\sigma,\epsilon},  c_s, \Omega_d^{\sigma,\epsilon}$ and $\Omega_s$ that ensure the conditions \eqref{eqn:ortho}.
% through which $c(t)$ and $\Omega(t)$ are uniquely defined as long as $\|v(t)\|_{L^2_w}\leq \delta_6$.  To this end, we compute the It\^{o} form of the processes $\langle v(t),\phi_{c(t)}\rangle$ and $\langle v(t),\zeta_{c(t)}\rangle$. 
Formally applying the It\^o lemma \cite[Theorem 4.32]{daprato} to \eqref{eqn:decomposition} with \eqref{eqn:postulate1} and \eqref{eqn:postulate2}, yields, after tedious computations
\begin{align}\label{eqn:vmod}
    % \d v &=\big[\mathcal{L}_{c(t)}v+\epsilon f(t)X(v,c)+Y^\sigma(v,c)\big]\ \d t+ \sigma Z(v,c)\ T_{\xi}  \d W_t^Q\\
    \d v &=\big[\mathcal{L}_{c(t)}v+Y^{\sigma,\epsilon}(v,c,t)\big]\ \d t+ \sigma Z(v,c)\ T_{\xi}  \d W_t^Q
    \end{align}
where
   \begin{align*}
   % X(v,c)=&\phi_{c}+v+ \Omega_f(v,c)\partial_x(\phi_c+v)+c_f(v,c)\partial_c\phi_c\\
   % Y^\sigma(v,c)=&N(v) + \Omega_d^\sigma(v,c)\partial_x(\phi_c+v)+c_d^\sigma(v,c)\partial_c\phi_c+\sigma^2 Y(v,c)\\
Y^{\sigma,\epsilon}(v,c,t)=&N(v)+\epsilon f(t)(\phi_c+v)+ \Omega_d^{\sigma,\epsilon}(v,c,t)\partial_x(\phi_c+v)-c_d^{\sigma,\epsilon}(v,c,t)\partial_c\phi_c+\sigma^2 Y_d(v,c).
\end{align*}
Here, $N(v)=-\partial_x(v^2)$ is the KdV nonlinearity, the drift contribution $Y_d$ is given by
\begin{align*}
    Y_d(v,c)=& \tfrac{1}{2} \big\|Q^{1/2}\Omega_s(v,c)\big\|_{L^2}^2 \partial_x^2v+  \tfrac{1}{2} \big\|Q^{1/2}\Omega_s(v,c)\big\|_{L^2}^2 \partial_x^2\phi_c   + \tfrac{1}{2}\big\|Q^{1/2}c_s(v,c)\big\|_{L^2}^2 \partial_c^2\phi_{c}\end{align*}
and the stochastic component is defined by
\begin{align*}
Z(v,c)[h]=&\Big(h +\big\langle\Omega_s(v,c),h\big\rangle\partial_x \Big)v   +\Big(h+ \big\langle\Omega_s(v,c),h\big\rangle \partial_x  - \big\langle c_s(v,c),h \big\rangle \partial_c  \Big)\phi_{c}.
\end{align*}
The formal adjoint of $Z$ is given by
\begin{align*}
Z^*(v,c)[g]=&(g+\phi_c)v +\Omega_s(v,c)\big\langle g,\partial_x (\phi_c+v)\big\rangle-  c_s(v,c) \langle g, \partial_c\phi_{c}\rangle.
\end{align*}
The evolution of $\big\langle v(t),\phi_{c(t)}\big\rangle$ can now formally be obtained by applying the It\^{o} product rule. However, we can not expect \eqref{eqn:vmod} to hold in the strong sense, as $v$ is not regular enough for $\mathcal{L}_{c(t)}v$ and $\partial_x^2 v$ to be well-defined. We take care of this technical issue in \Cref{app:mild} by resorting to a mild It\^{o} formula. The result is as follows.

\begin{lemma}[See \Cref{app:mild}]\label{lem:innerprod}
Assume  \hyperlink{hyp:initial}{S1} and \hyperlink{hyp:weight}{S2}. For each $t\geq0$, the inequalities
\[\|v(t')\|_{L^2_w}\leq \delta_1 \quad \text{and} \quad c(t')\in[c_{\mathrm{min}},c_{\mathrm{max}}],\quad t' \in[0,t]\] 
imply that
\begin{align}
    \d \big\langle v(t),\phi_{c(t)}\big\rangle=&\Big( c_d^{\sigma,\epsilon}(v,c,t) \langle v,\partial_c \phi_{c(t)}\rangle +\big\langle Y^{\sigma,\epsilon}(v,c,t),\phi_{c(t)}\big\rangle\Big)\ \d t\nonumber\\
    &+\sigma^2\Big(\tfrac{1}{2}\big\|Q^{1/2}c_s(v,c)\big\|_{L^2}^2 \langle v,\partial_c^2 \phi_{c(t)}\rangle+  \big\langle Q^{1/2}Z^*(v,c)[\partial_c\phi_{c(t)}],Q^{1/2}c_s(v,c) \big\rangle\Big)\ \d t \nonumber\\
    &+\sigma \langle v,\partial_c\phi_{c(t)}\rangle \big\langle c_s(v,c),T_{\xi}\d W_t^Q\big\rangle +\sigma\big\langle Z(v,c)[T_{\xi}\d W_t^Q],\phi_{c(t)}\big\rangle  \label{eqn:vphi}
\end{align}
and 
\begin{align}
    \d \big\langle v(t),\zeta_{c(t)}\big\rangle=&\Big( c_d^{\sigma,\epsilon}(v,c,t) \langle v,\partial_c \zeta_{c(t)}\rangle+\big\langle Y^{\sigma,\epsilon}(v,c,t),\zeta_{c(t)}\big\rangle \Big)\ \d t\nonumber\\
    &+\sigma^2\Big(\tfrac{1}{2}\big\|Q^{1/2}c_s(v,c)\big\|_{L^2}^2 \langle v,\partial_c^2 \zeta_{c(t)}\rangle+  \big\langle Q^{1/2}Z^*(v,c)[\partial_c\zeta_{c(t)}],Q^{1/2}c_s(v,c) \big\rangle\Big)\ \d t \nonumber\\
    &+\sigma \langle v,\partial_c\zeta_{c(t)}\rangle \big\langle c_s(v,c),T_{\xi}\d W_t^Q\big\rangle +\sigma\big\langle Z(v,c)[T_{\xi}\d W_t^Q],\zeta_{c(t)}\big\rangle \label{eqn:vzeta}
\end{align}
hold $\mathbb{P}$-almost surely.
\end{lemma}
\begin{remark}
In \eqref{eqn:vphi} and \eqref{eqn:vzeta} above, derivatives on $v$ are interpreted in the weak sense. 
\end{remark}

Solving for solutions $c_s(v,c)$ and $ \Omega_s(v,c)$ to the system
\begin{align*}
    \begin{bmatrix}
        \langle v,\partial_c\phi_c\rangle \big\langle c_s(v,c),h \big\rangle  +\big\langle Z(v,c)[h],\phi_c\big\rangle\\
          \langle v,\partial_c\zeta_c\rangle \big\langle c_s(v,c),h \big\rangle  +\big\langle Z(v,c)[h],\zeta_c\big\rangle
    \end{bmatrix}=\begin{bmatrix}
        0\\0
    \end{bmatrix}, \quad \forall h\in L^2_Q,
\end{align*}
reveals that
\begin{align*}
K_c(v)\begin{bmatrix}
        \langle c_s(v,c),h\rangle\\
        \langle \Omega_s(v,c),h\rangle 
    \end{bmatrix}
   % \begin{bmatrix}
   %      \langle v,\partial_c\phi_c\rangle \langle c_s(v,c),h \rangle   +\langle\Omega_s(v,c),h\rangle \langle\partial_x v ,\phi_c\rangle   -\langle c_s(v,c),h\rangle\langle  \partial_c  \phi_{c},\phi_c\rangle\\
   %        \langle v,\partial_c\zeta_c\rangle \langle c_s(v,c),h \rangle  + \langle\Omega_s(v,c),h\rangle \langle\partial_x (\phi_c+v),\zeta_c\rangle  - \langle c_s(v,c),h\rangle \langle \partial_c  \phi_{c},\zeta_c\rangle
   %  \end{bmatrix}
    =- \begin{bmatrix}
        \big\langle h(\phi_c+v),\phi_c\big\rangle\\\big\langle h(\phi_c+v),\zeta_c\big\rangle
    \end{bmatrix}
    \end{align*}
  for all $h\in L^2_Q$, where $K_c(v)$ is the matrix
  \begin{align*}
      K_c(v)=\begin{bmatrix}
        \langle -\phi_c+v,\partial_c\phi_c\rangle  & \langle \partial_x v,\phi_c\rangle\\
        \langle v,\partial_c\zeta_c\rangle-\langle \partial_c  \phi_{c},\zeta_c\rangle & \big\langle\partial_x (\phi_c+v),\zeta_c\big\rangle 
    \end{bmatrix}.
  \end{align*}
This implies that the mappings  $c_s(v,c)$ and $ \Omega_s(v,c)$ are given by
    \begin{align*}
    \begin{bmatrix}
        c_s(v,c)\\
        \Omega_s(v,c)
    \end{bmatrix}
   % \begin{bmatrix}
   %      \langle v,\partial_c\phi_c\rangle  c_s(v,c)+\Omega_s(v,c) \langle\partial_x v ,\phi_c\rangle   - c_s(v,c)\langle  \partial_c  \phi_{c},\phi_c\rangle\\
   %        \langle v,\partial_c\zeta_c\rangle  c_s(v,c) + \Omega_s(v,c) \langle\partial_x (\phi_c+v),\zeta_c\rangle  -  c_s(v,c) \langle \partial_c  \phi_{c},\zeta_c\rangle
   %  \end{bmatrix}
    =-K_c^{-1}(v)\begin{bmatrix}
        (\phi_c+v)\phi_c \\(\phi_c+v)\zeta_c
    \end{bmatrix}.
    \end{align*}
We solve for $c_d^{\sigma,\epsilon}$ and $\Omega_d^{\sigma,\epsilon}$ by decomposing 
\begin{align*}
    c_d^{\sigma,\epsilon}(v,c,t)=&c^0_d(v,c)+\epsilon f(t)c_f(v,c)+\sigma^2 c_d(v,c), \\
    \Omega_d^{\sigma,\epsilon}(v,c,t)=&\Omega^0_d(v,c)+\epsilon f(t)\Omega_f(v,c)+\sigma^2 \Omega_d(v,c),
\end{align*}
and isolating the $\sigma^2$ and $\epsilon$ dependent terms in \eqref{eqn:vphi} and \eqref{eqn:vzeta}. This yields
\begin{align*}
 \begin{bmatrix}
        c_d^0(v,c)\\
        \Omega_d^0(v,c)
    \end{bmatrix}=&-K_c^{-1}(v)\begin{bmatrix}
        \big\langle N(v),\phi_c\big\rangle\\
        \big\langle N(v),\zeta_c \big\rangle
    \end{bmatrix},\\
    \begin{bmatrix}
        c_f(v,c)\\
        \Omega_f(v,c)
    \end{bmatrix}=&-K_c^{-1}(v)\begin{bmatrix}
        \langle \phi_c+v,\phi_c\rangle\\
        \langle \phi_c+v,\zeta_c \rangle
    \end{bmatrix},
    \end{align*}
together with
    \begin{align*}
    \begin{bmatrix}
        c_d(v,c)\\
        \Omega_d(v,c)
    \end{bmatrix}=&-K_c^{-1}(v)\begin{bmatrix}
        \big\langle Y_d(v,c),\phi_c\big\rangle\\
        \big\langle Y_d(v,c),\zeta_c \big\rangle
    \end{bmatrix}-\tfrac{1}{2}\big\|Q^{1/2}c_s(v,c)\big\|_{L^2}^2K_c^{-1}(v)\begin{bmatrix}
         \langle v,\partial_c^2 \phi_c\rangle\\
         \langle v,\partial_c^2 \zeta_c\rangle
    \end{bmatrix}\\
    &-K_c^{-1}(v)\begin{bmatrix}
          \big\langle Q^{1/2}Z^*(v,c)[\partial_c\phi_c],Q^{1/2}c_s(v,c) \big\rangle\\
           \big\langle Q^{1/2}Z^*(v,c)[\partial_c\zeta_c],Q^{1/2}c_s(v,c) \big\rangle
    \end{bmatrix}.
\end{align*}

We remark that, in the absence of deterministic forcing ($\epsilon=0$), the modulation system derived above is equivalent to the system found in \cite[Section 2.3.2]{westdorp1}, and in the absence of stochastic forcing ($\sigma=0$) to that in \cite[Section 2]{westdorp2}. Evaluating the modulation system at $v=0$
gives rise to the approximation $c_{\mathrm{ap}}$ defined in \eqref{eqn:cap}, with $g_Q(c)=c_d(0,c)$. Now that the modulation system has taken concrete form, we can establish control on the modulation parameters, which will be important for our stability arguments later on.
\begin{lemma}\label{lem:modcontrol}
 Assuming \hyperlink{hyp:weight}{S2}, there exist constants $\delta_2, C_2>0$ such that for all $\tilde{v}\in L^2_w$ that satisfy $\|\tilde{v}\|_{L^2_w}\leq \delta_2$ and each $ \tilde{c}\in [c_{\mathrm{min}},c_{\mathrm{max}}]$
we have
\begin{align}
   \big\|Q^{1/2}c_s(\tilde{v},\tilde{c})-Q^{1/2}c_s(0,\tilde{c})\big\|_{L^2}+\big\|Q^{1/2}\Omega_s(\tilde{v},\tilde{c})-Q^{1/2}\Omega_s(0,\tilde{c})\big\|_{L^2}\leq & C_2\|\tilde{v}\|_{L^2_w},\label{eqn:cscontrol}\\
   \big|c_d^0(\tilde{v},\tilde{c})\big|+\big|\Omega_d^0(\tilde{v},\tilde{c})\big|\leq & C_2 \|\tilde{v}\|_{L^2_w}^2,\label{eqn:cd0control}\\
\big|c_d(\tilde{v},\tilde{c})-\tilde{c}_d(0,\tilde{c})\big|+\big|\Omega_d(\tilde{v},\tilde{c})-\Omega_d(0,\tilde{c})\big|\leq & C_2\big(1+\|\tilde{v}\|_{L^2_w}^2\big)\|\tilde{v}\|_{L^2_w},\label{eqn:cdcontrol}\\
\big|c_f(\tilde{v},\tilde{c})-c_f(0,\tilde{c})\big|+\big|\Omega_f(\tilde{v},\tilde{c})-\Omega_f(0,\tilde{c})\big|\leq & C_2\|\tilde{v}\|_{L^2_w}.\label{eqn:c?control}
\end{align}
\end{lemma}
\begin{proof}
    Setting out to control $K_{\tilde{c}}^{-1}(\tilde{v})$, we note that
\[K_{\tilde{c}}(0)=-\frac{9}{2}\begin{bmatrix}
{\tilde{c}}^{1/2}& 0\\
{\tilde{c}}^{-1} &{\tilde{c}}^{1/2}
\end{bmatrix}\]
is invertible, so that we can find constants $\tilde{C}_1,\tilde{C}_2>0$ that ensure $\|A^{-1}\|_{\mathrm{op}}\leq \tilde{C}_2$ for all $A\in \R^{2\times 2}$ which satisfy
\[\big|A_{ij}-[K_{{\tilde{c}}}(0)]_{ij}\big|\leq \tilde{C}_1, \quad  (i,j) \in \{1,2\}^2.\]
Here, $\|\cdot\|_{\mathrm{op}}$ denotes the operator-norm on $\big(\R^2, \|\cdot\|_1\big)$, chosen for convenience in the computations below. Now note that 
\begin{align*}
    \big|[K_{{\tilde{c}}}(\tilde{v})]_{ij}-[K_{{\tilde{c}}}(0)]_{ij}\big|\leq & \Big(\big|\langle \tilde{v},\partial_c\phi_{\tilde{c}}\rangle\big|+\big|\langle \tilde{v},\partial_x\phi_{\tilde{c}}\rangle\big|+\big|\langle \tilde{v},\partial_c\zeta_{\tilde{c}}\rangle \big|\Big)\\
    \leq & \big(\|\partial_c\phi_{\tilde{c}}\|_{L^2_{-w}}+\|\partial_x\phi_{\tilde{c}}\|_{L^2_{-w}}+\|\partial_c\zeta_{\tilde{c}}\|_{L^2_{-w}}\big)\|\tilde{v}\|_{L^2_w}
\end{align*}
for all $(i,j)\in\{1,2\}^2$. Ensuring $\delta_2\leq \big(\|\partial_c\phi_{\tilde{c}}\|_{L^2_{-w}}+\|\partial_x\phi_{\tilde{c}}\|_{L^2_{-w}}+\|\partial_c\zeta_{\tilde{c}}\|_{L^2_{-w}}\big)^{-1}\tilde{C}_1$, it follows that $\big\|K_{\tilde{c}}^{-1}(\tilde{v})\big\|_{\mathrm{op}}\leq \tilde{C}_2$. Estimating 
\begin{align*}
     \Big|\big\langle N(\tilde{v}),\phi_{{\tilde{c}}}\big\rangle\Big|+\Big|\big\langle N(\tilde{v}),\zeta_{{\tilde{c}}}\big\rangle\Big|=&\Big|\big\langle \partial_x(\tilde{v}^2),\phi_{{\tilde{c}}}\big\rangle\Big|+\Big|\big\langle \partial_x(\tilde{v}^2),\zeta_{{\tilde{c}}}\big\rangle\Big|\\
     =&\big|\langle \tilde{v}^2,\partial_x\phi_{{\tilde{c}}}\rangle\big|+\big|\langle \tilde{v}^2,\partial_c\phi_{{\tilde{c}}}\rangle\big|\\
     \leq&\tilde{C}_5\|\tilde{v}\|^2_{L^2_{w}}
\end{align*}
where
\[\tilde{C}_5= \sup_{x\leq 0}\Big(|e^{-2w x}\partial_x\phi_{{\tilde{c}}}(x)|+|e^{-2w x}\partial_c\phi_{{\tilde{c}}}(x)|\Big)\]
establishes \eqref{eqn:cd0control}. Writing
\begin{align*}
    \begin{bmatrix}
        c_f(\tilde{v},{\tilde{c}})-c_f(0,\tilde{c})\\
        \Omega_f(\tilde{v},{\tilde{c}})-\Omega_f(0,\tilde{c})
    \end{bmatrix}=-K_{\tilde{c}}^{-1}(\tilde{v})\begin{bmatrix}
        \langle \tilde{v},\phi_{\tilde{c}}\rangle\\
        \langle \tilde{v},\zeta_{\tilde{c}} \rangle
    \end{bmatrix}+\big(K_{\tilde{c}}(0)^{-1}-K_{\tilde{c}}^{-1}(\tilde{v})\big)\begin{bmatrix}
        \langle \phi_{\tilde{c}},\phi_{\tilde{c}}\rangle\\
        \langle \phi_{\tilde{c}},\zeta_{\tilde{c}} \rangle
    \end{bmatrix}
\end{align*}
establishes \eqref{eqn:c?control}. Writing
\begin{align*}
    \begin{bmatrix}
        Q^{1/2}c_s(\tilde{v},{\tilde{c}})-Q^{1/2}c_s(0,{\tilde{c}})\\
        Q^{1/2}\Omega_s(\tilde{v},{\tilde{c}})-Q^{1/2}\Omega_s(0,{\tilde{c}})
    \end{bmatrix}
    =-K_{\tilde{c}}^{-1}(\tilde{v})\begin{bmatrix}
        Q^{1/2}(\tilde{v}\phi_{\tilde{c}}) \\Q^{1/2}(\tilde{v}\zeta_{\tilde{c}})
    \end{bmatrix}+\big(K_{\tilde{c}}^{-1}(0)-K_{\tilde{c}}^{-1}(\tilde{v})\big)\begin{bmatrix}
        Q^{1/2}(\phi^2_{\tilde{c}}) \\Q^{1/2}(\phi_{\tilde{c}}\zeta_{\tilde{c}})
    \end{bmatrix}
    \end{align*}
    together with     \[\|Q^{1/2}g\|_{L^2}=\|\sqrt{\hat{q}}\hat{g}\|_{L^2}\leq \|\hat{q}\|^{1/2}_\infty\|\hat{g}\|_{L^2}\leq \|q\|^{1/2}_{L^1}\|g\|_{L^2}, \quad g\in L^2\]
    and 
    \[\|\tilde{v}\phi_{\tilde{c}}\|_{L^2}+\|\tilde{v}\zeta_{\tilde{c}}\|_{L^2}\leq \big(\|\phi_{\tilde{c}}\|_{L^2_{-w}}+\|\zeta_{\tilde{c}}\|_{L^2_{-w}}\big)\|\tilde{v}\|_{L^2_w}\]
    establishes \eqref{eqn:cscontrol}. Lastly, \eqref{eqn:cdcontrol} follows in the same way, using
    \begin{align*}
        \Big|\big\langle Y_d(\tilde{v},{\tilde{c}})-Y_d(0,{\tilde{c}}),\phi_{\tilde{c}}\big\rangle\Big|+&\Big|\big\langle Y_d(\tilde{v},{\tilde{c}})-Y_d(0,{\tilde{c}}),\zeta_{\tilde{c}}\big\rangle\Big|\\
        \leq &\tfrac{1}{2}\Big| \big\|Q^{1/2}\Omega_s(\tilde{v},{\tilde{c}})\big\|_{L^2}^2-\big\|Q^{1/2}\Omega_s(0,{\tilde{c}})\big\|_{L^2}^2\Big| \Big( \big|\langle \partial_x^2\phi_{\tilde{c}},\phi_{\tilde{c}}\rangle \big|+\big|\langle \partial_x^2\phi_{\tilde{c}}, \zeta_{\tilde{c}}\rangle \big|\Big)\\
        &\quad +\tfrac{1}{2}\Big| \big\|Q^{1/2}c_s(\tilde{v},{\tilde{c}})\big\|_{L^2}^2-\big\|Q^{1/2}c_s(0,{\tilde{c}})\big\|_{L^2}^2\Big| \Big( \big|\langle \partial_c^2\phi_{\tilde{c}},\phi_{\tilde{c}}\rangle \big|+\big|\langle \partial_c^2\phi_{\tilde{c}}, \zeta_{\tilde{c}}\rangle \big|\Big)\\
        &\quad +\tfrac{1}{2} \big\|Q^{1/2}\Omega_s(\tilde{v},{\tilde{c}})\big\|_{L^2}^2\Big(\big|\langle \tilde{v},\partial_x^2\phi_{\tilde{c}}\rangle\big|+\big|\langle \tilde{v},\partial_x^2\zeta_{\tilde{c}}\rangle \big|\Big)\\
        \leq &\tilde{C}_6 \big(1+\|\tilde{v}\|_{L^2_w}+\|\tilde{v}\|_{L^2_w}^2\big)\|\tilde{v}\|_{L^2_w}
    \end{align*}
    % \begin{align*} 
    %     \begin{bmatrix}
    %       \langle Q^{1/2}Z^*(v,c)[\partial_c\phi_c],Q^{1/2}c_s(v,c) \rangle\\
    %        \langle Q^{1/2}Z^*(v,c)[\partial_c\zeta_c],Q^{1/2}c_s(v,c) \rangle
    % \end{bmatrix}-\begin{bmatrix}
    %       \langle Q^{1/2}Z^*(0)[\partial_c\phi_c],Q^{1/2}c_s(0) \rangle\\
    %        \langle Q^{1/2}Z^*(0)[\partial_c\zeta_c],Q^{1/2}c_s(0) \rangle
    % \end{bmatrix}\\
    % =\begin{bmatrix}
    %       \langle Q^{1/2}(Z^*(v,c)-Z^*(0))[\partial_c\phi_c],Q^{1/2}c_s(v,c) \rangle\\
    %        \langle Q^{1/2}(Z^*(v,c)-Z^*(0))[\partial_c\zeta_c],Q^{1/2}c_s(v,c) \rangle
    % \end{bmatrix}+\begin{bmatrix}
    %       \langle Q^{1/2}Z^*(0)[\partial_c\phi_c],Q^{1/2}(c_s(v,c)-c_s(0))\rangle\\
    %        \langle Q^{1/2}Z^*(0)[\partial_c\zeta_c],Q^{1/2}(c_s(v,c)-c_s(0)) \rangle
    % \end{bmatrix}
    % \end{align*}
    and 
    \begin{align*}\Big\|\big(Z^*(\tilde{v},{\tilde{c}})-Z^*(0,{\tilde{c}})\big)[g]\Big\|_{L^2}\leq & \big\|(g+\phi_{\tilde{c}})\tilde{v}\big\|_{L^2} +\big\|\Omega_s(\tilde{v},{\tilde{c}})-\Omega_s(0,{\tilde{c}})\big\|_{L^2}\big|\langle \partial_x g, \phi_{\tilde{c}}+\tilde{v}\rangle\big|\\
    &+\big\|\Omega_s(\tilde{v},{\tilde{c}})\big\|_{L^2}\big|\langle \partial_x g, \tilde{v}\rangle\big|+  \big\|c_s(\tilde{v},{\tilde{c}})-c_s(0,{\tilde{c}})\big\|_{L^2} \big|\langle g, \partial_c\phi_{{\tilde{c}}}\rangle\big|\\
    \leq & \tilde{C}_7\big(1+\|\tilde{v}\|_{L^2_w}\big)\|\tilde{v}\|_{L^2_w}.\end{align*}
\end{proof}

\section{Local modulation system}
\label{sec:local}
Now that we have set up the modulation system $\big(v(t),c(t),\xi(t)\big)$, we turn our attention to the main goal of this paper: asserting that the remainder $v(t)$ defined through 
\[v(t,x)=u\big(t,x+\xi(t)\big)-\phi_{c(t)}(x)\]
remains small (measured in the norm $H^1_w)$. However, we do not base our arguments on the operator $\mathcal{L}_{c(t)}$ that represents the linear part of \eqref{eqn:vmod}. The main reason is that $\mathcal{L}_{c(t)}$ is non-autonomous, which complicates stability arguments based on the stability properties of the flow generated by $\mathcal{L}_{c}$. A second reason is that $v$ is defined through a stochastic shift of $u$, and hence \eqref{eqn:vmod} contains the term $\partial_x^2v$, which presents regularity issues. 

Given $T\geq0$, we therefore introduce a \textit{local} modulation system $\mathbf{m}^T:=\big(v^T,c^T,\xi^T\big)$ through
\begin{align}v^T(s,x)=&u\big(T+s,x+\xi(T)+c(T)s\big)-\Phi^T(\mathbf{m}^T(s),s,x),\label{eqn:local}\end{align}
where we have abbreviated 
\[\Phi^T(\mathbf{m}^T(s),s,x):=\phi_{c^T(s)}\big(x+\xi(T)+c(T)s-\xi^T(s)\big).\]
The local remainder $v^T(s)$ is defined by shifting $u$ with constant velocity $c(T)$. This freezes the wave around the origin in the absence of forcing, while avoiding an It\^o correction term of the form $\partial_x^2$. The parameter
$\xi^T(s)$ then has the interpretation as soliton position, and accounts for corrections due to the forcing. The local modulation parameters $c^T(s)$ and $\xi^T(s)$ are uniquely determined through the condition
\begin{align}\big\langle v^T(s),\phi_{c(T)}\big\rangle =\big\langle v^T(s),\zeta_{c(T)}\big\rangle =0.\label{eqn:ortholocal}\end{align}
In particular, we have  
\[\big(v^T(0),c^T(0),\xi^T(0)\big)=\big(v(T),c(T),\xi(T)\big).\]
In the absence of noise and forcing, with $v(T)=0$, the local modulation parameters keep their constant value $c^T(s)=c(T)$ and $\xi^T(s)=\xi(T)$. The unique existence of the decomposition \eqref{eqn:local} is guaranteed by the following lemma, a variation on \Cref{lem:implicit} tailored to the condition \eqref{eqn:ortholocal}.
\begin{lemma}\label{lem:implicitlocal}
    Assuming \hyperlink{hyp:weight}{S2}, there exists a constant $\delta_3>0$ so that for each $v_* \in H^1_w\cap H^1$ and $c_*,c_0\in[c_{\mathrm{min}},c_{\mathrm{max}}]$ with $\|v_*\|_{L^2_w},|c_*-c_0|\leq \delta_3$, there exist unique parameters $c>0, \xi\in \R$ and a unique function $v\in H^1_w\cap H^1$ that together enforce the identities
    \[\phi_{c_*}(x)+v_*(x)=\phi_{c}(x-\xi)+v(x-\xi) \quad \text{with} \quad \langle v,\phi_{c_0}\rangle =\langle v,\zeta_{c_0}\rangle=0.\]
\end{lemma}
Existence of the local decomposition is thus guaranteed as long as the remainder $v^T(s)$ and the difference $|c(T)-c^T(s)|$ remain under $\delta_3$. The advantage of defining the local system through the conditions \eqref{eqn:ortholocal} is that it demands orthogonality with respect to \textit{fixed} eigenfunctions $\phi_{c(T)}$ and $\zeta_{c(T)}$, facilitating the use of the stability properties of $\{e^{\mathcal{L}_{c(T)}t}\}_{t\geq 0}$ to control the growth of $v^T(s)$. % This local modulation system is constructed to approximate the global system $\big(v(t),c(t),\xi(t)\big)$ on a time interval $[T, T+\Delta T]$, while facilitating analysis of the remainder $v^T$. 
Rewriting \eqref{eqn:local} as 
\begin{align*}
v^T(s,x)
=&v\big(T+s,x+\xi(T)+c(T)s-\xi(T+s)\big)\\
&+\phi_{c(T+s)}\big(x+\xi(T)+c(T)s-\xi(T+s)\big)-\Phi^T(\mathbf{m}^T(s),s,x),
\end{align*}
we observe that as long as the local modulation parameters $c^T(s)$ and $\xi^T(s)$ do not deviate substantially from their global counterparts $c(T+s)$ and $\xi(T+s)$, then neither do $v^T(s)$ and $v(T+s)$. As long as this holds, we can understand the growth of $v(T+s)$ through $v^T(s)$. The following lemma asserts this correspondence.

\begin{lemma}\label{lem:correspondence}
Assuming \hyperlink{hyp:initial}{S1} and \hyperlink{hyp:weight}{S2}, there exists a constant $C_4>0$ such that the following holds true. For all $T,s,\delta\geq 0$, the inclusions
\[c^T(s),c(T+s)\in[\tfrac{1}{2}c_{\mathrm{min}},  2c_{\mathrm{max}}]\]
and the bound
\[\big|c^T(s)-c(T+s)\big|+\big|\xi^T(s)-\xi(T+s)\big|\leq \delta,\]
imply
\[\Big|\big\|v(T+s)\big\|_{H^1_w}-e^{w(\xi(T)+c(T)s-\xi(T+s))}\big\|v^T(s)\big\|_{H^1_w}\Big|\leq C_4\delta.\]
\end{lemma}
  \begin{proof}
      We compute
\begin{align*}
&\Big|\big\|e^{w\cdot}v(T+s,\cdot)\big\|_{H^1}-e^{w(\xi(T)+c(T)s-\xi(T+s))}\big\|e^{w\cdot}v^T(s,\cdot)\big\|_{H^1}\Big|\\
    =& \Big|\big\|e^{w\cdot}v(T+s,\cdot)\big\|_{H^1}-\big\|e^{w\cdot}v^T(s,\cdot-\xi(T)-c(T)s+\xi(T+s))\big\|_{H^1}\Big|.
    \end{align*}
Using the reverse triangle inequality, we get
\begin{align*}
   &\Big|\big\|e^{w\cdot}v(T+s,\cdot)\big\|_{H^1}-e^{w(\xi(T)+c(T)s-\xi(T+s))}\big\|e^{w\cdot}v^T(s,\cdot)\big\|_{H^1}\Big|\\
   \leq& \Big\|v(T+s)-v^T\big(s,\cdot-\xi(T)-c(T)s+\xi(T+s)\big)\Big\|_{H^1_w}\\
    =&\Big\|\phi_{c(T+s)}-\phi_{c^T(s)}\big(\cdot+\xi(T+s)-\xi^T(s)\big)\Big\|_{H^1_w}.
 \end{align*}  
The result now follows by exploiting the $O(e^{-\sqrt{c}|x|})$ decay of the wave-profile
\[\phi_c(x)=\tfrac{3c}{2}\sech^2(\sqrt{c}x/2)=\frac{6c}{(e^{-\sqrt{c}x/2}+e^{\sqrt{c}x/2})^2}\]
and its derivatives $\partial_x\phi_c, \partial_x^2\phi_c, \partial_c\phi_c$ and $\partial^2_{cx}\phi_c$. Indeed, it implies that the map $ c \mapsto e^{wx}\phi_c(x)+e^{wx}\partial_x\phi_c(x)$ is Lipschitz from $[0,2c_{\mathrm{max}}]$ to $L^2$.
\end{proof}

Let us proceed by describing the dynamics of the local modulation system. We once again introduce a phase-shift parameter $\Omega^T(s)$ through
\[\xi^T(s)=\int_0^s c^T(s')\d s'+\Omega^T(s)\] 
and will see that $c^T$ and $\Omega^T$ satisfy SDEs of the form
\begin{align}
    \d c^{T} =& c_d^{\sigma,\epsilon,T}(\mathbf{m}^T(s),s) \ \d s+\sigma \big\langle c^{T}_s(\mathbf{m}^T(s),s),T_{\xi(T)+c(T)s}\d W_{T+s}^Q\big\rangle,  \label{eqn:localc}\\
    \d {\Omega}^{T}  =& \Omega_d^{\sigma,\epsilon,T}(\mathbf{m}^T(s),s)  \ \d s+\sigma\big\langle \Omega^{T}_s(\mathbf{m}^T(s),s),T_{\xi(T)+c(T)s}\d W_{T+s}^Q\big\rangle.\label{eqn:localomega}
\end{align}
A formal application of It\^{o}'s lemma then shows that
\begin{align}\d v^T=\mathcal{L}_{c(T)}v^T\ \d s+Y^{\sigma,\epsilon,T}(\mathbf{m}^T(s),s)\ \d s+ \sigma Z^{T}(\mathbf{m}^T(s),s)\ T_{\xi(T)+c(T)s}\d W^Q_{T+s},\label{eqn:vTstrong}\end{align}
where
\begin{align}
    Y^{\sigma,\epsilon,T}(\mathbf{m}^T,s)=&Y_I^{\sigma,\epsilon,T}(\mathbf{m}^T,s)+\epsilon f(T+s)v^T\label{eqn:YT}
\end{align}
with
\begin{align}Y_{I}^{\sigma,\epsilon,T}(\mathbf{m}^T,s)=&2\partial_x\big((\phi_{c(T)}-\Phi^T(\mathbf{m}^T,s))v^T\big)+N(v^T)+\epsilon f(T+s)\Phi^T(\mathbf{m}^T,s)\label{eqn:YTI}\\
&-c_d^{\sigma,\epsilon,T}(\mathbf{m}^T,s) \partial_c\Phi^T(\mathbf{m}^T,s)+\Omega_d^{\sigma,\epsilon,T}(\mathbf{m}^T,s)\partial_x\Phi^T(\mathbf{m}^T,s)\nonumber\\
&+\tfrac{1}{2}\sigma^2\Big[\big\|Q^{1/2}c_s^{T}(\mathbf{m}^T,s)\big\|_{L^2}^2\partial_c^2 +\big\|Q^{1/2}\Omega_s^{T}(\mathbf{m}^T,s)\big\|_{L^2}^2\partial_x^2\Big]\Phi^T(\mathbf{m}^T,s) \nonumber\end{align}
and 
\begin{align}
    Z^{T}(\mathbf{m}^T,s)[h]= \big(\Phi^T(\mathbf{m}^T,s)+v^T\big)h+\Big(-\big\langle c_s^{T}(\mathbf{m}^T,s),h\big\rangle  \partial_c+\big\langle \Omega_s^{T}(\mathbf{m}^T,s),h\big\rangle \partial_x\Big)\Phi^T(\mathbf{m}^T,s).\label{eqn:Zt}
\end{align}
In \eqref{eqn:YT} and \eqref{eqn:Zt}, $\partial_c\Phi^T$ should be interpreted as
\[\partial_c\Phi^T(\mathbf{m}^T,s,x):=\partial_c\phi_{c^T}\big(x+\xi(T)+c(T)s-\xi^T\big).\]
 However, we can not rigorously justify \eqref{eqn:vTstrong}, as $v^T$ is not regular enough for $\mathcal{L}_{c(T)}v^T$ to be well-defined. We therefore pass to a mild formulation with respect to the flow generated by the linear equation $w_t=\mathcal{L}_{c(T)}w$ (see \Cref{thm:linearstability}).

\begin{proposition}[See \Cref{app:mild}]\label{prop:mild}
Assume  \hyperlink{hyp:initial}{S1} and \hyperlink{hyp:weight}{S2}. For each $T,s\geq0$, the inequalities
\[\|v^T(s')\|_{L^2_w}\leq \delta_3 \quad \text{and} \quad c^T(s')\in[\tfrac{1}{2}c_{\mathrm{min}},2c_{\mathrm{max}}],\quad s' \in[0,s]\] 
imply that
\begin{align}v^T(s)=&e^{\mathcal{L}_{c(T)}s}v(T)+\int_0^se^{\mathcal{L}_{c(T)}(s-s')}Y^{\sigma,\epsilon,T}(\mathbf{m}^T(s'),s')\d s'\nonumber\\
&+ \sigma \int_0^se^{\mathcal{L}_{c(T)}(s-s')}Z^{T}(\mathbf{m}^T(s'),s')\ T_{\xi(T)+c(T)s'}\d W^Q_{T+s'} ,\label{eqn:mild}\end{align}
$\mathbb{P}$-almost surely.
\end{proposition}

 It follows straightforwardly that
 \begin{align*}\d \big\langle v^T(s),\phi_{c(T)}\big\rangle =& \Big\langle Y^{\sigma,\epsilon,T}\big(\mathbf{m}^T(s),s\big),\phi_{c(T)}\Big\rangle \ \d s+\sigma \Big\langle Z^T\big(\mathbf{m}^T(s),s\big)[T_{\xi(T)+c(T)s}\d W_{T+s}^Q],\phi_{c(T)}\Big\rangle,\\
\d \big\langle v^T(s),\zeta_{c(T)}\big\rangle =& \Big\langle Y^{\sigma,\epsilon,T}\big(\mathbf{m}^T(s),s\big),\phi_{c(T)}\Big\rangle \ \d s+\sigma \Big\langle Z^T\big(\mathbf{m}^T(s),s\big)[T_{\xi(T)+c(T)s}\d W_{T+s}^Q],\zeta_{c(T)}\Big\rangle.
 \end{align*}
For the orthogonality conditions \eqref{eqn:ortholocal} to hold, we must have
\[\begin{bmatrix}
    \big\langle Z^T(\mathbf{m}^T,s)[h],\phi_{c(T)}\big\rangle\\
    \big\langle Z^T(\mathbf{m}^T,s)[h],\zeta_{c(T)}\big\rangle
\end{bmatrix}=0,\quad \forall h \in L^2_Q.\]
This shows that
\begin{align*}
         \begin{bmatrix}
        c^{T}_s(\mathbf{m}^T,s)\\
        \Omega^{T}_s(\mathbf{m}^T,s)
    \end{bmatrix}= (K^T)^{-1}(\mathbf{m}^T,s)\begin{bmatrix}
         (\Phi^T(\mathbf{m}^T,s)+v^T)\phi_{c(T)} \\
        ( \Phi^T(\mathbf{m}^T,s)+v^T)\zeta_{c(T)}
    \end{bmatrix},
\end{align*}
where 
\[K^T(\mathbf{m}^T,s)=\begin{bmatrix}
    \langle 
    \partial_c \Phi^T(\mathbf{m}^T,s),\phi_{c(T)}\rangle & -\langle 
    \partial_x \Phi^T(\mathbf{m}^T,s),\phi_{c(T)}\rangle\\
   \langle \partial_c \Phi^T(\mathbf{m}^T,s),\zeta_{c(T)}\rangle & -\langle 
    \partial_x \Phi^T(\mathbf{m}^T,s),\zeta_{c(T)}\rangle
    \end{bmatrix}.\]
The drift components $c_d^{\sigma,\epsilon,T}$, and
$\Omega_d^{\sigma,\epsilon,T}$ follow by solving
\begin{align*}
    \begin{bmatrix}
    \big\langle Y^{\sigma,\epsilon,T}(\mathbf{m}^T,s),\phi_{c(T)}\big\rangle\\
    \big\langle Y^{\sigma,\epsilon,T}(\mathbf{m}^T,s),\zeta_{c(T)}\big\rangle
\end{bmatrix} = 0,
\end{align*}
leading to 
\begin{align*}
    c_d^{\sigma,\epsilon,T}(\mathbf{m}^T,s)=&c^{0,T}_d(\mathbf{m}^T,s)+\epsilon f(T+s)c^T_f(\mathbf{m}^T,s)+\sigma^2 c^T_d(\mathbf{m}^T,s) \\
\Omega_d^{\sigma,\epsilon,T}(\mathbf{m}^T,s)=&\Omega^{0,T}_d(\mathbf{m}^T,s)+\epsilon f(T+s)\Omega^T_f(\mathbf{m}^T,s)+\sigma^2 \Omega^T_d(\mathbf{m}^T,s)
\end{align*}
with
\begin{align*}
    \begin{bmatrix}
        c^{0,T}_d(\mathbf{m}^T,s)\\
        \Omega^{0,T}_d(\mathbf{m}^T,s)
    \end{bmatrix}=&(K^T)^{-1}(\mathbf{m}^T,s)\begin{bmatrix}
        \big\langle N(v^T),\phi_{c(T)} \big\rangle \\
        \big\langle N(v^T),\zeta_{c(T)} \big\rangle
    \end{bmatrix}\\&+2(K^T)^{-1}(\mathbf{m}^T,s)\begin{bmatrix}
        \Big\langle \partial_x\big((\phi_{c(T)}-\Phi^T(\mathbf{m}^T,s))v^T\big),\phi_{c(T)}\Big\rangle\\
        \Big\langle \partial_x\big((\phi_{c(T)}-\Phi^T(\mathbf{m}^T,s))v^T\big),\zeta_{c(T)}\Big\rangle
    \end{bmatrix},\\
 \begin{bmatrix}
        c^{T}_f(\mathbf{m}^T,s)\\
        \Omega^{T}_f(\mathbf{m}^T,s)
    \end{bmatrix}=& (K^T)^{-1}(\mathbf{m}^T,s)\begin{bmatrix}
        \langle \Phi^T(\mathbf{m}^T,s)+v^T,\phi_{c(T)} \rangle \\
        \langle \Phi^T(\mathbf{m}^T,s)+v^T,\zeta_{c(T)} \rangle
    \end{bmatrix},\\
   \begin{bmatrix}
        c^{T}_d(\mathbf{m}^T,s)\\
        \Omega^{T}_d(\mathbf{m}^T,s)
    \end{bmatrix}=&\tfrac{1}{2} \|Q^{1/2}c^T_s(\mathbf{m}^T,s)\|_{L^2}^2(K^T)^{-1}(\mathbf{m}^T,s)\begin{bmatrix}
        \Big\langle \partial_c^2 \Phi^T(\mathbf{m}^T,s),\phi_{c(T)} \Big\rangle\\
        \Big\langle \partial_c^2\Phi^T(\mathbf{m}^T,s) ,\zeta_{c(T)} \Big\rangle
    \end{bmatrix}\\
    &+\tfrac{1}{2}\|Q^{1/2}\Omega_s^T(\mathbf{m}^T,s)\|_{L^2}^2(K^T)^{-1}(\mathbf{m}^T,s)\begin{bmatrix}
        \Big\langle \partial_x^2 \Phi^T(\mathbf{m}^T,s),\phi_{c(T)} \Big\rangle\\
        \Big\langle \partial_x^2\Phi^T(\mathbf{m}^T,s) ,\zeta_{c(T)} \Big\rangle
    \end{bmatrix}.
\end{align*}

We now establish control on the local modulation parameters, assuming a-priori control of the quantity
\begin{align}R^T_w(s):=\|v^T(s)\|_{L^2_w}+\big|c(T)-c^T(s)\big|+\big|\xi(T)+c(T)s-\xi^T(s)\big|.\label{eqn:aprioriw}\end{align}

\begin{lemma}\label{lem:locmodcontrol}
Assuming \hyperlink{hyp:initial}{S1} and \hyperlink{hyp:weight}{S2}, there exist constants $\delta_5, C_5>0$ such that following holds true. For all $T,s\geq 0$, the bounds
\[c(T),c^T(s)\in [\tfrac{1}{2}c_{\mathrm{min}},2c_{\mathrm{max}}]\quad \text{and}\quad R^T_w(s)\leq \delta_5\]
imply
\begin{align}
\big\|Q^{1/2}c^T_s(\mathbf{m}^T(s),s)\big\|_{L^2}+\big\|Q^{1/2}\Omega^T_s(\mathbf{m}^T(s),s)\big\|_{L^2}\leq & C_5\big(1+\|v^T(s)\|_{L^2_w}\big),\label{eqn:ctscontrol}\\
    \big|c_d^{T,0}(\mathbf{m}^T(s),s)\big|+\big|\Omega_d^{T,0}(\mathbf{m}^T(s),s)\big|
    \leq & C_5 R^T_w(s)\|v^T(s)\|_{L^2_w},\label{eqn:ctd0control}\\
    \big|c_f(\mathbf{m}^T(s),s)\big|+\big|\Omega_f(\mathbf{m}^T(s),s)\big|\leq & C_5\big(1+\|v^T(s)\|_{L^2_w}\big),\label{eqn:ct?control}\\
\big|c_d(\mathbf{m}^T(s),s)\big|+\big|\Omega_d(\mathbf{m}^T(s),s)\big|\leq & C_5\big(1+\|v^T(s)\|_{L^2_w}\big).\label{eqn:ctdcontrol}
\end{align}
\end{lemma}
\begin{proof}
   As in the proof of \Cref{lem:modcontrol}, note that
   \[K^T(\mathbf{m}^T(0),0)=\begin{bmatrix}
    \langle 
    \partial_c \phi_{c(T)},\phi_{c(T)}\rangle & 0\\
   \langle \partial_c \phi_{c(T)},\zeta_{c(T)}\rangle & -\langle 
    \partial_x \phi_{c(T)},\zeta_{c(T)}\rangle
    \end{bmatrix}.\]
   In particular, we can find constants $\tilde{C}_1,\tilde{C}_2>0$ that ensure $\|A^{-1}\|_{\mathrm{op}}\leq \tilde{C}_2$ for all $A\in \R^{2\times 2}$ which satisfy
\[\Big|A_{ij}-K_{ij}^T(\mathbf{m}^T(0),0)\Big|\leq \tilde{C}_1, \quad  (i,j) \in \{1,2\}^2.\]
Recalling that  $\Phi^T(\mathbf{m}^T(s),s)=\phi_{c^T(s)}\big(\cdot+\xi(T)+c(T)s-\xi^T(s)\big)$, the Lipschitz properties of $\phi_c$ imply
\begin{align*}
    \Big|K_{ij}^T(\mathbf{m}^T(s),s)-K_{ij}^T(\mathbf{m}^T(0),0)\Big|\leq  & \tilde{C}_3\Big(\big|\xi(T)+c(T)s-\xi^T(s)\big|+\big|c^T(s)-c(T)\big|\Big)
\end{align*}
for all $(i,j)\in\{1,2\}^2$. In case $\delta_5\leq \tilde{C}_3^{-1}\tilde{C}_1$, it follows that $\big\|(K^T)^{-1}(\mathbf{m}^T,s)\big\|_{\mathrm{op}}\leq \tilde{C}_2$. Estimating 
\begin{align*}
    \Big|\Big\langle \partial_x\big((\phi_{c(T)}-\Phi^T(\mathbf{m}^T,s))v^T(s)\big),\phi_{c(T)}\Big\rangle\Big|+
        \Big|\Big\langle \partial_x\big((\phi_{c(T)}-\Phi^T(\mathbf{m}^T(s),s))v^T(s)\big),\zeta_{c(T)}\Big\rangle\Big|\\
        \leq \big(\|\partial_x\phi_{c(T)}\|_{\infty}+\|\partial_x\phi_{c(T)}\|_{\infty}\big)\|\phi_{c(T)}-\Phi^T(\mathbf{m}^T(s),s)\|_{L^2_{-w}}\|v^T(s)\|_{L^2_w}
\end{align*}
 establishes \eqref{eqn:ctd0control}. The estimates \eqref{eqn:ctscontrol}, \eqref{eqn:ct?control} and \eqref{eqn:ctdcontrol} follow by applying the Cauchy-Schwarz inequality.
\end{proof}
We conclude this section with a result on the deterministic integral in \eqref{eqn:mild}. This is provided by \Cref{cor:shortest} below and forms the basis of our stability arguments. Our estimates for the deterministic terms ($\sigma=0$) in \eqref{eqn:mild} are similar to those in \cite{pegoweinstein}, save for our treatment of the nonlinearity $N(v^T)=-\partial_x(v^T)^2$. Here we follow the approach of Mizumachi and Tzvetkov \cite{mizumachi} which uses property \eqref{eqn:linearL1}. This allows us to control the nonlinear term based on the condition that $\|v^T\|_{L^2}$ is small. This is an improvement over the standard argument used in \cite{pegoweinstein}, which requires control of $\|\partial_x v^T\|_{L^2}$ and consequently a more cumbersome energy argument. 
% We introduce the quantity 
% \begin{align}\label{eqn:apriori}
%     R_0^T(s):=\big|c(T)-c^T(s)\big|+\big|\xi(T)+c(T)s-\xi^T(s)\big|+\|v^T(s)\big\|_{L^2},
% \end{align}
% for notational convenience below, which should be compared to $R_w^T(s)$ defined in \eqref{eqn:aprioriw}.
\begin{lemma}\label{lem:L1}
Assuming \hyperlink{hyp:initial}{S1} and \hyperlink{hyp:weight}{S2}, there exists a constant $C_6>0$ such that for each $T,s\geq 0$ the inequalities
\[c(T),c^T(s)\in [\tfrac{1}{2}c_{\mathrm{min}},2c_{\mathrm{max}}]\quad \text{and}\quad R^T_w(s)\leq \delta_5\]
imply
    \begin{align*}\Big\|Y_I^{\sigma,\epsilon,T}(\mathbf{m}^T(s),s)\Big\|_{L^1_w} \leq&C_{6}\sigma^2\big(1+\|v^T(s)\|_{L^2_w}\big)^2+C_6\epsilon\\&+C_6\Big(\|v^T(s)\|_{L^2}+R^T_w(s)\Big)\|v^T(s)\|_{H^1_w}.
\end{align*}
\end{lemma}
\begin{proof}
    The various terms in $Y_I^{\sigma,\epsilon,T}\big(\mathbf{m}^T(s),s\big)$ -- see \eqref{eqn:YTI} -- can be controlled as follows. The inequality
\begin{align*}
    \Big\|2\partial_x\big((\phi_{c(T)}-\Phi^T(\mathbf{m}^T(s),s)) v^T \big)\Big\|_{L^1_w}\leq & 2\big\|\phi_{c(T)}-\Phi^T(\mathbf{m}^T(s),s)\big\|_{H^1}\big\|v^T(s)\big\|_{H^1_w}\\
    \leq & \tilde{C}_1\Big(\big|c(T)-c^T(s)\big|+\big|\xi(T)+c(T)s-\xi^T(s)\big|\Big) \big\| v^T(s)\big\|_{H^1_w}
\end{align*}
follows from the fact that $c \mapsto \phi_c(x)+\partial_x\phi_c(x)$ is Lipschitz from $[0,c_{\mathrm{max}}]$ to $L^2$.
Via \Cref{lem:locmodcontrol} we find
\begin{align*}&\Big\|\Big[-c_d^{\sigma,\epsilon,T}(\mathbf{m}^T(s),s) \partial_c+\Omega_d^{\sigma,\epsilon,T}(\mathbf{m}^T(s),s)\partial_x \Big]\Phi^T(\mathbf{m}^T(s),s)\Big\|_{L^1_w}\\
\leq& \tilde{C}_2\Big|c_d^{\sigma,\epsilon,T}(\mathbf{m}^T(s),s) +\Omega_d^{\sigma,\epsilon,T}(\mathbf{m}^T(s),s) \Big| \\
\leq&\tilde{C}_2C_5R^T_w(s)\|v^T(s)\|_{L^2_w},
\end{align*}
where $\tilde{C}_2$ is a constant large enough to ensure
\[\|\partial_c\Phi^T(\mathbf{m}^T(s),s)\|_{L^1_w}+\|\partial_x\Phi^T(\mathbf{m}^T(s),s)\|_{L^1_w}\leq \tilde{C}_2.\]
Similarly,
\begin{align*}\big\|Q^{1/2}c_s^{T}(\mathbf{m}^T(s),s)\big\|_{L^2}^2\big\|\partial_c^2\Phi^T(\mathbf{m}^T(s),s)\big\|_{L^1_w}
% \leq& \tilde{C}_2\Big|c_d^{\sigma,\epsilon,T}(\mathbf{m}^T,s) +\sigma^2\big\|Q^{1/2}c_s^{T}(\mathbf{m}^T,s)\big\|_{L^2}^2+\Omega_d^{\sigma,\epsilon,T}(\mathbf{m}^T,s) +\sigma^2\big\|Q^{1/2}\Omega_s^{T}(\mathbf{m}^T,s)\big\|_{L^2}^2\Big| \\
\leq &  \tilde{C}_3\big(1+\|v^T(s)\|_{L^2_w}\big)^2,
\end{align*}
and
\begin{align*}
\big\|Q^{1/2}\Omega_s^{T}(\mathbf{m}^T(s),s)\big\|_{L^2}^2\big\|\partial_x^2\Phi^T(\mathbf{m}^T(s),s)\big\|_{L^1_w}
\leq &  \tilde{C}_3\big(1+\|v^T(s)\|_{L^2_w}\big)^2.
\end{align*}
Lastly,
% \[\big\|\int_0^t e^{\mathcal{L}_{c(T)}(t-s)}\partial_x(v^T(s)^2)\d s \big\|_{H^1_w} \leq M  \int_0^\infty e^{-b(t-s)}(t-s)^{-3/4}\d s \sup_{s\in [0,t]}\|\partial_x(v^T(s)^2)\|_{L^1_w}\]
% where 
\[\Big\|\partial_x\big((v^T(s))^2\big)\Big\|_{L^1_w}\leq 2\|v^T(s)\|_{L^2}\|v^T(s)\|_{H^1_w}, \]
and
\[\big\|\epsilon f(T+s)\Phi^T(\mathbf{m}^T(s),s)\big\|_{L^1_w}\leq \tilde{C}_4\epsilon.\]
\end{proof}

\begin{corollary}
\label{cor:shortest}
Assuming \hyperlink{hyp:initial}{S1} and \hyperlink{hyp:weight}{S2}, there exists a constant $C_7>0$ such that for each $T,s\geq 0$ the inequalities
\[ c_{\mathrm{min}}\leq c^T(s') \leq c_{\mathrm{max}} \quad \text{and} \quad R^T_w(s')\leq\delta_5,\quad s'\in[0,s]\]
imply that
\begin{align}
   \Big\|\int_0^s e^{\mathcal{L}_{c(T)}(s-s')}Y^{\sigma,\epsilon,T}\big(\mathbf{m}^T(s'),s')\d s'\Big\|_{H^1_w}
   \leq &C_7s\sup_{0\leq s'\leq s}\Big(\|v^T(s')\|_{L^2}+R^T_w(s')\Big) \big\| v^T(s')\big\|_{H^1_w}\nonumber\\
    &+C_7(\sigma^2+\epsilon )s,\label{eqn:Yint}
\end{align}
% and
% \[\Big\|Z^{T}\big(v^T(s)\big)T_{\xi(T)+c(T)s}\Big\|_{\mathrm{HS}(L^2_Q,H^1_w)}\leq C_7\Big(1+\big\|v^T(s)\big\|_{H^1_w}\Big), \quad s\in [0,t]\]
holds $\mathbb{P}$-almost surely.
\end{corollary}
\begin{proof}
Consider the decomposition
\[Y^{\sigma,\epsilon,T}\big(\mathbf{m}^T(s'),s'\big)=\underbrace{Y_I^{\sigma,\epsilon,T}\big(\mathbf{m}^T(s'),s'\big)}_{ \text{in } L^1_w}+\underbrace{\epsilon f(T+s')v^T(s')}_{ \text{in } L^2_w}.\] 
Both terms are contained in the stable subspace of $\big\{e^{\mathcal{L}_{c(T)}t}\big\}_{t\geq 0}$. We have via \eqref{eqn:linear}
\[\Big\|\int_0^s e^{\mathcal{L}_{c(T)}(s-s')}\big[f(T+s')v^T(s')\big]\d s' \Big\|_{H^1_w} \leq M  \int_0^s e^{-b(s-s')}(s-s')^{-1/2}\d s' \sup_{s'\in [0,s]}\big\|v^T(s')\big\|_{L^2_w}.\]
Here, $b$ and $M$ are the constants appearing in the semigroup-bounds \eqref{eqn:linear} and \eqref{eqn:linearL1}. We remark also that 
\[\sup_{s\geq 0}\int_0^s e^{-b(s-s')}(s-s')^{-1/2}\d s'<\infty.\]
Using \eqref{eqn:linearL1}, on the other hand, we have
\begin{align*}
    &\Big\|\int_0^s e^{\mathcal{L}_{c(T)}(s-s')}\Big[Y_I^{\sigma,\epsilon,T}\big(\mathbf{m}^T(s'),s'\big)\Big]\d s' \Big\|_{H^1_w} \\ \leq &M  \int_0^s e^{-b(s-s')}(s-s')^{-3/4}\d s\ \sup_{s'\in [0,s]}\Big\|Y_I^{\sigma,\epsilon,T}\big(\mathbf{m}^T(s'),s'\big)\Big\|_{L^1_w},\end{align*}
    where also
    \[\sup_{s\geq 0}\int_0^s e^{-b(s-s')}(s-s')^{-3/4}\d s<\infty.\]
 The result now follows by applying \Cref{lem:L1}.
\end{proof}

\section{Weighted norm control}
\label{sec:weighted}
In this section, we control the local remainder $v^T$ on time intervals $[0,\Delta T]$. We do so by exploiting the stability properties of the linear flow $\{e^{\mathcal{L}_{c(T)}t}\}_{t\geq 0}$ on weighted spaces. As pointed out, control of the local remainder $v^T$ transfers to the global remainder $v$ via \Cref{lem:correspondence}. 
The main result of this section is \Cref{prop:remainder} below, which bounds the probability that $v^T$ grows large on a time interval $[0,\Delta T]$. We show that $\|v^T\|_{H^1_w}$ only grows large with small probability, provided that 
\begin{itemize}
    \item[--] the soliton amplitude remains within fixed bounds;
    \item[--] the unweighted $L^2$-norm of $v^T$ remains small;
    \item[--] the difference between global and local modulation parameters remains small.
\end{itemize}
 An important ingredient for ensuring that we can repeat our stability argument on time intervals of size $\Delta T$ is the exponential decay of the semigroup after time $\Delta T$ in the first term of \eqref{eqn:mildestimate}. We thus require that $\Delta T$ is large enough to guarantee significant decay. We formulate this in the following condition, and fix constants $\delta_*, \eta_0$ for later use.% with $\delta=\min\{\frac{1}{3},\frac{1}{36M}\}$.
\begin{itemize}
\item[\textbf{C1}]{
  \phantomsection\hypertarget{hyp:constants}{}\textit{The constants $\Delta T,\delta_*>0$ and  $\eta_0>0$ satisfy
  \begin{itemize}
      \item $\Delta T=  \log(6M)/b$;
      \item $\delta_*\leq \min\big\{\frac{1}{216 MC_7(\Delta T+\Delta T^2)},\frac{1}{4}c_{\mathrm{min}},\frac{1}{2}c_{\mathrm{max}},\frac{1}{e^2+3C_4}\big\}$;
      \item $\eta_0\leq \min\big\{\delta_*,\delta_5,1,\frac{1}{8C_5(\Delta T+\Delta T^2)}\big\}$.
  \end{itemize}
 }}
\end{itemize}
%\todo{RW: zijn dit alle aannames?}
The constants $C_4, C_5, \delta_5$ and $C_7$ have been introduced in \Cref{lem:correspondence}, \Cref{lem:locmodcontrol} and \Cref{cor:shortest}, respectively. The constants $b,M$ are introduced in the semigroup-bounds \eqref{eqn:linear}.
To keep track of the weighted norm $\|v^T(s)\|_{H^1_w}$, we define for each $T,\eta>0$ the stopping time $\tau^T_{\mathrm{st}}$ as 
\[\tau^T_{\mathrm{st}}(\eta)=\sup\big\{s\geq 0:\big\|v^T(s)\big\|_{H^1_w}\leq \eta  \big\}.\]
We furthermore introduce stopping times $\tau^T_c, \tau^T_{\mathrm{en}}, \tau^{T}_{\mathrm{amp}}$ and $\tau^{T}_{\mathrm{pos}}$ which encode the conditions for stability:
\begin{align*}
    \tau^T_c=&\sup\big\{s\geq 0:c(T+s)\in[\tfrac{1}{2}c_{\mathrm{min}},2c_{\mathrm{max}}]  \big\};\\
   \tau^T_{\mathrm{en}}(\eta)=&\sup\big\{s\geq 0:\big\|v^T(s)\big\|_{L^2}\leq \eta  \big\};\\
   \tau^{T}_{\mathrm{amp,1}}(\eta)
   =&\sup\big\{s\geq 0:\big|c(T)-c^T(s)\big|\leq \delta_*\eta\}; \\
   \tau^{T}_{\mathrm{amp,2}}(\eta)
   =&\sup\big\{s\geq 0:\big|c(T+s)-c^T(s)\big|\leq \delta_*\eta\}; \\
  \tau^{T}_{\mathrm{pos,1}}(\eta)
   =&\sup\big\{s\geq 0:\big|\xi(T)+c(T)s-\xi^T(s)\big|\leq 2\Delta T \delta_*\eta  \big\};\\
   \tau^{T}_{\mathrm{pos,2}}(\eta)
   =&\sup\big\{s\geq 0:\big|\xi(T+s)-\xi^T(s)\big|\leq 2\Delta T \delta_*\eta  \big\},
\end{align*}
and we define
\[\tau^T_{\mathrm{mod}}(\eta)=\tau^{T}_{\mathrm{amp},1}(\eta)\wedge\tau^{T}_{\mathrm{amp},2}(\eta)\wedge\tau^{T}_{\mathrm{pos,1}}(\eta)\wedge \tau^{T}_{\mathrm{pos,2}}(\eta).\]
 Our result is then as follows.
 
\begin{proposition}[Short-time control]\label{prop:remainder}
Assuming \hyperlink{hyp:initial}{S1}, \hyperlink{hyp:weight}{S2} and \hyperlink{hyp:constants}{C1}, there exist constants $\delta_9> 0$ and $C_9\geq 1$ such that the following holds true. For all $\eta\in[0,\eta_0]$, $C_9\sigma,C_9\epsilon\in [0,\eta]$ and $T\geq 0$ the events
\[\mathcal{E}_1=\big\{\tau^T_{\mathrm{st}}(\eta) \leq \Delta T\wedge\tau^T_{\mathrm{en}}(\delta_*)\wedge\tau^T_{\mathrm{mod}}(\eta)\wedge\tau^T_c\big\}\]
and
\[\mathcal{E}_2=\big\{\big\|v^T(\Delta T)\big\|_{H^1_w}\geq \tfrac{\eta}{9M}\big\}\cap\big\{\tau^T_{\mathrm{st}}(\eta)\wedge\tau^T_{\mathrm{en}}(\delta_*)\wedge\tau^T_{\mathrm{mod}}(\eta)\wedge\tau^T_c\geq\Delta T\big\}\]
satisfy
\begin{align}\mathbb{P}\Big[ \mathcal{E}_1\cap \big\{\|v(T)\|_{H^1_w}\leq \tfrac{\eta}{3M} \big\}\Big]+\mathbb{P}\Big[ \mathcal{E}_2\cap \big\{\big\|v(T)\big\|_{H^1_w}\leq \tfrac{\eta}{3M} \big\}\Big]\leq e^{-\delta_9\eta^2/\sigma^2}.\label{eqn:shortcontrol1}\end{align}
\end{proposition}
Our main tool for establishing \Cref{prop:remainder} is based on the results of \Cref{sec:local}. We recall that the constant $C_7$ was introduced in \Cref{cor:shortest}. 
\begin{lemma}\label{lem:Meta}
    Assume \hyperlink{hyp:initial}{S1} and \hyperlink{hyp:weight}{S2}. For all $\sigma,\epsilon,T,s,\delta\geq0$, each $\delta_*>0$ and $\eta\in[0,\min\{\delta_*,\delta_5,1\}]$ that satisfy
    \begin{align}
s+s^2\leq \frac{\delta}{6C_7\delta_*}\quad \text{and} \quad \max\{\sigma^2s, \epsilon s\}\leq\frac{\delta \eta}{3C_7},\label{eqn:Metaassump}\end{align}
the inequalities  
    \begin{align*}
        \big\|v^T(s')\big\|_{H^1_w}\leq \eta \quad \text{and} \quad 
        \|v^T(s')\|_{L^2}+R^T_w(s')\leq \delta_*(1+2s),\quad s'\in[0,s]
    \end{align*}
    imply
\begin{align}
    \big\|v^T(s)\big\|_{H^1_w}\leq & Me^{-b s} \big\|v(T)\big\|_{H^1_w}+\delta\eta+\sigma \Big\| \int_0^se^{\mathcal{L}_{c(T)}(s-s')}Z^{T}\big(\mathbf{m}^T(s'),s'\big)\ T_{\xi(T)+c(T)s'}\d W^Q_{T+s'}\Big\|_{H^1_w}.\label{eqn:mildestimate}
\end{align}
% and
% \[\Big\|Z^{T}\big(v^T(s)\big)T_{\xi(T)+c(T)s}\Big\|_{\mathrm{HS}(L^2_Q,H^1_w)}\leq C_7\Big(1+\big\|v^T(s)\big\|_{H^1_w}\Big), \quad s\in [0,t]\]
% hold $\mathbb{P}$-a.s. on the set
% \[\big\{\omega\in \Omega: c_{\mathrm{min}}\leq c^T(s) \leq c_{\mathrm{max}} \quad \text{and} \quad\|v(s)\|_{L^2_w}\leq\delta_*,\quad s\in[0,t]\big\}.\]
\end{lemma}
\begin{proof}
We apply \Cref{cor:shortest} to estimate \eqref{eqn:mild}:
    \begin{align*}
    \big\|v^T(s)\big\|_{H^1_w}\leq & Me^{-b s} \big\|v(T)\big\|_{H^1_w}+C_7(\sigma^2+\epsilon)s+2C_7\delta_*(s+s^2)\eta\nonumber\\
 &+\sigma \Big\| \int_0^se^{\mathcal{L}_{c(T)}(s-s')}Z^{T}\big(\mathbf{m}^T(s'),s'\big)\ T_{\xi(T)+c(T)s'}\d W^Q_{T+s'}\Big\|_{H^1_w}.
\end{align*}
The result hence follows from the assumptions \eqref{eqn:Metaassump}.
\end{proof}
% The constants $b$ and $M$ in \eqref{eqn:mildestimate} are the constants appearing in the semigroup-bounds \eqref{eqn:linear}. 
We remark that the constants $\delta_*,\eta_0$ introduced in \hyperlink{hyp:constants}{C1} ensure that \Cref{lem:Meta} may be applied on the interval $[0,\Delta T]$. With \Cref{lem:Meta} established, we furthermore require control of the stochastic convolution present in \eqref{eqn:mildestimate}. Our main tool for doing so is the following.
\begin{theorem}[Gaussian tails of stochastic convolution, see \Cref{app:mild}]\label{thm:gaussian}
There exists a constant $K>0$ such that the following holds true. Suppose that $\big\{S(t)\big\}_{t\geq 0}$ is a $C_0$-semigroup on a Hilbert space $\mathcal{H}$ satisfying
    \[\sup_{t\geq 0}\big\|S(t)\big\|_{\mathcal{L}(\mathcal{H})}\leq M\]
    for some $M\geq 1$, and  $g\in L^p(\Omega;L^p(0,T;\mathrm{HS}(L_Q^2,\mathcal{H}))$ satisfies
    \[\int_0^T\mathbb{E}\Big[ \big\|g(t)\big\|^p_{\mathrm{HS}(L^2_Q,\mathcal{H})}\Big]\d t\leq B^pT, \quad p>2\]
    for some $B\geq 0$. Then, for  $\lambda > eBKM\sqrt{ T}$ we have the inequality
    \[\mathbb{P}\Big[ \sup_{t\in[0,T]}\big\|\int_0^t S(t-s)g(s)\d W^Q_s \big\|_{\mathcal{H}} \geq\lambda \Big] \leq e^{-(eBKM)^{-2}\lambda^2/T}. \]
\end{theorem}

In the sequel, we will apply \Cref{thm:gaussian} to stochastic convolutions with respect to the semigroup $\{e^{\mathcal{L}_{c(T)}t}\}_{t\geq0}$, as well as ordinary stochastic integrals. In the latter case, we take the semigroup in \Cref{thm:gaussian} to be the trivial semigroup, i.e. the identity operator. Below, we explicitly compute the norms of various Hilbert-Schmidt operators used in the sequel.
\begin{lemma}\label{lem:HS}
   Assuming \hyperlink{hyp:initial}{S1} and \hyperlink{hyp:weight}{S2}, there exists a constant $C_{10}$ such that for all $\tilde{\xi}\in\R$, $g\in L^2$ and $h\in H^1_w$, we have
   \begin{align*}
       \|hT_{\tilde{\xi}}\cdot\|_{\mathrm{HS}(L^2_Q,H^1_w)}\leq& C_{10}\|h\|_{H^1_w},\\
        \big\|\langle g,T_{\tilde{\xi}}\cdot\rangle h\big\|_{\mathrm{HS}(L^2_Q,H^1_w)}=& \|Q^{1/2}g\|_{L^2}\|h\|_{H^1_w},\\
       \big\|\langle g,T_{\tilde{\xi}} \cdot\rangle\big\|_{\mathrm{HS}(L_Q^2,\R)}=&\|Q^{1/2}g\|_{L^2},
       \end{align*}
       while for all $g\in L^1$ we have
       \begin{align*}
            \big\|\langle g,T_{\tilde{\xi}} \cdot\rangle\big\|_{\mathrm{HS}(L_Q^2,\R)}\leq& C_{10}\|g\|_{L^1}.
   \end{align*}
\end{lemma}
\begin{proof}
We compute that pointwise multiplication with a function $h\in H^1_w$ leads to the identity
\begin{align*}
  \|hT_{\tilde{\xi}}\cdot\|^2_{\mathrm{HS}(L^2_Q,H^1_w)}=&\sum_{k=0}^\infty \|hQ^{1/2}e_k\|^2_{H^1_w}  \\
  =&\sum_{k=0}^\infty \int_{\R} e^{2wx}\big(h^2(x)+h_x^2(x)\big)\big\langle q_{1/2}(x-\cdot),e_k\big\rangle^2+e^{2wx}h^2(x)\big\langle q'_{1/2}(x-\cdot),e_k\big\rangle^2 \d x  \\
  =&\| q_{1/2}\|_{L^2}^2\|h\|_{H^1_w}^2+\| q'_{1/2}\|_{L^2}^2\|h\|_{L^2_w}^2.
\end{align*}
Inner products against a function $g\in L^2$ lead to
\begin{align*}
     \big\|\langle g,T_{\tilde{\xi}} \cdot\rangle\big\|^2_{\mathrm{HS}(L_Q^2,\R)}=&\sum_{k=0}^\infty \big|\langle g,T_{\tilde{\xi}} Q^{1/2}e_k\rangle\big|^2=\|Q^{1/2}g\|_{L^2}^2,
\end{align*}
and hence also
\begin{align*}
    \big\|\langle g,T_{\tilde{\xi}}\cdot\rangle h\big\|^2_{\mathrm{HS}(L^2_Q,H^1_w)}=&\sum_{k=0}^\infty \big\|\langle g,Q^{1/2}e_k\rangle h\big\|^2_{H^1_w}= \big\|\langle g,T_{\tilde{\xi}} \cdot\rangle\big\|^2_{\mathrm{HS}(L_Q^2,\R)}\|h\|^2_{H^1_w}=\|Q^{1/2}g\|_{L^2}^2\|h\|^2_{H^1_w}.
\end{align*}
For $g\in L^1$:
\begin{align*}
     \big\|\langle g,T_{\tilde{\xi}} \cdot\rangle\big\|^2_{\mathrm{HS}(L_Q^2,\R)}=\|Q^{1/2}g\|_{L^2}^2\leq\|\sqrt{\hat{q}}\|_{L^2}^2\|\hat{g}\|_\infty^2=\|\hat{q}\|_{L^1}\|g\|^2_{L^1}.
\end{align*}
Here we note that $\hat{q}\in L^1$, since $q$ is assumed to be an element of $H^1$ in \hyperlink{hyp:initial}{S1} and
\begin{align*}
    \|\hat{q}\|_{L^1}^2\leq \int_{\R}(1+|\omega|^2)|\hat{q}(\omega)|^2\d \omega\times \int_{\R}\frac{1}{1+|\omega|^2}\d \omega\leq \tilde{C}_1\|q\|_{H^1}^2.
\end{align*}
\end{proof}

With these preliminaries in place,
control on the stochastic integral in \eqref{eqn:mildestimate} is provided by the following lemma.
\begin{lemma}\label{lem:zcontrol}
Assuming \hyperlink{hyp:initial}{S1} and \hyperlink{hyp:weight}{S2}, there exists a constant $C_{11}>0$ so that for each $T,s\geq 0$ and $\tilde{\xi} \in \R$, the bounds
\[c(T),c^T(s)\in [\tfrac{1}{2}c_{\mathrm{min}},2c_{\mathrm{max}}]\quad \text{and}\quad R^T_w(s)\leq \delta_5\]
imply
    \[\Big\|Z^{T}\big({\mathbf{m}}^T(s),s\big)T_{\tilde{\xi}}\Big\|_{\mathrm{HS}(L^2_Q,H^1_w)}\leq C_{11}\Big(1+\big\|{v}^T(s)\big\|_{H^1_w}\Big).\]
\end{lemma}
\begin{proof}
From \eqref{eqn:Zt}, a straightforward application of the triangle inequality yields the $\mathbb{P}$-a.s. bound
\begin{align*}\big\|Z^{T}({\mathbf{m}}^T(s),s)T_{\tilde{\xi}}\big\|_{\mathrm{HS}(L^2_Q,H^1_w)}\leq&\big\|(\Phi^T({\mathbf{m}}^T(s),s)
+{v}^T(s))T_{\tilde{\xi}}\cdot\big\|_{\mathrm{HS}(L^2_Q,H^1_w)}\\
&+\Big\|\big\langle c_s^{T}({\mathbf{m}}^T(s),s),T_{\tilde{\xi}}\cdot\big\rangle  \partial_c\Phi^T({\mathbf{m}}^T(s),s)\Big\|_{\mathrm{HS}(L^2_Q,H^1_w)}\\
&+\Big\|\big\langle \Omega_s^{T}({\mathbf{m}}^T(s),s),T_{\tilde{\xi}}\cdot\big\rangle \partial_x\Phi^T({\mathbf{m}}^T(s),s)\Big\|_{\mathrm{HS}(L^2_Q,H^1_w)}.
\end{align*}
Applying \Cref{lem:HS} yields
\[\big\|Z^{T}({\mathbf{m}}^T(s),s)T_{\tilde{\xi}}\big\|_{\mathrm{HS}(L^2_Q,H^1_w)}\leq \tilde{C}_4\sigma\Big(1+\big\|Q^{1/2}c_s^T({\mathbf{m}}^T(s),s)\big\|_{L^2}+\big\|Q^{1/2}\Omega_s^T({\mathbf{m}}^T(s),s)\big\|_{L^2}+\|{v}^T(s)\|_{H^1_w}\Big) \]
and applying \Cref{lem:locmodcontrol} provides the result.
\end{proof}
Having established control on $v^T$ via \Cref{lem:Meta} and \Cref{lem:zcontrol}, we are ready to prove the main result of this section: \Cref{prop:remainder}.
\begin{proof}[Proof of \Cref{prop:remainder}]
Let us fix  $C_9=18C_7\Delta T$. Writing 
\[\tau=\min\big\{\tau^T_{\mathrm{st}}(\eta),\tau^T_{\mathrm{en}}(\delta_*),\tau^T_{\mathrm{mod}}(\eta),\tau^T_c,\Delta T\big\},\]
we may establish control of $\|v^T(\tau)\|_{H^1_w}$ by applying \Cref{lem:Meta} with $\delta=1/3$:
\begin{align*}
     \big\|v^T(\tau)\big\|_{H^1_w}\leq & M \big\|v(T)\big\|_{H^1_w}+\eta/3+\sigma \sup_{0\leq s\leq \Delta T}I(s),
\end{align*}
where we have abbreviated
\[I(s)=\Big\|\int_0^s e^{\mathcal{L}_{c(T)}(s-s')}1_{[0,\tau]}(s')Z^{T}\big(\mathbf{m}^T(s'),s'\big)\ T_{\xi(T)+c(T)s'}\d W^Q_{T+s'}\Big\|_{H^1_w}.\]
Suppose now that $\tau^T_{\mathrm{st}}(\eta)\leq \Delta T$, meaning that the stopping time $\tau^T_{\mathrm{st}}$ is activated because the bound $\eta$ is reached on $[0,\Delta T]$, while also 
\[\tau^T_{\mathrm{st}}(\eta) \leq \min\big\{\tau^T_{\mathrm{en}}(\delta_*),\tau^T_{\mathrm{mod}}(\eta),\tau^T_c\big\}.\] 
Then,
\begin{align*}
\eta=&\big\|v^T\big(\tau^T_{\mathrm{st}}(\eta)\big)\big\|_{H^1_w}=\big\|v^T(\tau)\big\|_{H^1_w}
\leq  M\big\|v(T)\big\|_{H^1_w}+\eta/3+\sigma\sup_{0\leq s\leq \Delta T}I(s).
\end{align*}
It follows that the event $\mathcal{E}_1\cap \{\|v(T)\|_{H^1_w}\leq \tfrac{\eta}{3M}\}$ can only happen if
\[\sigma\sup_{0\leq s\leq \Delta T}I(s)\geq \eta/3,\] 
so that
\begin{align*}
    \mathbb{P}\Big[\mathcal{E}_1\cap \big\{\|v(T)\|_{H^1_w}\leq \tfrac{\eta}{3M} \big\}\Big]\leq \mathbb{P}\Big[\sigma\sup_{0\leq s\leq \Delta T}I(s)\geq \eta/3\Big].
\end{align*}
To control this probability, note that \Cref{lem:zcontrol} implies that $\mathbb{P}$-almost surely
\[ 1_{[0,\tau]}(s')\Big\|Z^{T}\big(\mathbf{m}^T(s'),s'\big)T_{\xi(T)+c(T)s'}\Big\|_{\mathrm{HS}(L^2_Q,H^1_w)}\leq C_{11}(1+\eta), \quad s'\in [0,\Delta T].\]
By applying \Cref{thm:gaussian} with the semigroup $\{e^{\mathcal{L}_{c(T)t}}\}_{t\geq 0}$ restricted to its stable subspace, and by increasing $C_9$ to meet $C_9\geq 2eC_7K\sqrt{\Delta T}$, we find
\[\mathbb{P}\Big[\sigma \sup_{0\leq s\leq \Delta T}I(s)\geq \eta/3\Big]\leq e^{-\delta_9\eta^2/\sigma^2},\]
with 
\[\delta_9=\big(2e C_7M^{-1}K\big)^{-2}/\Delta T.\] 

It remains to establish the same bound on $\mathcal{E}_2$. To this end, suppose that 
\[\tau=\Delta T\quad \text{and}\quad  \|v(T)\|_{H^1_w}\leq \tfrac{\eta}{3M},\quad \text{but}\quad \big\|v^T(\Delta T)\big\|_{H^1_w}\geq \tfrac{\eta}{9M}.\] 
Upon increasing $C_9$ to ensure that \Cref{lem:Meta} may be applied on $[0,\Delta T]$ with $\delta=1/36M$, we obtain
\[\tfrac{\eta}{9M} \leq \big\|v^T(\Delta T)\big\|_{H^1_w}\leq \tfrac{1}{6}\big\|v(T)\big\|_{H^1_w}+\tfrac{\eta}{36M}+\sigma I(\Delta T),\]
where \hyperlink{hyp:constants}{C1} produces the factor $\tfrac{1}{6}$ above. The conclusion follows, upon decreasing $\delta_9$ to
\[\delta_9=\big(72e C_7K\big)^{-2}/\Delta T,\] 
via the tail bound 
\begin{align*}
    \mathbb{P}\Big[\mathcal{E}_2\cap \big\|v(T)\big\|_{H^1_w}\leq \tfrac{\eta}{3M}\Big]\leq \mathbb{P}\Big[\sigma I(\Delta T)\geq \tfrac{\eta}{36M}\Big]\leq e^{-\delta_9 \eta^2/\sigma^2}.
\end{align*}
\end{proof}

\section{Local control of modulation parameters}
\label{sec:parameters}
In this section, we establish several facts regarding the modulation parameters $c,\xi$ and their local counterparts. 
Our goal is to address one of the conditions for stability formulated in \Cref{sec:weighted}: an estimate on the local modulation parameters, encoded by the stopping time \[\tau^T_{\mathrm{mod}}(\eta)=\tau^{T}_{\mathrm{amp},1}(\eta)\wedge\tau^{T}_{\mathrm{amp},2}(\eta)\wedge\tau^{T}_{\mathrm{pos,1}}(\eta)\wedge \tau^{T}_{\mathrm{pos,2}}(\eta)\] 
introduced in \Cref{sec:weighted}, where
\begin{align*}
   \tau^{T}_{\mathrm{amp,1}}(\eta)
   =&\sup\big\{s\geq 0:\big|c(T)-c^T(s)\big|\leq \delta_*\eta\}; \\
   \tau^{T}_{\mathrm{amp,2}}(\eta)
   =&\sup\big\{s\geq 0:\big|c(T+s)-c^T(s)\big|\leq \delta_*\eta\}; \\
  \tau^{T}_{\mathrm{pos,1}}(\eta)
   =&\sup\big\{s\geq 0:\big|\xi(T)+c(T)s-\xi^T(s)\big|\leq 2\Delta T\delta_*\eta  \big\};\\
   \tau^{T}_{\mathrm{pos,2}}(\eta)
   =&\sup\big\{s\geq 0:\big|\xi(T+s)-\xi^T(s)\big|\leq 2\Delta T\delta_*\eta  \big\}.
\end{align*}
We show that the probability that one of the stopping times above is activated on $[0,\Delta T]$, while the local perturbation $v^T$ is small and the global amplitude is within the bounds $[c_{\mathrm{min}},c_{\mathrm{max}}]$, satisfies an exponential tail estimate. Recall from \Cref{sec:weighted} that
\begin{align*}
\tau^T_c=&\sup\big\{s\geq 0:c(T+s)\in[\tfrac{1}{2}c_{\mathrm{min}},2c_{\mathrm{max}}]  \big\},\\
\tau^T_{\mathrm{st}}(\eta)=&\sup\big\{s\geq 0:\big\|v^T(s)\big\|_{H^1_w}\leq \eta  \big\},
\end{align*}
and let us furthermore introduce the stopping time
\[t_c=\sup\{t\geq 0:c(t) \in [c_{\mathrm{min}},c_{\mathrm{max}}]\},\]
which signals that $c(t)$ exits its bounds $[c_{\mathrm{min}},c_{\mathrm{max}}]$. 
\begin{proposition}[Control of modulation parameters]\label{prop:controlofmod}
    Assuming \hyperlink{hyp:initial}{S1}, \hyperlink{hyp:weight}{S2} and \hyperlink{hyp:constants}{C1}, there exist constants $\delta_{12},C_{12}> 0$ such that the following holds true. For all $\eta\in[0,\eta_0]$, all $C_{12}\sigma,C_{12}\epsilon\in [0,\eta]$ and $T\geq 0$ the stopping times $\tau^T_{\mathrm{mod}}, \tau^T_{\mathrm{st}}$ and $t_c$ satisfy
    \[\mathbb{P}\big[\{\tau^T_{\mathrm{mod}}(\eta) \leq \tau^T_{\mathrm{st}}(\eta)\wedge \Delta T\} \cap \{T+\tau_{\mathrm{mod}}^T(\eta)  \leq t_c \}\big]\leq e^{-\delta_{12}\eta^2/\sigma^2}.\] 
\end{proposition}

We first treat the stopping time $\tau^{T}_{\mathrm{amp,1}}$ by estimating the local amplitude $c^T$ through
    \begin{align}\label{eqn:cdiff}
        \big|c(T)-c^T(s)\big|\leq & \int_0^s \Big|c_d^{\sigma,\epsilon,T}\big(\mathbf{m}^T(s'),s'\big)\Big| \ \d s'+\sigma \Big|\int_0^s\big\langle c^{T}_s\big(\mathbf{m}^T(s'),s'\big),T_{\xi(T)+c(T)s'}\d W_{T+s'}^Q\big\rangle\Big|.
    \end{align}
This is obtained straightforwardly from \eqref{eqn:localc}. We recall that $\eta_0$ below is introduced in \Cref{prop:remainder}.

\begin{lemma}\label{lem:controloflocalamp}
   Assuming \hyperlink{hyp:initial}{S1}, \hyperlink{hyp:weight}{S2} and \hyperlink{hyp:constants}{C1}, there exist constants $\delta_{13},C_{13}> 0$ such that the following holds true. For all $\eta\in[0,\eta_0]$, all $C_{13}\sigma,C_{13}\epsilon\in [0,\eta]$ and $T\geq 0$ the stopping times $\tau^{T}_{\mathrm{amp,1}}, \tau^{T}_{\mathrm{pos,1}},  \tau^T_{\mathrm{st}}$ and $\tau^{T}_c$ satisfy
    \[\mathbb{P}\big[\tau^{T}_{\mathrm{amp,1}}(\eta) \leq \tau^{T}_{\mathrm{pos,1}}( \eta)\wedge\tau^T_{\mathrm{st}}(\eta)\wedge\tau^{T}_c\wedge\Delta T\big]\leq e^{-\delta_{13}\eta^2/\sigma^2}.\]
\end{lemma}

\begin{proof}
 Writing $\tau=\min\{\tau^{T}_{\mathrm{amp,1}}({\eta}),\tau^{T}_{\mathrm{pos,1}}({\eta}), \tau^T_{\mathrm{st}}(\eta), \tau^{T}_c\}$, we apply \Cref{lem:locmodcontrol} to \eqref{eqn:cdiff} and find
    \begin{align*}
 \big|c(T)-c^T(s)\big|
        \leq  &C_5 \int_0^s  R_w^T(s')\|v^T(s')\|_{L^2_w}\d s'\\
        &+C_5 \epsilon \int_0^sf(T+s')\big(1+\|v^T(s')\|_{L^2_w}\big)\d s' +C_5 \sigma^2\int_0^s \big(1+\|v^T(s')\|_{L^2_w}\big)\d s'\\
        &+\sigma \Big|\int_0^s\big\langle c^{T}_s\big(\mathbf{m}^T(s'),s'\big),T_{\xi(T)+c(T)s'}\d W_{T+s'}^Q\big\rangle\Big|,
    \end{align*}
    $\mathbb{P}$-almost surely for $s'\in [0,\tau]$. We thus have
    \begin{align*}
 \big|c(T)-c^T(\tau)\big|
        \leq  &C_5 \Delta T\Big(  \delta_*(2+2\Delta T )\eta^2+2(\epsilon + \sigma^2) \Big)+\sigma\sup_{0\leq s \leq \Delta T} C(s),
    \end{align*}
    where we have abbreviated
    \[C(s):=\Big|\int_0^s1_{[0,\tau]}(s')\big\langle c^{T}_s\big(\mathbf{m}^T(s'),s'\big),T_{\xi(T)+c(T)s'}\d W_{T+s'}^Q\big\rangle\Big|.\]
We note via \Cref{lem:HS} that $\mathbb{P}$-almost surely for $s'\in [0,\tau]$, the integrand above satisfies
\begin{align*}
    \Big\|\big\langle c^{T}_s\big(\mathbf{m}^T(s'),s'\big),T_{\xi(T)+c(T)s'}\cdot\big\rangle \Big\|^2_{\mathrm{HS}(L^2_Q,\R)}
    \leq &C_5^2\big(1+\|v^T(s')\|_{L^2_w}\big)^2\leq C_5^2(1+\eta)^2.
\end{align*}
  Suppose now that $\tau^{T}_{\mathrm{amp,1}}(\eta)\leq\tau^{T}_{\mathrm{pos,1}}(\eta)\wedge \tau^T_{\mathrm{st}}(\eta)\wedge \tau^T_c$ and that the stopping time $\tau^{T}_{\mathrm{amp,1}}$ is activated because the bound $\delta_*\eta$ is reached on $[0,\Delta T]$. In this case, 
\begin{align*}
    \delta_*\eta =&\big|c(T)-c^T(\tau)\big|
    \leq \delta_*\eta/4+2C_5 \Delta T (\epsilon +\sigma^2)+\sigma \sup_{0\leq s \leq \Delta T} C(s),
\end{align*}
where we have used that $C_5 \Delta T \delta_*(2+2\Delta T)\eta^2\leq \delta_*\eta/4$ via \hyperlink{hyp:constants}{C1}.
Ensuring that also $2C_5 \Delta T (\epsilon+\sigma^2)\leq \delta_*\eta/4$ via $C_{13}$, this can only happen if 
\[\sigma \sup_{0\leq s \leq \Delta T} C(s)\geq \delta_*\eta/2.\]
Ensuring furthermore that $\sigma \sqrt{\Delta T}\leq \frac{\delta_*\eta}{2eC_5(1+\eta_0)K}$, \Cref{thm:gaussian} implies the tail bound
\begin{align}\label{eqn:tailbound}
    \mathbb{P}\Big[\sigma \sup_{0\leq s\leq \Delta T}C(s)\geq \delta_*/2\Big]\leq e^{-\delta_{13}\eta^2/\sigma^2},
\end{align}
with
\[\delta_{13}=\big(2eC_5(1+\eta_0)K\big)^{-2}\delta_*^2/\Delta T.\]
Hence, 
\[\mathbb{P}\big[\tau^{T}_{\mathrm{amp,1}}(\eta) \leq \tau^{T}_{\mathrm{pos,1}}( \eta)\wedge\tau^T_{\mathrm{st}}(\eta)\wedge\tau^{T}_c\wedge\Delta T\big]\leq e^{-\delta_{13}\eta^2/\sigma^2 }.\]

\end{proof}
Next, we set out to control the stopping time $\tau^{T}_{\mathrm{pos},1}$ via the estimate
\begin{align}\label{eqn:localxicontrol}
        \big|\xi(T)+c(T)s-\xi^T(s)\big|\leq& \int_0^s \big|c^T(s')-c(T)\big|+\big|\Omega_d^{\sigma,\epsilon,T}\big(\mathbf{m}^T(s'),s'\big)\big|\d s' \\
        &+\sigma\Big|\int_0^s\big\langle \Omega^{T}_s\big(\mathbf{m}^T(s'),s'\big),T_{\xi(T)+c(T)s'}\d W_{T+s'}^Q\big\rangle \Big|\nonumber
    \end{align}
on the local soliton position $\xi^T$. Note here the dependence on $|c^T(s')-c(T)\big|$, which is under control before time $\tau^{T}_{\mathrm{amp},1}$. 
\begin{lemma}\label{lem:controloflocalpos}
    Assuming \hyperlink{hyp:initial}{S1}, \hyperlink{hyp:weight}{S2} and \hyperlink{hyp:constants}{C1}, there exist constants $\delta_{14},C_{14}>0$ such that the following holds true. For all $\eta\in[0,\eta_0]$, all $C_{14}\sigma,C_{14}\epsilon\in [0,\eta]$ and $T\geq 0$ the stopping times $\tau^{T}_{\mathrm{amp,2}}, \tau^{T}_{\mathrm{pos,2}}, t_{\mathrm{st}}$ and $t_c$ satisfy
   \[\mathbb{P}\big[\tau^{T}_{\mathrm{pos,1}}(\eta) \leq \tau^{T}_{\mathrm{amp,1}}(\eta)\wedge\tau^T_{\mathrm{st}}(\eta)\wedge\tau^{T}_c\wedge\Delta T\big]\leq e^{-\delta_{14}\eta^2/\sigma^2}.\]
\end{lemma}
\begin{proof}
 Let us once more write  $\tau=\min\{\tau^{T}_{\mathrm{amp,1}}({\eta}),\tau^{T}_{\mathrm{pos,1}}({\eta}), \tau^T_{\mathrm{st}}(\eta), \tau^{T}_c\}$. From \eqref{eqn:localxicontrol} we obtain the inequality 
\begin{align*}
    2\Delta T\delta_*\eta =&  \big|\xi(T)+c(T)\tau-\xi^T(\tau)\big|  \\
    \leq&\delta_*\Delta T\eta+C_5 \Delta T \big(\delta_*(2+2\Delta T )\eta^2+\epsilon (1+\eta) +\sigma^2 (1+\eta)\big)\\
    &+\sigma \sup_{0\leq s \leq \Delta T} \Big|\int_0^s1_{[0,\tau]}(s')\big\langle \Omega^{T}_s\big(\mathbf{m}^T(s'),s'\big),T_{\xi(T)+c(T)s'}\d W_{T+s'}^Q\big\rangle\Big|.
\end{align*}
Ensuring $C_5 \big(\delta_*(2+2\Delta T )\eta^2+\epsilon (1+\eta) +\sigma^2 (1+\eta)\big)\leq \delta_*\eta/2$ as well as $\sigma \leq \frac{\sqrt{\Delta T}\delta_*\eta}{2eC_5(1+\eta_0)K}$ via $C_{14}$, the result follows from the tail bound 
\[\mathbb{P}\Big[\sigma \sup_{0\leq s \leq \Delta T} \Big|\int_0^s1_{[0,\tau]}(s')\big\langle \Omega^{T}_s\big(\mathbf{m}^T(s'),s'\big),T_{\xi(T)+c(T)s'}\d W_{T+s'}^Q\big\rangle\Big|\geq \Delta T\eta/2\Big]\leq e^{-\delta_{14} \eta^2/\sigma^2},\]
where
\[\delta_{14}=\Delta T\big(2eC_5(1+\eta_0)K\big)^{-2}\delta_*^2,\]
analogous to \eqref{eqn:tailbound}.
\end{proof}

Control of the global amplitude is established by estimating \eqref{eqn:postulate1} as
 \begin{align*}
        \big|c(T)-c(T+s)\big|\leq & \int_T^{T+s} \Big|c_d^{\sigma,\epsilon}\big(v(t'),c(t'),t'\big)\Big| \ \d t'+\sigma \Big|\int_T^{T+s}\big\langle c_s\big(v(t'),c(t')\big),T_{\xi(t')}\d W_{t'}^Q\big\rangle\Big|.
    \end{align*}
    The difference between the local and global positions is controlled via
\[|\xi(T+s)-\xi^T(s)|\leq \big|\xi(T)+c(T)s-\xi^T(s)\big|+\big|\xi(T)+c(T)s-\xi(T+s)\big|,\]
where
\begin{align*}
    \big|\xi(T)+c(T)s-\xi(T+s)\big|\leq& 
    \int_{T}^{T+s} |c(T)-c(t')|+\big|\Omega_d^{\sigma,\epsilon}(v(t'),c(t'),t'\big)\big| \d t'\\
    &+\sigma \Big|\int_{T}^{T+s}\langle \Omega_s(v(t'),c(t')\big),T_{\xi(t')} \d W_{t'}^Q\rangle \Big|.
\end{align*}
This leads to the following estimates on $\tau^T_{\mathrm{amp},2}$ and  $\tau^T_{\mathrm{pos},2}$.
\begin{lemma}\label{lem:controlofglobal} Assuming \hyperlink{hyp:initial}{S1}, \hyperlink{hyp:weight}{S2} and \hyperlink{hyp:constants}{C1}, there exist constants $\delta_{15},C_{15}> 0$ such that the following holds true. For all $\eta\in[0,\eta_0]$, all $C_{15}\sigma,C_{15}\epsilon\in [0,\eta]$ and $T\geq 0$ the stopping times $\tau^{T}_{\mathrm{amp,2}},\tau^{T}_{\mathrm{pos,2}}, t_{\mathrm{st}}$ and $t_c$ satisfy
  \[\mathbb{P}\big[\{\tau^{T}_{\mathrm{amp,2}}({\eta}) \leq \tau^{T}_{\mathrm{pos,2}}(\eta)\wedge\Delta T\}\cap \{T+\tau^{T}_{\mathrm{amp,2}}({\eta})\leq t_{\mathrm{st}}(2\eta)\wedge t_c\}\big]\leq e^{-\delta_{15}\eta^2/\sigma^2},\]
  and
\[\mathbb{P}\big[\{ \tau^{T}_{\mathrm{pos,2}}(\eta)\leq \tau^{T}_{\mathrm{amp,2}}({\eta})\wedge\Delta T\}\cap \{T+\tau^{T}_{\mathrm{pos,2}}(\eta)\leq t_{\mathrm{st}}(2\eta)\wedge t_c\}\big]\leq e^{-\delta_{15}\eta^2/\sigma^2}.\]
\end{lemma}
\begin{proof}
    These bounds follow from computations fully analogous to those in the proofs of \Cref{lem:controloflocalamp} and \Cref{lem:controloflocalpos}.
\end{proof}
As a last preparation, we show how the correspondence between the local perturbation $v(T+s)$ and global perturbation $v^T(s)$ manifests in the stopping times $\tau_{\mathrm{st}}^T$ and its global counterpart $t_{\mathrm{st}}$. We recall that
\[t_{\mathrm{st}}(\eta)=\sup\big\{t\geq 0:\|v(t)\|_{H^1_w} \leq\eta\big\}.\]
\begin{lemma}\label{lem:monotonicity}
     Assume \hyperlink{hyp:initial}{S1}, \hyperlink{hyp:weight}{S2} and \hyperlink{hyp:constants}{C1}. For each $T\geq 0$ and $\eta\in[0,\eta_0]$, the stopping times $\tau^T_{\mathrm{st}}, \tau^{T}_{\mathrm{mod}}, t_c$ and $t_\mathrm{st}$ satisfy 
    \[\min\{T+\tau^T_{\mathrm{st}}(\eta),T+\tau^{T}_{\mathrm{mod}}(\delta_*\eta),t_c\}\leq t_\mathrm{st}(2\eta),\]
    $\mathbb{P}$-almost surely.
\end{lemma}
\begin{proof}
Writing  $\tau=\min\{\tau^T_{\mathrm{st}}(\eta),\tau^{T}_{\mathrm{mod}}(\eta),t_c-T\}$, we have
\[c(T+s)\in[c_{\mathrm{min}},  c_{\mathrm{max}}], \quad
\text{and}\quad \big|c^{ T}(s)-c(T+s)\big|\leq2\delta_*\eta, \quad \text{for}\quad s\in [0,\tau].\]
Since $\delta_*\leq \min\{\frac{1}{4}c_{\mathrm{min}},\frac{1}{2}c_{\mathrm{max}}\}$ via \hyperlink{hyp:constants}{C1}, it follows that also 
\[c^{T}(s)\in[\tfrac{1}{2}c_{\mathrm{min}},  2c_{\mathrm{max}}],\quad s\in [0,\tau].\]
\Cref{lem:correspondence} now gives
\[\Big|\big\|v({ T}+s)\big\|_{H^1_w}-e^{w(\xi( T)+c(T)s-\xi(T+s))}\big\|v^{ T}(s)\big\|_{H^1_w}\Big|\leq 3C_4\delta_*\eta.\]
We then obtain the bound 
\[\|v(T+s)\|_{H^1_w}\leq e^{w(\xi( T)+c(T)s-\xi(T+s))}\|v^T(s)\|_{H^1_w}+3C_4\eta\leq e^{2\eta_0}\delta_*\eta+3C_4\delta_*\eta<2\eta,\]
for all $s\in[0,\tau]$, where we have used that $\delta_*<(e^2+3C_4)^{-1}$ via \hyperlink{hyp:constants}{C1}.
\end{proof}
We are then ready to collect our results and establish control of $\tau_{\mathrm{mod}}^T$.
We note that the event 
\[\tau^T_{\mathrm{mod}}(\eta) \leq \tau^T_{\mathrm{st}}(\eta)\wedge \Delta T\quad \text{while}\quad T+\tau_{\mathrm{mod}}^T(\eta)  \leq t_c\]
implies that one of the events
\begin{align*}
    \mathcal{B}_1=& \big\{\tau^{T}_{\mathrm{amp,1}}(\eta) \leq \tau^{T}_{\mathrm{pos,1}}(  \eta)\wedge\tau^{T}_{\mathrm{amp,2}}({\eta})\wedge\tau^T_{\mathrm{st}}(\eta)\wedge\Delta T\big\}\cap \big\{T+\tau^{T}_{\mathrm{amp,1}}(\eta) \leq t_c\big\};\\
    \mathcal{B}_2=& \big\{\tau^{T}_{\mathrm{pos,1}}( \eta) \leq \tau^{T}_{\mathrm{amp,1}}(\eta)\wedge\tau^{T}_{\mathrm{amp,2}}({\eta})\wedge\tau^T_{\mathrm{st}}(\eta)\wedge\Delta T\big\}\cap\big\{ T+ \tau^{T}_{\mathrm{pos,1}}( \eta) \leq t_c\big\};\\
    \mathcal{B}_3=& \big\{\tau^{T}_{\mathrm{amp,2}}({\eta}) \leq \tau^{T}_{\mathrm{mod}}(\eta)\wedge\tau^T_{\mathrm{st}}(\eta) \wedge \Delta T\big\}\cap\big\{ T+\tau^{T}_{\mathrm{amp,2}}({\eta})\leq  t_c\big\}; \\
    \mathcal{B}_4=& \big\{\tau^{T}_{\mathrm{pos,2}}( \eta)\leq \tau^{T}_{\mathrm{mod}}({\eta})\wedge\tau^T_{\mathrm{st}}(\eta)\wedge\Delta T\big\}\cap\big\{ T+\tau^{T}_{\mathrm{pos,2}}( \eta)\leq t_c\big\},
\end{align*}
holds, depending on which of $\tau^T_{\mathrm{amp},1},\tau^T_{\mathrm{amp},2},\tau^T_{\mathrm{pos},1}$ and $\tau^T_{\mathrm{pos},2}$ is smallest.
\begin{proof}[Proof of \Cref{prop:controlofmod}]
We bound the probability of the events $\mathcal{B}_1, \mathcal{B}_2, \mathcal{B}_3$ and $\mathcal{B}_4$. In case $\mathcal{B}_1$ or $ \mathcal{B}_2$ holds, note that $t_c-T\wedge \tau^T_{\mathrm{amp},2}(\eta) \leq \tau^T_c$. Indeed, 
\begin{align*}
    |c^T(s)-c(T+s)|\leq \eta_0 \quad \text{and} \quad c(T+s)\in[c_{\mathrm{min}},c_{\mathrm{max}}]
\end{align*}
implies that $c^T(s)\in[\tfrac{1}{2}c_{\mathrm{min}},2c_{\mathrm{max}}]$. Thus, the probability of $\mathcal{B}_1\cup \mathcal{B}_2$ is controlled by \Cref{lem:controloflocalamp} and \Cref{lem:controloflocalpos}, respectively.

In case $\mathcal{B}_3$ holds, \Cref{lem:monotonicity} implies $T+\tau^{T}_{\mathrm{amp,2}}({\eta})\leq  t_{\mathrm{st}}(2\eta)$. Similarly, the event $\mathcal{B}_4$ implies $T+\tau^{T}_{\mathrm{pos,2}}({\eta})\leq  t_{\mathrm{st}}(2\eta)$. Hence, the probability of $\mathcal{B}_3\cup\mathcal{B}_4$ is bounded via \Cref{lem:controlofglobal}. 
\end{proof}

\section{Global control}
\label{sec:energy}
In this section, we establish long-time control of the global modulation system $\big(v(t),c(t),\xi(t)\big)$ introduced in \Cref{sec:global}. Our first result controls the probability that the soliton amplitude $c(t)$ exits its bounds $[c_{\mathrm{min}},c_{\mathrm{max}}]$.
Secondly, we analyze the growth of the remainder $v(t)$ in the unweighted space $L^2$. Besides being interesting in its own right, this is a crucial ingredient toward controlling the weighted norm of $v$ via \Cref{prop:remainder}. Since the traveling-wave operator $\mathcal{L}_c$ is not exponentially stable on $L^2$, it is standard to analyze the unweighted norm via energy arguments \cite{pegoweinstein, westdorp2}.  

Henceforth, we assume an integrability condition on the forcing term $f$, limiting the total energy contribution of the deterministic forcing in \eqref{eqn:skdv}.
\begin{itemize}
\item[\textbf{C2}]{
  \phantomsection\hypertarget{hyp:decay}{}\textit{The constant $E>0$ satisfies
  \[c_{\mathrm{min}}\leq c_*e^{-3E} \quad  \text{and} \quad c_{\mathrm{max}}\geq c_*e^{3E}.\] 
  The forcing term $f$ lies in $L^1\big([0,\infty)\big)$ and the constant  $\epsilon_0 \in (0,\infty]$ is small enough to ensure
  \begin{align*}\epsilon_0\int_0^\infty|f(t)|\d t\leq E.\end{align*}
 }}
\end{itemize}
With this condition in place, the soliton amplitude $c(t)$ stays within the bounds $[c_{\mathrm{min}},c_{\mathrm{max}}]$ for times that are small with respect to $\sigma^2$. Recall the notation
\begin{align*}t_{c}=&\sup\big\{t\geq 0: c(t)\in[c_{\mathrm{min}},c_{\mathrm{max}}]  \big\},\\
t_{\mathrm{st}}(\eta)=& \sup\big\{t\geq 0:\|v(t)\|_{H_w^1}\leq \eta\big\},
\end{align*}
the global counterparts of $\tau^T_c$ and $\tau^T_{\mathrm{st}}(\eta)$ from \Cref{sec:weighted}. The constant $\delta_{17}$ below will be introduced in \Cref{lem:logc}.
\begin{proposition}
\label{prop:tccontrol}   Assume \hyperlink{hyp:initial}{S1}, \hyperlink{hyp:weight}{S2}, \hyperlink{hyp:constants}{C1} and \hyperlink{hyp:decay}{C2}. For each $\eta\in[0,\eta_0]$, each $\epsilon\in[0,\epsilon_0]$, and $\sigma, T\geq 0$ that satisfy $\sigma^2 T, \eta^2 T\leq \delta_{17} $ we have
    \[\mathbb{P}[t_c \leq  T\wedge t_{\mathrm{st}}(\eta)]\leq e^{-\delta_{17}/(\sigma^2 T)}.\]
\end{proposition}
 The second main result of this section concerns the stopping time
 \[t_{\mathrm{en}}(\eta)=\sup\big\{t\geq 0:\big\|v(t)\big\|_{L^2}\leq \eta  \big\},\]
a global counterpart of $\tau^T_{\mathrm{en}}$ of \Cref{sec:weighted}. The probability that $\|v(t)\|_{L^2}$ remains small on an interval $[0,T]$ is controlled by the forcing parameters $\sigma$ and $\epsilon$.
 \begin{proposition}\label{prop:l2control}
    Assume \hyperlink{hyp:initial}{S1}, \hyperlink{hyp:weight}{S2}, \hyperlink{hyp:constants}{C1} and \hyperlink{hyp:decay}{C2}. There exist constants $C_{16}, \delta_{16}>0$ such that for all $\sigma,T,\lambda>0$ satisfying $\sigma^2 T \leq \delta_{16}$, each $\eta\in[0,\eta_0]$ and each $\epsilon\in[0,\epsilon_0]$, we have
     \[\mathbb{P}\Big[t_{\mathrm{en}}(\lambda)\leq T\wedge t_c \wedge t_{\mathrm{st}}(\eta)\Big]\leq C_{16}\lambda^{-1}\Big(T\big( \sigma^2+\epsilon\eta+\eta^2\big)+\sqrt{T}\sigma  \eta\Big).\]
\end{proposition}
We first establish \Cref{prop:tccontrol}, making use of the evolution equation %for $\log(c(t))$:
\begin{align}\d \log(c(t))=&\Big(\frac{c_d^{\sigma,\epsilon}(v(t),c(t),t)}{c(t)}-\sigma^2\frac{\big\|Q^{1/2}c_s(v(t),c(t))\big\|_{L^2}^2}{2c^2(t)} \ \Big) \ \d t+\sigma \frac{\langle c_s(v(t),c(t)),T_{\xi}\d W_t^Q\rangle }{c(t)},\label{eqn:logc}
\end{align}
which one finds by applying the It\^o lemma \cite[Theorem 4.32]{daprato} to \eqref{eqn:postulate1}. In particular, as an intermediate result, we control the stopping time 
\[t_{\mathrm{log}}=\sup\{t\geq 0 :|\log(c(t)/c_*)-\frac{4}{3}\epsilon \int_0^t|f(t')|\d t'|\leq E\}.\]
It is then ensured that $c(t)$ remains within the limits $[c_{\mathrm{min}},c_{\mathrm{max}}]$ before $t_{\log}$, and consequently $t_c\geq t_{\log}$.
\begin{lemma}\label{lem:logc}   Assuming \hyperlink{hyp:initial}{S1}, \hyperlink{hyp:weight}{S2}, \hyperlink{hyp:constants}{C1} and \hyperlink{hyp:decay}{C2}, there exists a constant $\delta_{17}>0$ such that the following holds true. For each $\eta\in[0,\eta_0]$, each $\epsilon\in[0,\epsilon_0]$, and $\sigma, T\geq 0$ satisfying $\sigma^2 T, \eta^2 T\leq \delta_{17} $ we have
    \[\mathbb{P}[t_{\mathrm{log}}\leq t_{\mathrm{st}}(\eta)\wedge T]\leq e^{-\delta_{17}/\sigma^2T}.\]
\end{lemma}
\begin{proof}
Using the identity $c^{-1}c_f(0,c)=\frac{4}{3}$, we estimate \eqref{eqn:logc} as
    \begin{align*}
        \Big|\log(c(t)/c_*)-\frac{4}{3}\epsilon \int_0^t|f(t')|\d t'\Big|\leq& \int_0^t\Big|\frac{c_d^0(v(t'),c(t'))}{c(t')}\Big|+\epsilon |f(t')|\Big|\frac{c_f(v(t'),c(t'))-c_f(0,c(t'))}{c(t')}\Big| \ \d t'\\
        &\quad+\sigma^2\int_0^t\Big|\frac{c_d(v(t'),c(t'))}{c(t')}\Big|-\frac{\big\|Q^{1/2}c_s(v(t'),c(t'))\big\|_{L^2}^2}{2c^2(t')} \ \d t'\\
&\quad+\sigma \Big|\int_0^t\frac{\langle c_s(v(t'),c(t')),T_{\xi}\d W_{t'}^Q\rangle }{c(t')}\Big|\\
\leq&C_2t (\eta^2+\sigma^2)+C_2\epsilon \eta \int_0^t |f(t')|\d t'\\
&\quad+\sigma \Big|\int_0^t\frac{\langle c_s(v(t'),c(t')),T_{\xi}\d W_{t'}^Q\rangle }{c(t')}\Big|,
    \end{align*}
for $t\in[0,t_{\log}\wedge t_{\mathrm{st}}(\eta)]$. Condition \hyperlink{hyp:decay}{C2} now gives
\[\epsilon \eta \int_0^t |f(t')|\d t'\leq E\eta.\]
Ensuring $C_2T (\eta^2+\sigma^2)+C_2E \eta \leq L/2,$
the inequality $t_{\mathrm{log}}\leq t_{\mathrm{st}}(\eta)\wedge T$ implies
\begin{align*}
     E=&\Big|\log(c(t_{\mathrm{log}})/c_*)-\frac{4}{3}\int_0^{t_{\mathrm{log}}}|f(t')|\d t'\Big|\\
     \leq & E/2+\sigma \sup_{0\leq t \leq T}\Big|\int_0^t 1_{[0,t_{\mathrm{log}}\wedge t_{\mathrm{st}}(\eta)]}(t')\frac{\langle c_s(v(t'),c(t')),T_{\xi}\d W_{t'}^Q\rangle }{c(t')}\Big|
\end{align*}
and consequently 
\[\sigma \sup_{0\leq t \leq T}\Big|\int_0^t 1_{[0,t_{\mathrm{log}}\wedge t_{\mathrm{st}}(\eta)]}(t')\frac{\langle c_s(v(t'),c(t')),T_{\xi}\d W_{t'}^Q\rangle }{c(t')}\Big|\geq E/2.\]
The probability of this event is bounded by the tail estimate \Cref{thm:gaussian}.
\end{proof}
\begin{proof}[Proof of \Cref{prop:tccontrol}]
    The result now follows as an immediate corollary to \Cref{lem:logc}, in view of the fact that $t_c\geq t_{\log}$.
\end{proof} 
We now turn our attention to the energy result, \Cref{prop:l2control}. Our first result regarding the global modulation system outlined in \Cref{sec:global} is as follows.

\begin{lemma}
Assume  \hyperlink{hyp:initial}{S1} and \hyperlink{hyp:weight}{S2}. For each $t\geq0$, the inequalities
\[\|v(t')\|_{L^2_w}\leq \delta_1 \quad \text{and} \quad c(t')\in[c_{\mathrm{min}},c_{\mathrm{max}}],\quad t' \in[0,t]\] 
imply that
   %  \begin{align*}
   %        \|v(t)\|_{L^2}^{2p}=&p\int_0^t(\sigma^2+2\epsilon f(s))\|v\|_{L^2}^{2p}\d s+\sigma^2p\int_0^t \|v\|_{L^2}^{2p-2}\left(6c^{3/2}- 9 c_d(v,c) c^{1/2} - \tfrac{9}{4} \|Q^{1/2}c_s(v,c)\|^2 c^{-1/2}+\|v\|_{L^2}^2\right) \d s\\
   %        &- 9p\int_0^t \|v\|_{L^2}^{2p-2}c^0_d(v) c^{1/2} \d s -9p\sigma c^{1/2}\int_0^t \|v\|_{L^2}^{2p-2}\langle c_s(v,c)-c_s(0),\d W_s^Q\rangle+2p\sigma \int_0^t\|v\|_{L^2}^{2p-2}\langle 2\phi_cv+v^2,T_{\xi} \d W_s^Q\rangle\\
   %        &+\frac{\sigma^2}{2}p(p-1)\int_0^t\|v\|_{L^2}^{2p-2} \langle c_s(v,c)-c_s(0),\phi_c v+v^2 \rangle_{HS(L^2_Q,\R)} \d s
   % \end{align*}
   \begin{align}
          \big\|v(t)\big\|_{L^2}^2=&\sigma^2\int_0^t\Big(6c^{3/2}- 9 c_d(v,c) c^{1/2} - \tfrac{9}{4} \big\|Q^{1/2}c_s(v,c)\big\|_{L^2}^2 c^{-1/2}+\big\|v\big\|_{L^2}^2\Big) \d t'\nonumber\\
          &- 9\int_0^t c^0_d(v,c) c^{1/2} \d t' +\epsilon\int_0^tf(t')\Big(2\|v\|_{L^2}^2-9\big(c_f(v,c)-c_f(0,c)\big)\Big)\d t'\nonumber\\
   &-9\sigma \int_0^t c^{1/2}\big\langle c_s(v,c)-c_s(0,c),\d W_{t'}^Q\big\rangle+2\sigma \int_0^t\big\langle 2\phi_cv+v^2,T_{\xi} \d W_{t'}^Q\big\rangle\label{eqn:energy}
   \end{align}
  holds $\mathbb{P}$-almost surely.
   %\Big( 81c\|Q^{1/2}(c_s(v,c)-c_s(0))\|_{L^2}^2+2\|Q^{1/2}(\phi_c v+v^2)\|_{L^2}^2-18c^{1/2}\langle Q^{1/2}(c_s(v,c)-c_s(0)),Q^{1/2}(\phi_c v+v^2)\rangle\Big)
   \end{lemma}
\begin{proof}
 
We observe that 
\begin{align}\big\|v(t)\big\|_{L^2}^2=\big\|u(t)\big\|_{L^2}^2-\big\|\phi_{c(t)}\big\|_{L^2}^2\label{eqn:difference}\end{align}
by Pythagoras' theorem and the orthogonality condition \eqref{eqn:ortho}. Via It\^{o}'s lemma \cite[Theorem 1]{mild} applied to the functional $M(u)=\|u\|_{L^2}^2$, we obtain
\begin{align}\label{eqn:du}
    \d \|u\|_{L^2}^2=&\sigma^2 \|u\|^2_{L^2}\ \d t+2\epsilon f(t)\|u\|_{L^2}^2\ \d t+2 \sigma \langle u,u \ \d W_t^Q\rangle\\
    =&\big(\sigma^2 +2\epsilon f(t)\big)\big(6c^{3/2}(t)+\|v\|_{L^2}^2\big)\ \d t+2 \sigma \big\langle \phi_c^2+2\phi_c v+v^2, \ T_{\xi}  \d W_t^Q\big\rangle.\nonumber\end{align}
Here we have used that $M:L^2\to \R$ is twice differentiable with Fr\'echet derivatives
\begin{align*}
    \d M(u)[v]=2\langle u,v\rangle, \quad \d^2 M(u)[v,w]=2\langle v,w\rangle, 
\end{align*}
and that the Airy group $\{e^{-\partial_x^3t}\}_{t\in \R}$ is a $C_0$-group of isometries on $L^2$, in the sense that
\begin{align*}
    M(e^{-\partial_x^3t}u)=\|u\|_{L^2}^2, \quad \d M(e^{-\partial_x^3t}u)[e^{-\partial_x^3t}v]=2\langle u,v\rangle, \quad \d^2 M(e^{-\partial_x^3t}u)[e^{-\partial_x^3t}v,e^{-\partial_x^3t}w]=2\langle v,w\rangle. 
\end{align*}
On the other hand, we can also employ It\^o's lemma to compute
\begin{align}\label{eqn:dphi}
    \d \|\phi_{c(t)}\|_{L^2}^2=\d \big(6c^{3/2}(t)\big) =& \Big( 9c_d^{\sigma,\epsilon}(v,c,t) c^{1/2} + \tfrac{9}{4}\sigma^2 \big\|Q^{1/2}c_s(v,c)\big\|_{L^2}^2 c^{-1/2} \Big)\ \d t \\
    &+ 9\sigma c^{1/2} \big\langle c_s(v,c),T_{\xi}  \d W^Q_t\big\rangle,
    % =&\left( 9c^0_d(v) c^{1/2} + 9\sigma^2 c_d(v,c) c^{1/2} + \tfrac{9}{4} \sigma^2 c_s(v,c)^2 c^{-1/2} \right) \d t + 9\sigma c^{1/2} c_s(v,c) \d \epsilon_t
\end{align}
where we recall that
\[c_d^{\sigma,\epsilon}(v,c,t)=c_d^0(v,c)+\epsilon f(t)c_f(v,c)+\sigma^2 c_d(v,c).\]
% \[=O(\|v\|_{L^2_w}^2)+\sigma^2(\tfrac{92}{15}-\tfrac{8\pi^2}{45})c^{3/2}(t) +O(\sigma^2v)\d t + (12 \sigma c^{3/2}(t)+O(\sigma v))\d \epsilon_t.\]
Note here that $9 c^{1/2} c_s(0,c)=2\phi_c^2$, thus the leading-order diffusion term in $\d \|u\|_{L^2}^2$ equals that in $\d \|\phi_{c(t)}\|_{L^2}^2$. Similarly,
\[c_d^{\sigma,\epsilon}(0,c,t)=\tfrac{4}{3}\epsilon f(t)c+\sigma^2 c_d(0,c)\]
shows that the leading-order $\epsilon$-dependent term in $\d \|u\|_{L^2}^2$ equals that in $\d \|\phi_{c(t)}\|_{L^2}^2$. The result now follows by subtracting \eqref{eqn:dphi} from \eqref{eqn:du}.

% The result then follows by taking a power $p$ via the It\^{o} formula.
\end{proof}  
% Denoting 
% \[ X(s):=9\sigma c^{1/2} \langle c_s(v(s))-c_s(0),\cdot\rangle+2\sigma \int_0^t\langle 2\phi_cv+v^2,T_{\xi} \cdot\rangle,\]
Control on the stochastic integrals in \eqref{eqn:energy} is provided in the following lemma.
\begin{lemma}\label{lem:L2stoch}
Assuming \hyperlink{hyp:initial}{S1} and \hyperlink{hyp:weight}{S2}, there exists a constant $C_{18}>0$ such that for all $\tilde{v}\in L^2_w$ satisfying $\|\tilde{v}\|_{L^2_w}\leq \delta_2$, all $\tilde{\xi}\in\R$ and $ \tilde{c}\in[c_{\mathrm{min}},c_{\mathrm{max}}]$
we have the inequalities
\begin{align*}
 \big\|\langle 2\phi_{\tilde{c}}\tilde{v}+\tilde{v}^2,T_{\tilde{\xi}}\cdot\rangle\big\|_{\mathrm{HS}(L_Q^2,\R)}
 \leq C_{18} \Big(\|\tilde{v}\|_{L_w^2}+\|\tilde{v}\|^2_{L^2}\Big)
\end{align*}
and \begin{align*}
   \big\|\langle c_s(\tilde{v},{\tilde{c}})-c_s(0,{\tilde{c}}),T_{\tilde{\xi}}\cdot\rangle\big\|_{\mathrm{HS}(L_Q^2,\R)} \leq C_{18}\|\tilde{v}\|_{L_w^2}.
\end{align*}
\end{lemma}
\begin{proof}
Applying \Cref{lem:HS}, we have
\begin{align*}
\big\|\langle 2\phi_{\tilde{c}}\tilde{v}+\tilde{v}^2,T_{\tilde{\xi}}\cdot\rangle\big\|_{\mathrm{HS}(L_Q^2,\R)}
 \leq&  2\|Q^{1/2}(\phi_{\tilde{c}}\tilde{v})\|_{L^2}+C_{10}\|\tilde{v}^2\|_{L^1}\\
 \leq & \|q\|^{1/2}_{L^1} \|\phi_{\tilde{c}}\|_{L^2_{-w}}\|\tilde{v}\|_{L^2_w}+C_{10}\|\tilde{v}\|_{L^2}^2,
\end{align*}
and
\begin{align*}
    \big\|\langle c_s(\tilde{v},{\tilde{c}})-c_s(0,{\tilde{c}}),T_{\tilde{\xi}}\cdot\rangle\big\|_{\mathrm{HS}(L_Q^2,\R)}= \big\|Q^{1/2}c_s(\tilde{v},{\tilde{c}})-Q^{1/2}c_s(0,{\tilde{c}})\big\|_{L^2}.
\end{align*}
The result follows by applying \Cref{lem:modcontrol}.
\end{proof}

% \begin{corollary}\label{cor:energy}
% Assuming \hyperlink{hyp:initial}{S1} and \hyperlink{hyp:cbound}{S2}, there exists a constant $C_3$ such that for each $t\geq 0$ the inequalities
%     \begin{align*}
%    \|v(t)\|_{L^2}^2\leq& C_3\sigma^2 t \sup_{0\leq s\leq t}\left(1+\|v(s)\|_{L^2_w}^2 +\|v(s)\|_{L^2}^2\right)+ C_3t \sup_{0\leq s\leq t}\|v(s)\|^2_{L^2_w} +C_3\epsilon t\sup_{0\leq s\leq t}|f(s)|\|v(s)\|_{L^2}^2\\
%    &+\sigma \int_0^t X(s)\d W^Q_{s},
%    \end{align*}
%     and
%     \[ \|X(s)\|^2_{HS(L_Q^2,\R)}\leq C_3\Big( \|v\|_{L^2_w}+ \|v\|_{L^2}^2\Big), \quad s\in[0,t],\]
%    hold $\mathbb{P}$-a.s. on the set
% \[\{\omega\in \Omega: c_{\mathrm{min}}\leq c(s) \leq c_{\mathrm{max}},\quad s\in[0,t]\}.\]
% \end{corollary}
% \begin{corollary}\label{cor:energy}
% Assuming \hyperlink{hyp:initial}{S1}-\hyperlink{hyp:weight}{S3}, there exists a constant $C_3$ such that for each $T\geq 0$ we have
%     \begin{align*}
%    \mathbb{E}\sup_{0\leq t' \leq T\wedge t_{c}}\|v(t')\|_{L^2}^2\leq&\Big( \int_0^{T\wedge t_{c}} \left(\sigma^2+\epsilon\mathbb{E}\|v(t)\|_{L^2_w}+ \mathbb{E}\|v(t)\|_{L^2_w}^2 \right) \d t+\sigma \mathbb{E}\Big[\Big(\int_0^{T\wedge t_{c}} \|v(t)\|^2_{L^2_w}+ \|v(t)\|_{L^2}^4\d t\Big)^{1/2}\Big]\Big)\\
%    &\times C_3\exp\big( C_3\int_0^T \sigma^2+2\epsilon |f(t)|\d t\big).
%    \end{align*}
% \end{corollary}

\begin{proof}[Proof of \Cref{prop:l2control}]
Using \eqref{eqn:energy}, we may $\mathbb{P}$-almost surely estimate 
 \begin{align*}
          \|v(t)\|_{L^2}^2\leq&C_2\int_0^t \big(\sigma^2+2\epsilon |f(t')|\big)\|v(t')\|_{L^2}^{2}\d t'+C_2 t \left(\sigma^2+\epsilon\eta+ \eta^2 \right)\\
   &9\sigma \Big|\int_0^t \big\langle c_s(v,c)-c_s(0,c),\d W_{t'}^Q\big\rangle\Big|+2\sigma \Big|\int_0^t\big\langle 2\phi_cv+v^2,T_{\xi} \d W_{t'}^Q\big\rangle\Big|,
   \end{align*}
for $t\in [0,t_c\wedge t_{\mathrm{st}}(\eta)]$, where we have controlled the modulation parameters in the deterministic integrals via \Cref{lem:modcontrol}. After taking a supremum and expectations,
 \begin{align}\label{eqn:left}
          \mathbb{E}\sup_{0\leq t' \leq T\wedge t_{c}\wedge t_{\mathrm{st}}(\eta)}\|v(t')\|_{L^2}^2\leq&C_2\int_0^T \big(\sigma^2+2\epsilon |f(t)|\big)\mathbb{E}\sup_{0\leq t'\leq t\wedge t_{c}\wedge t_{\mathrm{st}}(\eta)}\|v(t')\|_{L^2}^{2}\d t\\
          &+C_2 T \left(\sigma^2+\epsilon\eta+ \eta^2 \right)+\sigma I(T),\nonumber\end{align}
where we have abbreviated
\begin{align*}
I(T)=&
9 \mathbb{E}\sup_{0\leq t \leq T\wedge t_{c}\wedge t_{\mathrm{st}}(\eta)}\Big|\int_0^t \big\langle c_s(v,c)-c_s(0,c),\d W_{t'}^Q\big\rangle\Big|\\
&+2  \mathbb{E}\sup_{0\leq t \leq T\wedge t_{c}\wedge t_{\mathrm{st}}(\eta)}\Big|\int_0^t\big\langle 2\phi_cv+v^2,T_{\xi} \d W_{t'}^Q\big\rangle\Big|.
   \end{align*}
The Burkholder-Davis-Gundy inequality \cite[proposition 2.1]{BDG} together with \Cref{lem:L2stoch} yields the control
\begin{align*}
   I(T)
   \leq& C_{18} \mathbb{E}\Big[\Big(\int_0^{T\wedge t_{c}\wedge t_{\mathrm{st}}(\eta)} \|v(t)\|^2_{L^2_w}+ \|v(t)\|_{L^2}^4\d t\Big)^{1/2}\Big]\\
   \leq& C_{18}\eta \sqrt{T}+C_{18}\sqrt{T} \mathbb{E}\sup_{0\leq t \leq T\wedge t_{c}\wedge t_{\mathrm{st}}(\eta)} \|v(t)\|_{L^2}^2.
   \end{align*}
If $\delta_{16}$ is small enough to ensure $C_{18}\sigma \sqrt{T}\leq 1/2$, we may bring this last term to the left in \eqref{eqn:left}:
 \begin{align*}
          \tfrac{1}{2}\mathbb{E}\sup_{0\leq t \leq T\wedge t_{c}\wedge t_{\mathrm{st}}(\eta)}\|v(t)\|_{L^2}^2\leq&C_2\int_0^T \big(\sigma^2+2\epsilon |f(t)|\big)\mathbb{E}\sup_{0\leq t'\leq t\wedge t_{c}\wedge t_{\mathrm{st}}(\eta)}\|v(t')\|_{L^2}^{2}\d t\\
          &+C_2 T \left(\sigma^2+\epsilon\eta+ \eta^2 \right)+C_2\sigma\sqrt{T}\eta.\end{align*}
Grönwall's inequality then yields
    \begin{align*}
   \mathbb{E}\sup_{0\leq t \leq T\wedge t_{c}\wedge t_{\mathrm{st}}(\eta)}\|v(t)\|_{L^2}^2\leq& 2C_2e^{C_2\sigma^2 T}e^{2C_2\epsilon\int_0^T|f(t)|\d t}\Big(T \left(\sigma^2+\epsilon\eta+ \eta^2 \right)+\sigma\sqrt{T}\eta\Big).
   \end{align*}
Noting that \hyperlink{hyp:decay}{C2} implies
\[\epsilon\int_0^T|f(t)|\d t\leq E,\]
% We note that
% \[\sigma \mathbb{E}\Big[\Big(\int_0^{T\wedge t_{c}\wedge t_{\mathrm{st}}(\eta)} \|v(t)\|_{L^2}^4\d t\Big)^{1/2}\Big]\leq  \sigma \sqrt{T} \mathbb{E}\sup_{0\leq t \leq T\wedge t_{c}\wedge t_{\mathrm{st}}(\eta)} \|v(t)\|_{L^2}^2,\]
the result now follows via Markov's inequality:
\begin{align*}
    \mathbb{P}\Big[\sup_{0\leq t \leq T\wedge t_{c}\wedge t_{\mathrm{st}}(\eta)}\|v(t)\|_{L^2}^2\geq \lambda\Big]
    \leq &2C_2e^{C_2\delta_{16}+2C_2E}\lambda^{-1}\Big(T\big( \sigma^2+\epsilon\eta+\eta^2\big)+\sigma\sqrt{T} \eta\Big).
\end{align*}
\end{proof}

%     \begin{lemma}
%     Assume \hyperlink{hyp:constants}{C1}
%      and let $C>0$. Then there exist constants $T_{\mathrm{max}}, \delta_{15}$ such that for all $\sigma>0$  and $T\in [0,\sigma^{-2}T_{\mathrm{max}})$ the stopping times satisfy
% \[\mathbb{P}[(\tau_A < \min\{\tau^{C\sigma}_B,t_{c}\}) ]\leq e^{-\delta_{15}/\sigma^2T}.\]
% \end{lemma}
% \begin{proof}
% Letting $\tau=\min\{\tau_A,\tau^{C\sigma}_B\}$
%    we have
%   \begin{align*}\|v(\tau)\|_{L^2}^2&\leq C_3\sigma^2 T(1+C^2+ \|v(s)\|_{L^2}^2+C^2\sigma^2+(1+\sigma^{a-2})\delta_*^2/4)+\sigma \|\int_0^\tau X(s) \d W^Q_s\|_{L^2}.
%     \end{align*}
%     In case $\tau_A<\tau^{C\sigma}_B=T$, we have $\|v(\tau)\|_{L^2}^2=\|v(\tau_A)\|_{L^2}^2=\delta_*^2/4$. In this case, as long as $T\in [0,\delta_*^2/8C_3(1+2C^2+\delta_*^2/4)\sigma^{-2})$, it follows that 
% \[\sigma\|\int_0^\tau X(s) \d W^Q_s\|_{L^2} \geq \delta_*^2/8.\]
% We now show that this has exponentially small probability. Via \Cref{cor:energy}:
% \[  1_{[0,\tau]}(s)\|X(s)\|^2_{HS(L_Q^2,\R)}\leq C_3\Big( C+ \delta_*^2/4\Big), \quad s\in[0,T],\]
% It follows via \Cref{thm:gaussian} that for $T\in [0,(\delta_*^{-2}8eC_3(C+\delta_*^2/4)K)^{-2}\sigma^{-2})$
% \[\mathbb{P}[\sigma\|\int_0^\tau X(s) \d W^Q_s\|_{L^2} \geq \delta_*^2/8]\leq e^{-(\delta_*^{-2}8eC_3(C+\delta_*^2/4)K)^2/\sigma^2T}.\]
% \end{proof}

\section{Nonlinear stability}\label{sec:proofmain}
In this section, we collect our results and prove the stability result \Cref{thm:bigthm}. Our goal here is to show that the event $t_{\mathrm{st}}(\eta)\geq T$ occurs with high probability, by ensuring that 
   \begin{itemize}
       \item[--] the unweighted $L^2$-norm of $v(t)$ remains under control on $[0,T]$ (\Cref{prop:l2control});
       \item[--] the difference between the local and global modulation parameters remains under control on $[0,T]$ (\Cref{prop:controlofmod});
       \item[--] the weighted norm of $v(t)$ remains under control on $[0,T]$ (\Cref{prop:remainder}).
   \end{itemize}
Although our primary interest lies in the latter, our proof requires all of the above to hold. Since the stability results of \Cref{sec:weighted} and \Cref{sec:parameters} hold on intervals of length $\Delta T$, we partition the interval $[0,T]$ into
\[[0,T]\subseteq \bigcup_{n=0}^{\lceil T/\Delta T\rceil-1}[n\Delta T,(n+1)\Delta T),\]
and seek to establish stability on each intermediate interval $[n\Delta T,(n+1)\Delta T)$.
% contains the event  
% \[\sup_{0\leq t \leq \lceil\frac{T_{\mathrm{max}}}{\epsilon \Delta T}\rceil\Delta T\wedge \tau}\|v(t)\|_{H^1_w} \leq 3M\eta.\]
   Consider, therefore, for each $n\in\N_0$ the stability event
    $\mathcal{S}_n(\eta)\subseteq \Omega$, defined as 
    the set of $\omega\in \Omega$ for which:
    \begin{enumerate}[(i)]
    \item the local stopping times at $n\Delta T$ satisfy
    \[\tau^{n\Delta T}_{\mathrm{mod}}(\eta),\tau^{n\Delta T}_{\mathrm{en}}(\delta_*),\tau^{n\Delta T}_{\mathrm{st}}(\eta)\geq \Delta T;\]
    \item at the end point $(n+1)\Delta T$ we have
    \begin{align*}
    \|v^{n\Delta T}(\Delta T)\|_{H^1_w}\leq \tfrac{\eta}{9M},\quad \|v((n+1)\Delta T)\|_{H^1_w}\leq \tfrac{\eta}{3M}.
    \end{align*}
    \end{enumerate}
    
    The following result allows us to establish high probability of the local stability events %$\mathcal{S}_0(\eta),\mathcal{S}_1(\eta),\ldots $ 
    in a recursive manner: we bound the probability that $\mathcal{S}_{n}(\eta)$ fails to hold, provided  $\mathcal{S}_{n-1}(\eta)$ occurred. We do so under the condition that the global modulation system is under control. In particular, for each $n\in \N$, we define the stability event $\mathcal{G}_n(\eta)$ as the set of $\omega\in \Omega$ for which the global stopping times
    $t_c$ and $t_{\mathrm{en}}$ reach
    \begin{align}\label{eqn:globalevent}
    \quad t_c, t_{\mathrm{en}}(\delta_*/2)\geq n\Delta T+ \min\{\tau^{n\Delta T}_{\mathrm{mod}}(\eta),\tau^{n\Delta T}_{\mathrm{en}}(\delta_*),\tau^{n\Delta T}_{\mathrm{st}}(\eta)\} .\end{align}
 Below, $S_{n}^c$ denotes the complement $\Omega \setminus S_{n}$.

\begin{proposition}\label{prop:chain}
   Assuming \hyperlink{hyp:initial}{S1}, \hyperlink{hyp:weight}{S2} and \hyperlink{hyp:constants}{C1}, there exist constants $C_{19}, \delta_{19}>0$, such that for each $n\in\N$ and each $\eta\in[0,\eta_0]$ the stability events $\mathcal{S}_{n-1},\mathcal{S}_{n}(\eta)$ and $\mathcal{G}_n(\eta)$ satisfy
\[\mathbb{P}\Big[\mathcal{S}_{n-1}(\eta)\cap \mathcal{G}_n(\eta)\cap\mathcal{S}^c_{n}(\eta)\Big]\leq  C_{19}e^{-\delta_{19}\eta^2/\sigma^2}.\]
\end{proposition}

% \begin{figure}[h]
%     \centering
    
% \pgfplotsset{width=0.8\textwidth, , height=0.4\textwidth}
% \begin{tikzpicture}
%   % Define variables
%   \pgfmathsetmacro{\eps}{2}
%   \pgfmathsetmacro{\Meps}{1.5}

%   \begin{wxis}[
%     domain=0:12, % Adjust domain to show three periods
%     samples=100,
%     xlabel=$t$,
%     %ylabel={$\sigma v_1^T$},
%     xtick={0, 4, 8, 12},
%     xticklabels={0, $\Delta T$, $2\Delta T$, $3\Delta T$},
%     ytick={\eps, \eps/3, 3*\Meps},
%     yticklabels={$\eta$, $\eta/3$, $3M\eta$},
%     ymin=0,
%     ymax=6,
%     axis lines=middle,
%     grid=both,
%     minor grid style={gray!25},
%     major grid style={gray!50},
%   ]
%     % Exponential decay function, periodically extended with period 2
%     \addplot[domain=0:12, samples=300, thick, blue] {\eps/3 + (3*\Meps - \eps/3)*exp(-1.1*mod(x, 4))};
%     \addplot[domain=0:12, samples=300, thick, red] {\eps + (3*\Meps - \eps/3)*exp(-1.1*mod(x, 4))};
%   \end{wxis}
% \end{tikzpicture}
%     \caption{With high probability, $ v^T$ stays under the blue curve, and $v$ stays under the red curve.}
% \end{figure}
\begin{proof}
Let us write 
\[\bar{t}=t_c\wedge t_{\mathrm{en}}(\delta_*/2)\quad \text{and} \quad \tau=\tau^{n\Delta T}_{\mathrm{mod}}(\eta)\wedge\tau^{n\Delta T}_{\mathrm{en}}(\delta_*)\wedge\tau^{n\Delta T}_{\mathrm{st}}(\eta).\] 
Assuming $\mathcal{S}_{n-1}$, we distinguish three scenarios through which condition (i) in  $\mathcal{S}_{n}$ can fail to hold: 
\begin{align*}
% \mathcal{A}_1=&\big\{t_c<(n+1)\Delta T\big\}\cap \{t_c\leq \bar{t}\}\cap\{t_c\leq n\Delta T+\tau\};\\
%     \mathcal{A}_2=&\big\{t_{\mathrm{st}}(2\eta)<(n+1)\Delta T\big\}\cap \{t_{\mathrm{st}}(2\eta)\leq \bar{t}\}\cap\{t_{\mathrm{st}}(2\eta)\leq n\Delta T+\tau\};\\
%     \mathcal{A}_3=&\big\{t_{\mathrm{en}}(\delta_*/2)<(n+1)\Delta T\big\}\cap \{t_{\mathrm{en}}(\delta_*/2)\leq \bar{t}\}\cap\{t_{\mathrm{en}}(\delta_*/2)\leq n\Delta T+\tau\};\\
    \mathcal{A}_1=&\big\{\tau^{n\Delta T}_{\mathrm{mod}}(\eta)<\Delta T\big\}\cap \{n\Delta T+\tau^{n\Delta T}_{\mathrm{mod}}(\eta)\leq \bar{t}\}\cap\{\tau^{n\Delta T}_{\mathrm{mod}}(\eta)\leq \tau\};\\
   \mathcal{A}_2=&\big\{\tau^{n\Delta T}_{\mathrm{en}}(\delta_*)<\Delta T\big\}\cap \{n\Delta T+\tau^{n\Delta T}_{\mathrm{en}}(\delta_*)\leq \bar{t}\}\cap\{\tau^{n\Delta T}_{\mathrm{en}}(\delta_*)\leq \tau\};\\
   \mathcal{A}_3=&\big\{\tau^{n\Delta T}_{\mathrm{st}}(\eta)<\Delta T\big\}\cap \{n\Delta T+\tau^{n\Delta T}_{\mathrm{st}}(\eta)\leq \bar{t}\}\cap\{\tau^{n\Delta T}_{\mathrm{st}}(\eta)\leq \tau\}.
\end{align*}
The events $\mathcal{A}_1,\mathcal{A}_2$ and $\mathcal{A}_3$ categorize which of the stopping times was first activated before time $(n+1)\Delta T$. Thus, the event that item (i) of $S_n$ fails to hold coincides with 
\[\mathcal{A}=\mathcal{A}_1\cup\mathcal{A}_2\cup\mathcal{A}_3.\]
We proceed by estimating the probabilities of the events $\mathcal{A}_1,\mathcal{A}_2$ and $\mathcal{A}_3$.
\begin{enumerate}
    \item  The event $\mathcal{A}_1$ while also $ \mathcal{S}_{n-1}(\eta)$ implies 
    \[\tau^{n\Delta T}_{\mathrm{mod}}(\eta) \leq \tau^{n\Delta T}_{\mathrm{st}}(\eta)\wedge \Delta T \quad \text{and}\quad n\Delta T+\tau_{\mathrm{mod}}^T(\eta)  \leq t_c.\]
    \Cref{prop:controlofmod} then gives
    \[\mathbb{P}\Big[ \mathcal{S}_{n-1}(\eta)\cap \mathcal{G}_n(\eta)\cap \mathcal{A}_1\Big]\leq e^{-\delta_{12}\eta^2/\sigma^2}.\]
    \item  The event $\mathcal{A}_2$ while also $\mathcal{S}_{n-1}(\eta)$ implies
\begin{align}0<\tau^{n\Delta T}_{\mathrm{en}}(\delta_*)\leq \tau^{n\Delta T}_{\mathrm{mod}}(\eta) \quad \text{and} \quad  n\Delta T+\tau^{n\Delta T}_{\mathrm{en}}(\delta_*)\leq t_{\mathrm{en}}(\delta_*/2)\wedge t_c.\label{eqn:contradict}\end{align}
This event has probability zero. Indeed, we have
\[c^T(s),c(T+s)\in [\tfrac{1}{2}c_{\mathrm{min}},2c_{\mathrm{max}}],\quad \text{for}\quad s\in [0,\tau_{\mathrm{mod}}^{n\Delta T}(\eta)\wedge (t_c-n\Delta T)].\]
\Cref{lem:correspondence} then gives
\[\|v^T(s)\|_{H^1}\leq 2\eta+\|v(T+s)\|_{H^1} <\delta_*,
\]
which contradicts \eqref{eqn:contradict}.
    \item The event $\mathcal{A}_3$ while also $\mathcal{S}_{n-1}(\eta)$ implies
    \[\tau^{n\Delta T}_{\mathrm{st}}(\eta)\leq \Delta T \ \ \text{while} \quad \|v(n\Delta T)\|_{H^1_w}\leq \tfrac{\eta}{3M}\]
    and
    \[\tau^{n\Delta T}_{\mathrm{st}}(\eta) \leq \min\big\{\tau^{n\Delta T}_{\mathrm{en}}(\delta_*),\tau^{n\Delta T}_{\mathrm{mod}}(\eta),\tau^{n\Delta T}_c\big\}.\]
    Via \Cref{prop:remainder}, we then obtain
    \[\mathbb{P}\Big[\mathcal{S}_{n-1}(\eta)\cap \mathcal{G}_n(\eta)\cap\mathcal{A}_3\Big]\leq e^{-\delta_9\eta^2/\sigma^2}.\]
\end{enumerate}
Summarizing our results so far, we have shown that
\[\mathbb{P}\Big[\mathcal{S}_{n-1}(\eta)\cap \mathcal{G}_n(\eta) \cap \mathcal{A}\Big]\leq e^{-\delta_{12}\eta^2/\sigma^2}+e^{-\delta_9\eta^2/\sigma^2}.\]
To complete the proof, we turn our attention to item (ii) of $\mathcal{S}_n$, which fails to hold in the event that
\[\mathcal{B}=\{
    \|v^{n\Delta T}(\Delta T)\|_{H^1_w}> \tfrac{\eta}{9M}\} \cup \{\|v((n+1)\Delta T)\|_{H^1_w}> \tfrac{\eta}{3M}\}.\]
Note that the event
\[\|v^{n\Delta T}(\Delta T)\|_{H^1_w}> \tfrac{\eta}{9M}\quad \text{while} \quad \mathcal{S}_{n-1}(\eta)\quad \text{and} \quad \mathcal{A}^c \quad \text{hold}\] 
occurs with probability less than $e^{-\delta_9\eta^2/\sigma^2}$ through \Cref{prop:remainder}. In the likely case that \[\|v^{n\Delta T}(\Delta T)\|_{H^1_w}\leq \tfrac{\eta}{9M},\] 
it follows via \Cref{lem:correspondence} that also \[\|v((n+1)\Delta T)\|_{H^1_w}\leq \tfrac{\eta}{3M}.\]
Hence, 
\[\mathbb{P}\Big[\mathcal{S}_{n-1}(\eta) \cap \mathcal{G}_n(\eta)\cap \mathcal{B}\cap \mathcal{A}^c\Big]\leq e^{-\delta_9\eta^2/\sigma^2}.\]
The proof is then completed by noting that
\begin{align*}
   \mathcal{S}_{n-1}(\eta)\cap \mathcal{G}_n(\eta) \cap \mathcal{S}_n^c=\Big(\mathcal{S}_{n-1}(\eta)\cap \mathcal{G}_n(\eta) \cap \mathcal{A}\Big)\cup \Big(\mathcal{S}_{n-1}(\eta)\cap \mathcal{G}_n(\eta) \cap \mathcal{B}\Big),
\end{align*}
where the probability of the latter can be estimated as
\[\mathbb{P}\Big[\mathcal{S}_{n-1}(\eta)\cap \mathcal{G}_n(\eta) \cap \mathcal{B}\Big]\leq \mathbb{P}\Big[\mathcal{S}_{n-1}(\eta)\cap \mathcal{G}_n(\eta)\cap \mathcal{A}\Big]+\mathbb{P}\Big[\mathcal{S}_{n-1}(\eta)\cap \mathcal{G}_n(\eta) \cap \mathcal{B}\cap \mathcal{A}^c\Big]. \]
\end{proof}

%With \Cref{prop:chain} established, 
We are now in shape to prove \Cref{thm:bigthm}. Given $c_*,E>0$, we fix $c_{\mathrm{min}}$ and $c_{\mathrm{max}}$ via \hyperlink{hyp:decay}{C2}. Picking  $w\in(0,\sqrt{c_{\mathrm{min}}/3})$ then ensures that we are in the setting of \hyperlink{hyp:weight}{S2}. Lastly, let the constants $\Delta T,\delta_*$ and  $\eta_0$ satisfy \hyperlink{hyp:constants}{C1}. We set out to control the probability of the slightly larger event 
\[\mathcal{C}(\eta)=\{\min\{t_{\mathrm{st}}(2\eta),t_c,t_{\mathrm{en}}(\delta_*/2)\}<T\}.\] 
Writing $\overline{t}=\min\{t_{\mathrm{st}}(2\eta),t_c,t_{\mathrm{en}}(\delta_*/2)\}$, we categorize $\mathcal{C}(\eta)$ into three scenarios:
\begin{align*}
    \mathcal{C}_1(\eta)=&\{t_c<T\}\cap \{t_c\leq \overline{t}\};\\
    \mathcal{C}_2(\eta)=&\{t_{\mathrm{en}}(\delta_*/2)<T\}\cap \{t_{\mathrm{en}}(\delta_*/2)\leq \overline{t}\};\\
    \mathcal{C}_3(\eta)=&\{t_{\mathrm{st}}(2\eta)<T\}\cap \{t_{\mathrm{st}}(2\eta)\leq \overline{t}\},   
\end{align*}
corresponding to which of the stopping times is hit first.
We now subdivide the event $\mathcal{C}_3(\eta)$ by noting that
for each realisation $\omega\in\mathcal{C}_3(\eta)$, the stopping time $t_{\mathrm{st}}(2\eta)(\omega)$ is contained in an interval
\[\big[n(\omega)\Delta T, n(\omega)\Delta T+\Delta T\big),\quad n(\omega)\in \{0, 1,\ldots,\lceil T/\Delta T \rceil-1\}.\]
In view of \Cref{lem:monotonicity}, we in turn find that
   \[n(\omega)\Delta T+\min\{\tau^{n(\omega)\Delta T}_{\mathrm{st}}(\eta)(\omega),\tau^{n(\omega)\Delta T}_{\mathrm{mod}}(\eta)(\omega)\}\leq t_\mathrm{st}(2\eta)(\omega).\]
Hence, $\mathcal{C}_3(\eta)$ implies that either
\[ \mathcal{C}_{3;i}(\eta):=\bigcup_{n=1}^{\lceil T/\Delta T \rceil-1} \Big( \mathcal{S}_{n-1}(\eta)\cap\mathcal{G}_{n}(\eta) \cap\mathcal{S}^c_n(\eta)\Big);\]
i.e. the `chain' of stable events was interrupted, or stability failed on the first interval
\[ \mathcal{C}_{3;ii}(\eta):=\mathcal{S}_{0}^c(\eta)\cap\mathcal{G}_{0}(\eta).\]
% \begin{align*}
% \mathcal{C}_{3;i}(\eta)=&\bigcup_{n\in \{1,\ldots,\lceil T/\Delta T \rceil-1\}} \Big( \mathcal{S}_{n-1}(\eta)\cap\mathcal{G}_{n}(\eta) \cap\mathcal{S}^c_n(\eta)\Big),\\
% \mathcal{C}_{3;ii}(\eta)=&\bigcup_{n\in \{0,1,\ldots,\lceil T/\Delta T \rceil-1\}} \Big(\bigcap_{k=0}^{n} \mathcal{S}_k(\eta)\cap \big\{\tau_{\mathrm{st}}^{n\Delta T}(\eta)\wedge\tau_{\mathrm{mod}}^{n\Delta T}(\eta)\leq (\overline{t}-n\Delta T) \wedge \Delta T \big\}\Big).
% \end{align*}
% Hence, $\mathcal{C}_{3;i}(\eta)$ denotes the event that $\mathcal{G}_{n}(\eta)$ holds for some $n\in \{0,1,\ldots,\lceil T/\Delta T \rceil-1\}$, but not $\bigcap_{k=0}^{n} \mathcal{S}_k(\eta)$. The event $\mathcal{C}_{3;ii}$ denotes that $\bigcap_{k=0}^{n} \mathcal{S}_k(\eta)$ holds for some $n\in \{0,1,\ldots,\lceil T/\Delta T \rceil-1\}$, yet one of the stopping times $\tau_{\mathrm{st}}^{n\Delta T}(\eta)$ or $\tau_{\mathrm{mod}}^{n\Delta T}(\eta)$ is activated before $\overline{t}-n\Delta T$ on $[0,\Delta T]$. 
In summary, we have obtained
\begin{equation}
\label{eq:stb:decomp:c}
    \mathcal{C}(\eta)\subseteq\mathcal{C}_1(\eta)\cup\mathcal{C}_2(\eta)\cup\mathcal{C}_{3;i}(\eta)\cup\mathcal{C}_{3;ii}(\eta).
\end{equation}
Each of these events can be readily controlled using our prior results.
%Let us first describe how \Cref{prop:chain} may be used to bound the probability of $\mathcal{C}_3$. 

% either $\tau_{\mathrm{st}}^{n(\omega)\Delta T}(\eta)$ or $\tau_{\mathrm{mod}}^{n(\omega)\Delta T}(\delta_*\eta)$ is activated before $\overline{t}-n(\omega)\Delta T$ and $\Delta T$. Possibly, this was already the case on an earlier interval, and hence we consider the smallest integer $\tilde{n}(\omega)$ for which either $\tau_{\mathrm{st}}^{\tilde{n}(\omega)\Delta T}(\eta)$ or $\tau_{\mathrm{mod}}^{\tilde{n}(\omega)\Delta T}(\delta_*\eta)$ is activated before $\overline{t}-\tilde{n}(\omega)\Delta T$ and $\Delta T$. 

% Observe that, $\mathbb{P}$-almost surely,
% \begin{align}\mathcal{G}_{T/(\Delta T)}(\eta)\cap\bigcap^{T/(\Delta T)-1}_{k=0}S_k(\eta)\subseteq \{t_{\mathrm{st}}(2\eta)\geq T\} \label{eqn:inclusion}.\end{align}

% Indeed, assuming $\omega\in \bigcap^{T/(\Delta T)-1}_{k=0}S_k(\eta)$, it follows that for each $ s\in[0,\Delta T]$ and $k\in \{0,\ldots,  T/(\Delta T)-1\}$ 
% {\color{red}Fill in}
% The inclusion \eqref{eqn:inclusion} is the final ingredient for the proof of \Cref{thm:bigthm}.

\begin{proof}[Proof of \Cref{thm:bigthm}]
For ease of exposition, we consider $T\geq 1$ for which $T/(\Delta T)\in \N$. We proceed by bounding the probability of each of the events $\mathcal{C}_1(\eta), \mathcal{C}_2(\eta), \mathcal{C}_{3;i}(\eta)$ and $\mathcal{C}_{3;ii}(\eta)$ in \eqref{eq:stb:decomp:c}.
%$\mathcal{C}_1,\mathcal{C}_2$ and $\mathcal{C}_3$.
\begin{enumerate}
    \item \Cref{prop:tccontrol} gives
    \begin{align}\mathbb{P}[\mathcal{C}_1(\eta)]\leq e^{-\delta_{17}/(\sigma^2 T)}.\label{eqn:P1}\end{align}
    \item \Cref{prop:l2control} gives
\begin{align}\mathbb{P}\Big[\mathcal{C}_2(\eta)\Big]\leq C_{16}(\delta_*/2)^{-1}\Big(T\big( \sigma^2+2\epsilon\eta+4\eta^2\big)+2\sqrt{T}\sigma  \eta\Big).\label{eqn:P2}\end{align}
% \item  {\color{red}Noting that $ \{t_{\mathrm{st}}(2\eta)\leq \overline{t}\}\subseteq\mathcal{G}_{T/(\Delta T)}(\eta)$,} it follows from \eqref{eqn:inclusion} that
% \[\mathcal{C}_3(\eta)\subseteq\{t_{\mathrm{st}}(2\eta)<T\}\cap \mathcal{G}_{T/(\Delta T)}(\eta)\subseteq  \Big(\bigcap^{T/(\Delta T)-1}_{k=0}S_k(\eta)\Big)^c\cap  \mathcal{G}_{T/(\Delta T)}(\eta).\]
\item Applying \Cref{prop:chain} on the $T/\Delta T-1$ intervals yields
\begin{align}\mathbb{P}[\mathcal{C}_{3;i}(\eta)]\leq  \Big(\frac{T}{\Delta T}-1\Big) C_{19}e^{-\delta_{19}\eta^2/\sigma^2}.\label{eqn:P31}\end{align}
Similarly,
\begin{align}
    \mathbb{P}[\mathcal{C}_{3;ii}(\eta)]\leq C_{19}e^{-\delta_{19}\eta^2/\sigma^2}. \label{eqn:P32}
\end{align}
% The probability of $\mathcal{C}_{3;ii}(\eta)$ is bounded through \Cref{prop:remainder} and \Cref{prop:controlofmod} on each of the $T/\Delta T$ intervals:
% \begin{align} \mathbb{P}[\mathcal{C}_{3;ii}(\eta)]\leq \frac{T}{\Delta T}e^{-\delta_{12}\eta^2/\sigma^2}+\frac{T}{\Delta T}e^{-\delta_9\eta^2/\sigma^2}.\label{eqn:P32}\end{align}
\end{enumerate}
The bounds \eqref{eqn:P1}, \eqref{eqn:P2}, \eqref{eqn:P31} and \eqref{eqn:P32} together with the union bound
\[\mathbb{P}[\mathcal{C}(\eta)]\leq \mathbb{P}[\mathcal{C}_1(\eta)]+\mathbb{P}[\mathcal{C}_2(\eta)]+\mathbb{P}[\mathcal{C}_{3;i}(\eta)]+\mathbb{P}[\mathcal{C}_{3;ii}(\eta)]\]
imply that 
\begin{align}\mathbb{P}[t_{\mathrm{st}}(2\eta)<T]\leq N(\eta,\sigma,T),\label{eqn:Neta}\end{align}
where
\[N(\eta,\sigma,T)=\tilde{C}T(\eta^2+e^{-\delta_{19} \eta^2/\sigma^2}),\]
for a sufficiently large constant $\tilde{C}>0$. Observe now that for any $\tilde{\eta}\in [0,\eta]$, it $\mathbb{P}$-almost surely holds that $t_{\mathrm{st}}(2\tilde{\eta})\leq t_{\mathrm{st}}(2\eta)$. We hence have the inclusion
\[\{t_{\mathrm{st}}(2\eta)<T\}\subseteq \{t_{\mathrm{st}}(2\tilde{\eta})<T\},\]
allowing \eqref{eqn:Neta} to be improved to
\[\mathbb{P}[t_{\mathrm{st}}(2\eta)<T]\leq \inf_{0\leq \tilde{\eta}\leq \eta}N(\eta,\sigma,\epsilon,T).\]
For $\sigma^2<  \delta_{19}$, we compute that $N(\eta,\sigma,T)$ is minimized at $\tilde{\eta}^2=\frac{\sigma^2}{\delta_{19}}\log(\frac{\delta_{19}}{\sigma^2})$, and 
\[\inf_{\tilde{\eta}\geq 0 }N(\eta,\sigma,T)=\frac{\tilde{C}}{\delta_{19}}T{\sigma^2}\Big(\log\Big(\frac{\delta_{19}}{\sigma^2}\Big)+1\Big).\]
Hence, in case $\eta^2\leq \frac{\sigma^2}{\delta_{19}}\log(\frac{\delta_{19}}{\sigma^2})$, we have 
\[\mathbb{P}[t_{\mathrm{st}}(2\eta)<T]\leq N(\eta,\sigma,T)\leq  \tilde{C}T\Big(\frac{\sigma^2}{\delta_{19}}\log\Big(\frac{\delta_{19}}{\sigma^2}\Big)+e^{-\delta_{19} \eta^2/\sigma^2}\Big).\]
On the other hand, for $\eta^2> \frac{\sigma^2}{\delta_{19}}\log(\frac{\delta_{19}}{\sigma^2})$, we find
\[\mathbb{P}[t_{\mathrm{st}}(2\eta)<T]\leq \inf_{\tilde{\eta}\geq 0 }N(\eta,\sigma,T)=\frac{\tilde{C}}{\delta_{19}}T{\sigma^2}\Big(\log\Big(\frac{\delta_{19}}{\sigma^2}\Big)+1\Big).\]
Both estimates can be absorbed
by the bound
%which means we have %For all $\eta\geq 0 $, we now have 
%established the bound 
\eqref{eqn:finalresult}, completing the proof.
\end{proof}

\section{Validity of reduced dynamics}\label{sec:validity}
%With the stability result \Cref{thm:bigthm} asserted, it remains to prove
Our goal here is to establish the remaining approximation result \Cref{thm:validity}. This result concerns the validity of the approximation $c_{\mathrm{ap}}(t)$ defined in \eqref{eqn:cap} for the soliton amplitude $c(t)$. We thus set out to analyze the evolution equation
        \begin{align}
        \d\big(c(t)-c_{\mathrm{ap}}(t)\big)= & c^0_d(v,c)\ \d t+\epsilon f(t)\big(c_f(v,c)-\tfrac{4}{3}c_{\mathrm{ap}}\big)\ \d t+\sigma^2 \big(c_d(v,c)-c_d(0,c_{\mathrm{ap}})\big) \ \d t\nonumber\\
        &+\sigma \big\langle c_s(v,c)-\tfrac{2}{9}c_{\mathrm{ap}}^{-1/2}\phi_{c_{\mathrm{ap}}}^2, T_{\xi} \d W_t^Q\big\rangle, \label{eqn:ccap}
    \end{align}
    which one finds by subtracting \eqref{eqn:cap} from \eqref{eqn:postulate1}.
    Recalling the constants $C_2$ and $\delta_2$ introduced in \Cref{lem:modcontrol}, we obtain the following useful bounds on 
    the terms above.
    \begin{lemma}\label{lem:lip}
            Assuming \hyperlink{hyp:initial}{S1} and  \hyperlink{hyp:weight}{S2}, there exists a constant $C_{20}>0$ so that for 
            all $c_1,c_{2}\in[\tfrac{1}{2}c_{\mathrm{min}},2c_{\mathrm{max}}]$ 
            and all $v\in L^2_w$ that satisfy $\|v\|_{L^2_w}\leq\delta_2$, we have
    \begin{align*}
    \big|c_f(v,c_1)-\tfrac{4}{3}c_{2}\big|\leq &C_2\|v\|_{L_w^2}+\tfrac{4}{3}|c_1-c_{2}|,\\
    \big|c_d(v,c_1)-c_d(0,c_{2})\big|\leq& C_2\big(1+\|v\|_{L^2_w}+\|v\|_{L^2_w}^2\big)\|v\|_{L^2_w}+C_{20}|c_1-c_{2}|, \\
    \big\|Q^{1/2}c_s(v,c_1)-Q^{1/2}[\tfrac{2}{9}c_{2}^{-1/2}\phi_{c_{2}}^2]\big\|_{L^2}\leq& C_2\|v\|_{L_w^2}+C_{20}|c_1-c_{2}|.\end{align*}
    \end{lemma}
\begin{proof}
Recalling that $c_f(0,c)=\tfrac{4}{3}c$, we estimate
\[\big|c_f(v,c_1)-\tfrac{4}{3}c_{2}\big|\leq \big|c_f(v,c_1)-c_f(0,c_1)\big|+\big|\tfrac{4}{3}c_1-\tfrac{4}{3}c_{2}\big|.\]
Applying \Cref{lem:modcontrol} to estimate the first term yields
\[\big|c_f(v,c_1)-\tfrac{4}{3}c_{2}\big|\leq C_2\|v\|_{L_w^2}+\tfrac{4}{3}|c_1-c_{2}|.\]
The remaining inequalities follow analogously through the Lipschitz bound
\begin{align*}
    |c_d(0,c_1)-c_d(0,c_2)|+\|Q^{1/2}c_s(0,c_1)-Q^{1/2}c_s(0,c_2)\|_{L^2}\leq C_{20}|c_1-c_2|
\end{align*}
on $[\tfrac{1}{2}c_{\mathrm{min}},2c_{\mathrm{max}}]$.
\end{proof}
As a further preparation, we establish control on the stochastic integral in \eqref{eqn:ccap}. Recall the notation
\[t_{\mathrm{ap}}(\lambda)=\sup\big\{t\geq 0:|c(t)-c_{\mathrm{ap}}(t)|\leq\lambda\big\},\]
and let $\lambda_0=\min\{\tfrac{1}{2}c_{\mathrm{min}},c_{\mathrm{max}}\}$.

\begin{lemma}\label{lem:BDG}
 Assuming \hyperlink{hyp:initial}{S1}, \hyperlink{hyp:weight}{S2} and  \hyperlink{hyp:constants}{C1}, there exists a constant $C_{21}>0$ so that for each $T\geq0$ and $\eta\in[0, \delta_2]$
we have%\todo{[hjh: regel te lang (kijk goed of elders ook nog...)]}
    \begin{align*}
    \mathbb{E}\sup_{0\leq t \leq T\wedge t_{c}\wedge t_{\mathrm{ap}}(\lambda_0)\wedge t_{\mathrm{st}}(\eta)} \Big|&\int_0^t \langle c_s(v,c)-\tfrac{2}{9}c_{\mathrm{ap}}^{-1/2}\phi_{c_{\mathrm{ap}}}^2, T_{\xi} \d W_{t'}^Q \rangle\Big|
    \\\leq & C_{21} \sqrt{T}\mathbb{E}\sup_{0\leq {t'} \leq T \wedge t_{\mathrm{st}}(\eta)}\|v({t'})\|_{L^2_w}
    +C_{21} \sqrt{T}\mathbb{E}\sup_{0\leq {t'} \leq T\wedge t_{c}\wedge t_{\mathrm{ap}}(\lambda_0)} |c({t'})-c_{\mathrm{ap}}({t'})|.
\end{align*}
    \end{lemma}
    \begin{proof}
Let us write $\overline{t}= t_{c}\wedge t_{\mathrm{ap}}(\lambda_0)\wedge t_{\mathrm{st}}(\eta)$. The Burkholder-Davis-Gundy inequality \cite[proposition 2.1]{BDG} provides control of the stochastic integral:
%as follows:
\begin{align*}
    &\mathbb{E}\sup_{0\leq t \leq T\wedge\overline{t}} \Big|\int_0^t \langle c_s(v,c)-\tfrac{2}{9}c_{\mathrm{ap}}^{-1/2}\phi_{c_{\mathrm{ap}}}^2, T_{\xi} \d W_{t'}^Q \rangle\Big|\\
    \leq & \tilde{C}_1 \mathbb{E}\Big[ \Big(\int_0^{T\wedge\overline{t}}\big\|Q^{1/2}c_s(v,c)-Q^{1/2}[\tfrac{2}{9}c_{\mathrm{ap}}^{-1/2}\phi_{c_{\mathrm{ap}}}^2]\big\|_{L^2}^2\d t'\Big)^{1/2} \Big].\end{align*}
\Cref{lem:lip} then yields
    \begin{align*}
    &\mathbb{E}\sup_{0\leq t \leq T\wedge\overline{t}} \Big|\int_0^t \langle c_s(v,c)-\tfrac{2}{9}c_{\mathrm{ap}}^{-1/2}\phi_{c_{\mathrm{ap}}}^2, T_{\xi} \d W_{t'}^Q \rangle\Big|\\
    \leq&  \tilde{C}_2 \mathbb{E}\Big[ \Big(\int_0^{T\wedge\overline{t}}\|v\|^2_{L^2_w}+|c-c_{\mathrm{ap}}|^2\d t\Big)^{1/2} \Big]\\
    \leq & \tilde{C}_2 \sqrt{T}\mathbb{E}\sup_{0\leq t \leq T \wedge t_{\mathrm{st}}(\eta)}\|v(t)\|_{L^2_w}+\tilde{C}_2 \sqrt{T}\mathbb{E}\sup_{0\leq t \leq T\wedge t_{c}\wedge t_{\mathrm{ap}}(\lambda_0)} |c(t)-c_{\mathrm{ap}}(t)|.
\end{align*}
\end{proof}
 We are now ready to control $|c(t)-c_{\mathrm{ap}}(t)|$ via a Gr\"onwall argument, resulting in conditional control of the stopping time $t_{\mathrm{ap}}$.%\todo{[hjh: heb je hier vlak boven geintroduceerd; hoeft niet. Kort tekstje van wat je hier doet? ]}
\begin{lemma}  \label{lem:markov} Assuming \hyperlink{hyp:initial}{S1}, \hyperlink{hyp:weight}{S2}, \hyperlink{hyp:constants}{C1} and \hyperlink{hyp:decay}{C2}, there exists a constant $C_{22}>0$ so that for each $T\geq1, \eta\in(0,\delta_2]$, each $\sigma,\epsilon \in [0,\eta]$ and each $\lambda\in(0,\lambda_0]$,
we have%\todo{[hjh: je parameters moeten ook nog positief zijn; ook elders]}
        \[\mathbb{P}\big[t_{\mathrm{ap}}(\lambda)<T\cap t_{\mathrm{st}}(\eta)\wedge t_c \geq T\big]\leq C_{22}Te^{\sigma^2 T}\frac{\eta^2}{\lambda}.\] 
\end{lemma}
\begin{proof}
    
    Applying \Cref{lem:lip} and \Cref{lem:BDG} to the SDE \eqref{eqn:ccap} yields
            \begin{align*}
        |c(t)-c_{\mathrm{ap}}(t)|\leq  & C_2\int_0^t \|v\|^2_{L^2_w}\d t'+C_2\epsilon \int_0^t \|v\|_{L^2_w}\d t'+\tfrac{4}{3}\epsilon \int_0^t|f(t')||c-c_{\mathrm{ap}}|\d t'\\
        &+C_2\sigma^2 \int_0^t (1+\|v\|_{L^2_w}^2)\|v\|_{L^2_w}\d t'+C_{20}\sigma^2\int_0^t|c-c_{\mathrm{ap}}| \ \d t'\\
        &+\sigma\Big|\int_0^t \langle c_s(v,c)-\tfrac{2}{9}c_{\mathrm{ap}}^{-1/2}\phi_{c_{\mathrm{ap}}}^2, T_{\xi} \d W_{t'}^Q\rangle\Big|,
    \end{align*}
     for all $t\in[0,t_{c}\wedge t_{\mathrm{ap}}(\lambda_0)]$. Let us once more write $\overline{t}=t_{c}\wedge t_{\mathrm{ap}}(\lambda_0)\wedge t_{\mathrm{st}}(\eta)$. In addition, we introduce the notation
     \[E(t):=\mathbb{E}\sup_{0\leq t' \leq t\wedge \overline{t}} |c(t')-c_{\mathrm{ap}}(t')|,\quad t\geq 0.\]
     Inspecting the bound above implies
     %Taking a supremum over $t\in[0, \overline{t}]$ and expectations gives
\begin{align}
    E(T)\leq  & C_2T \eta^2+C_2\epsilon T \eta+\tfrac{4}{3}\epsilon \int_0^T|f(t)| E(t)\d t\nonumber\\
        &+C_2\sigma^2 T (1+\eta^2)\eta+C_{20}\sigma^2\int_0^TE(t)\d t \nonumber\\
        &+\sigma C_{21} \sqrt{T} \Big(\eta+E(T)\Big).\label{eqn:toleft}
\end{align}
% and
% \begin{align*}
%      \mathbb{E}\sup_{0\leq s \leq T\wedge t_{c}\wedge t_{\mathrm{st}}(\eta)} \Big|\int_0^t\langle c^{-1/2}\phi_{c}^2-c_{\mathrm{ap}}^{-1/2}\phi_{c_{\mathrm{ap}}}^2, T_{\xi} \d W_s^Q\rangle\Big|\leq &\tilde{C}_2 \mathbb{E}\Big[ \Big(\int_0^{T\wedge t_{c}\wedge t_{\mathrm{st}}(\eta)}\|Q^{1/2}c^{-1/2}\phi_{c}^2-Q^{1/2}c_{\mathrm{ap}}^{-1/2}\phi_{c_{\mathrm{ap}}}^2\|_{L^2}^2\d t\Big)^{1/2} \Big]\\
%      \leq & \tilde{C}_1 \tilde{C}_2 \mathbb{E}\Big[ \Big(\int_0^{T\wedge t_{c}\wedge t_{\mathrm{st}}(\eta)}|c-c_{\mathrm{ap}}|^2\d t\Big)^{1/2} \Big]\\
%      \leq& \tilde{C}_1 \tilde{C}_2 \sqrt{T}\mathbb{E}\sup_{0\leq s \leq T\wedge t_{c}\wedge t_{\mathrm{st}}(\eta)} |c(s)-c_{\mathrm{ap}}(s)|.
% \end{align*}
Imposing the restriction $\sigma\sqrt{T}\leq \frac{1}{2C_{21}}$, we may bring the last term in \eqref{eqn:toleft} to the left to find
\begin{align*}
    \tfrac{1}{2}E(T)\leq  & C_2T \eta^2+C_2\epsilon T \eta+\tfrac{4}{3}\epsilon \int_0^T|f(t)| E(t)\d t\\
        &+C_2\sigma^2 T (1+\eta^2)\eta+C_{20}\sigma^2\int_0^TE(t)  \d t+C_{21}\sigma \eta \sqrt{T}.
\end{align*}
Gr\"onwall's inequality then yields
\begin{align*}
    \tfrac{1}{2}E(T)\leq  & C_2\exp\Big( \int_0^T C_{20}\sigma^2+\frac{4}{3}\epsilon |f(t)|\d t\Big)\\
    &\times\Big(T \eta^2+\epsilon T \eta+\sigma^2 T (1+\eta^2)\eta+\frac{C_{21}}{C_2}\sigma \eta \sqrt{T} \Big),
\end{align*}
in which \hyperlink{hyp:decay}{C2} ensures
\[\exp\Big( \int_0^T C_{20}\sigma^2+\frac{4}{3}\epsilon |f(t)|\d t\Big)\leq \exp\Big( C_{20}\sigma^2T+\frac{4}{3}E\Big).\]
Markov's inequality finally yields
\begin{align*}
    \mathbb{P}\Big[\sup_{0\leq t' \leq T\wedge \overline{t}} |c(t')-c_{\mathrm{ap}}(t')|\geq\lambda \Big]\leq &\frac{2}{\lambda}  C_2\exp\Big( C_{20}\sigma^2T+\frac{4}{3}E\Big)\\
    &\times\Big(T \eta^2+\epsilon T \eta+\sigma^2 T (1+\eta^2)\eta+\frac{C_{21}}{C_2}\sigma \eta \sqrt{T} \Big).
\end{align*}
To complete the proof, note that $t_{\mathrm{ap}}(\lambda)<T$ while $t_{\mathrm{st}}(\eta)\wedge t_c\geq T$ implies 
\[\sup_{0\leq t'\leq T\wedge \overline{t}}|c(t')-c_{\mathrm{ap}}(t')|> \lambda .\]
\end{proof}
% which in turn, via \Cref{lem:correspondence} contains 
% \[\bigcap_{k=0}^{\lceil\frac{\epsilon^{-1}T_{\mathrm{max}}}{\Delta T}\rceil-1}B_k\]
% where $B_k$ is the event
% \[\sup_{0\leq t \leq \Delta T\wedge \tau_k}\|v^{k\Delta T}(t)\|_{H^1_w} \leq \eta/2 \quad \text{and} \quad \|v(k\Delta T)\|_{H^1_w} \leq \eta/6M \]
% while
% \[\sup_{0\leq t \leq \Delta T\wedge \tau}|c^{k\Delta T}(t)-c(k\Delta T+s)|+|\xi^{k\Delta T}(t)-\xi(k\Delta T+s)|+|\xi(k\Delta T)+c(k\Delta T)t-\xi(k\Delta T+s)|\leq \eta/2C.\]
%  Using that $\bigcap_{k=0}^N B_k\subseteq A^c$ implies $\mathbb{P}(A)\leq \sum_{k=0}^N \mathbb{P}(B_k^c)$,
% we set out to bound $\mathbb{P}(B_k^c)$.
\begin{proof}[Proof of \Cref{thm:validity}]
For any $\eta\geq 0$, we have the union bound
\[\mathbb{P}[t_{\mathrm{ap}}(\lambda)<T]\leq  \mathbb{P}[t_{\mathrm{ap}}(\lambda)<T\cap t_{\mathrm{st}}(\eta)\wedge t_c\geq T]+\mathbb{P}[t_{\mathrm{st}}(\eta)\wedge t_c<T]\]
and hence
\[\mathbb{P}[t_{\mathrm{ap}}(\lambda)<T]\leq  \inf_{\eta \geq 0}\Big(\mathbb{P}[t_{\mathrm{ap}}(\lambda)<T\cap t_{\mathrm{st}}(\eta)\wedge t_c\geq T]+\mathbb{P}[t_{\mathrm{st}}(\eta)\wedge t_c<T]\Big).\]
Applying \Cref{thm:bigthm} and \Cref{lem:markov} now gives
\[\mathbb{P}[t_{\mathrm{ap}}(\lambda)<T]\leq (C+C_{22})T \inf_{\eta \geq 0}\Big(\frac{\eta^2}{\lambda}+e^{-\delta \eta^2/\sigma^2}\Big)+CT\sigma^2\log(1/\sigma).\]
In case $\sigma^2<\lambda \delta$, the infimum is attained at 
\[\eta^2=\frac{\sigma^2}{\delta}\log\Big(\frac{\lambda \delta}{\sigma^2}\Big)\]
yielding
\[\mathbb{P}[t_{\mathrm{ap}}(\lambda)<T]\leq (C+C_{22})T {\frac{\sigma^2}{\lambda\delta}\Big(\log\Big(\frac{\lambda \delta}{\sigma^2}\Big)}+1\Big)+CT\sigma^2\log(1/\sigma).\]
Upon increasing the constant $C$ if necessary, we obtain the result \eqref{eqn:approximation}.
\end{proof}

\appendix

\section{Stopping times}
\label{app:stop}
Here, we provide an overview of the various stopping times introduced throughout the work. In relation to the global modulation system $\big(v(t),c(t),\xi(t)\big)$ introduced in \Cref{sec:global}, we 
use %have for each $\eta>0$ 
the stopping times
\begin{align*}
t_c=&\sup\{t\geq 0:c(t) \in [c_{\mathrm{min}},c_{\mathrm{max}}]\};\\
t_{\mathrm{log}}=&\sup\{t\geq 0 :|\log(c(t)/c_*)-\frac{4}{3}\epsilon \int_0^t|f(t')|\d t'|\leq E\};\\
t_{\mathrm{st}}(\eta)=& \sup\big\{t\geq 0:\|v(t)\|_{H_w^1}\leq \eta\big\};\\
t_{\mathrm{en}}(\eta)=&\sup\big\{t\geq 0:\big\|v(t)\big\|_{L^2}\leq \eta  \big\};\\
t_{\mathrm{ap}}(\eta)=&\sup\big\{t\geq 0:|c(t)-c_{\mathrm{ap}}(t)|\leq\eta\big\}.
\end{align*}
For each $T\geq 0$, the following stopping times are related to the local modulation system \linebreak $\big(v^T(s),c^T(s),\xi^T(s)\big)$ introduced in \Cref{sec:local}:
\begin{align*}\tau^T_{\mathrm{st}}(\eta)=&\sup\big\{s\geq 0:\big\|v^T(s)\big\|_{H^1_w}\leq \eta  \big\};\\
    \tau^T_c=&\sup\big\{s\geq 0:c(T+s)\in[\tfrac{1}{2}c_{\mathrm{min}},2c_{\mathrm{max}}]  \big\};\\
   \tau^T_{\mathrm{en}}(\eta)=&\sup\big\{s\geq 0:\big\|v^T(s)\big\|_{L^2}\leq \eta  \big\};\\
   \tau^{T}_{\mathrm{amp,1}}(\eta)
   =&\sup\big\{s\geq 0:\big|c(T)-c^T(s)\big|\leq \delta_*\eta\}; \\
   \tau^{T}_{\mathrm{amp,2}}(\eta)
   =&\sup\big\{s\geq 0:\big|c(T+s)-c^T(s)\big|\leq \delta_*\eta\}; \\
  \tau^{T}_{\mathrm{pos,1}}(\eta)
   =&\sup\big\{s\geq 0:\big|\xi(T)+c(T)s-\xi^T(s)\big|\leq 2\Delta T \delta_*\eta  \big\};\\
   \tau^{T}_{\mathrm{pos,2}}(\eta)
   =&\sup\big\{s\geq 0:\big|\xi(T+s)-\xi^T(s)\big|\leq 2\Delta T \delta_*\eta  \big\},
\end{align*}
and
\[\tau^T_{\mathrm{mod}}(\eta)=\tau^{T}_{\mathrm{amp},1}(\eta)\wedge\tau^{T}_{\mathrm{amp},2}(\eta)\wedge\tau^{T}_{\mathrm{pos,1}}( \eta)\wedge \tau^{T}_{\mathrm{pos,2}}( \eta).\]

\section{Technical proofs}
\label{app:mild}
Here, we provide the proofs of various lemmas which were omitted in 
\Cref{sec:global}, \Cref{sec:local}  and \Cref{sec:weighted}.
\begin{proof}[Proof of \Cref{lem:innerprod}]
We derive the evolution equation for $\big\langle v(t),\phi_c(t)\big\rangle$, noting that the remaining evolution equation for $\big\langle v(t),\zeta_c(t)\big\rangle$ follows analogously. First, we introduce the notation
 \[\big\langle v(t),\phi_{c(t)}\big\rangle=\big\langle u(t,\cdot+\xi(t)),\phi_{c(t)}\big\rangle -\big\langle\phi_{c(t)},\phi_{c(t)}\big\rangle =\big\langle u(t),\phi_{c(t)}(\cdot-\xi(t))\big\rangle -6c^{3/2}(t)=:F\big(u(t),c(t),\xi(t)\big).\]
 We now apply the mild It\^{o} formula to the functional $F:H^1\times\R\times \R\to \R$. We therefore interpret the tuple $(u,c,\xi)$ as a mild process with respect to the $C_0$-group $\{S(t)\}_{t\in \R}$ given by
 \begin{align*}
     S(t)=\begin{bmatrix}
         e^{-\partial_x^3t} & 0 & 0\\
         0 & I_{\R} &0\\
         0 & 0 & I_{\R}
     \end{bmatrix}, \quad t\in \R.
 \end{align*}
 We collect that $F$ is twice Fr\'echet differentiable, with first derivatives
 \begin{align*}
     \d_uF(u,c,\xi)[v]=&\big\langle v,\phi_c(\cdot-\xi)\big\rangle,\\
     \d_cF(u,c,\xi)=&\big\langle u,\partial_c\phi_c(\cdot-\xi)\big\rangle-9c^{1/2},\\
     \d_{\xi} F(u,c,\xi)=&-\big\langle u,\partial_x\phi_c(\cdot-\xi)\big\rangle,
\end{align*}
and second derivatives
\begin{align*}
     \d^2_{uu}F(u,c,\xi)[v,w]=&0,\\
     \d^2_{cc}F(u,c,\xi)=&\big\langle u,\partial^2_c\phi_c(\cdot-\xi)\big\rangle-\tfrac{9}{2}c^{-1/2},\\
     \d^2_{\xi\xi}F(u,c,\xi)=&\big\langle u,\partial_x^2\phi_c(\cdot-\xi)\big\rangle,\\
     \d^2_{uc}F(u,c,\xi)[v]=&\big\langle v,\partial_c\phi_c(\cdot-\xi)\big\rangle,\\
     \d^2_{u\xi}F(u,c,\xi)[v]=&-\big\langle v,\partial_x\phi_c(\cdot-\xi)\big\rangle,\\
     \d^2_{c\xi}F(u,c,\xi)=&-\big\langle u,\partial^2_{cx}\phi_c(\cdot-\xi)\big\rangle.
 \end{align*}
 For any orthonormal basis $\{e_k\}_{k=0}^\infty$ of $L^2$,
 \cite[Theorem 1]{mild} then yields:
 \begin{align}\label{eqn:Fmild}
    F\big(u(t),c(t),\xi(t)\big)=& F\big(e^{-\partial_x^3t}\phi_{c_*},c_*,0\big)+\int_0^tF_1(t,t')\d t'+\sigma^2\int_0^tF_2(t,t')\d t'\\
    &+\sigma^2\sum_{k=0}^\infty\int_0^tF_3(t,t',k)\d t' +\sigma\int_0^tF_4(t,t')\d W^ Q_{t'},\nonumber
    \end{align}
with
\begin{align*}
    F_1(t,t')=& \d_uF\big(e^{-\partial_x^3(t-t')}u,c,\xi\big)\big[ e^{-\partial_x^3 (t-t')}(-\partial_x(u^2)+\epsilon f(t')u)\big]\\
    &+ \d_cF\big(e^{-\partial_x^3(t-t')}u,c,\xi\big) c_d^{\sigma,\epsilon} + \d_{\xi} F\big(e^{-\partial_x^3(t-t')}u,c,\xi\big) \xi_d^{\sigma,\epsilon},\\
     F_2(t,t')=&\tfrac{1}{2}\d^2_{cc}F\big(e^{-\partial_x^3(t-t')}u,c,\xi)\big\|Q^{1/2}c_s \|_{L^2}^2+\tfrac{1}{2}\d^2_{\xi\xi}F\big(e^{-\partial_x^3(t-t')}u,c,\xi\big)\|Q^{1/2}\xi_s \|_{L^2}^2\\
    &+\d^2_{c\xi}F\big(e^{-\partial_x^3(t-t')}u,c,\xi\big)\langle Q^{1/2}\xi_s ,Q^{1/2}c_s \rangle,\\
     F_3(t,t',k)=&\d^2_{uc}F\big(e^{-\partial_x^3(t-t')}u,c,\xi\big)
     [e^{-\partial_x^3(t-t')}uQ^{1/2}e_k]\langle c_s ,T_{\xi} Q^{1/2}e_k\rangle\\
      &+ \d^2_{u\xi}F\big(e^{-\partial_x^3(t-t')}u,c,\xi\big)
     [e^{-\partial_x^3(t-t')}uQ^{1/2}e_k]\langle \xi_s ,T_{\xi} Q^{1/2}e_k\rangle,\\
    F_4(t,t')[h]=& \d_uF\big(e^{-\partial_x^3(t-t')}u,c,\xi\big)[ e^{-\partial_x^3 (t-t')}uh]\\
    &+ \d_cF\big(e^{-\partial_x^3(t-t')}u,c,\xi\big) \langle c_s ,T_{\xi} h\rangle+\d_{\xi} F\big(e^{-\partial_x^3(t-t')}u,c,\xi\big) \langle \xi_s ,T_{\xi}h\rangle.
 \end{align*}
 Above, we have suppressed the dependence of $c_d^{\sigma,\epsilon},\xi_d^{\sigma,\epsilon}$ on $(v(t'),c(t'),t')$ and of $c_s,\xi_s$ on $(v(t'),c(t'))$. Substituting the derivatives, this becomes
 \begin{align*}
 F\big(e^{-\partial_x^3t}\phi_{c_*},c_*,0\big)=&\langle e^{-\partial_x^3t}\phi_{c_*},\phi_{c_*}\rangle -6c_*^{3/2}, \\
    F_1(t,t')=& \big\langle  e^{-\partial_x^3 (t-t')}(-\partial_x(u^2)+\epsilon f(t')u),\phi_c(\cdot-\xi)\big\rangle\\
    &+\big(\langle e^{-\partial_x^3(t-t')}u,\partial_c\phi_c(\cdot-\xi)\rangle-9c^{1/2}\big) c_d^{\sigma,\epsilon}- \big\langle e^{-\partial_x^3(t-t')}u,\partial_x\phi_c(\cdot-\xi)\big\rangle \xi_d^{\sigma,\epsilon},\\
     F_2(t,t')=&\tfrac{1}{2}\big(\langle e^{-\partial_x^3(t-t')}u,\partial^2_c\phi_c(\cdot-\xi)\rangle-\tfrac{9}{2}c^{-1/2}\big)\|Q^{1/2}c_s\|_{L^2}^2\\
     &+\tfrac{1}{2}\big\langle e^{-\partial_x^3(t-t')}u,\partial_x^2\phi_c(\cdot-\xi)\big\rangle\|Q^{1/2}\xi_s\|_{L^2}^2\\
    &-\big\langle e^{-\partial_x^3(t-t')}u,\partial^2_{cx}\phi_c(\cdot-\xi)\big\rangle\big\langle Q^{1/2}\xi_s,Q^{1/2}c_s\big\rangle,\\
     F_3(t,t',k)=&\big\langle e^{-\partial_x^3(t-t')}uQ^{1/2}e_k,\partial_c\phi_c(\cdot-\xi)\big\rangle\big\langle c_s,T_{\xi} Q^{1/2}e_k\big\rangle\\
     &-\langle e^{-\partial_x^3(t-t')}uQ^{1/2}e_k,\partial_x\phi_c(\cdot-\xi)\rangle\langle \xi_s,T_{\xi} Q^{1/2}e_k\rangle,\\
    F_4(t,t')[h]=&  \big\langle  e^{-\partial_x^3 (t-t')}u h,\phi_c(\cdot-\xi)\big\rangle+ \big(\langle e^{-\partial_x^3(t-t')}u,\partial_c\phi_c(\cdot-\xi)\rangle-9c^{1/2}\big) \langle c_s,T_{\xi}h\rangle\\
    &-\big\langle e^{-\partial_x^3(t-t')}u,\partial_x\phi_c(\cdot-\xi)\big\rangle \langle \xi_s,T_{\xi}h\rangle.
 \end{align*}
We now show how to convert the mild expression \eqref{eqn:Fmild} into a strong form\footnote{This computation resembles the procedure for passing from a mild to strong solution in case sufficient regularity is available.}, focusing on the case $\sigma=\epsilon=0$ for ease of exposition. 
Differentiating $\big\langle v(t),\phi_{c(t)}\big\rangle$ via Leibniz' rule gives 
 \begin{align*}
    \partial_t\big\langle v(t),\phi_{c(t)}\big\rangle=&\partial_tF\big(u(t),c(t),\xi(t)\big)=\partial_tF\big(e^{-\partial_x^3t}\phi_{c_*},c_*,0\big)+F_1(t,t)+\int_0^t \partial_tF_1(t,t')\d t',
    \end{align*}
and thus
\begin{align*}
    \partial_t\big\langle v(t),\phi_{c(t)}\big\rangle=&\big\langle e^{-\partial_x^3t}\phi_{c_*},\partial_x^3\phi_{c_*}\big\rangle- \big\langle \partial_x(u^2(t)),\phi_c(\cdot-\xi)\big\rangle-\int_0^t \big\langle \partial_x\big(u^2(t')\big),e^{\partial_x^3 (t-t')}\partial_x^3\phi_c(\cdot-\xi)\big\rangle\d {t'}\\
    &+\big(\langle u(t),\partial_c\phi_c(\cdot-\xi)\rangle-9c^{1/2}\big)c_d^0\big(v(t)\big)+\int_0^t \big\langle u(t'),e^{\partial_x^3(t-t')}\partial_x^3\partial_c\phi_c(\cdot-\xi)\big\rangle c_d^0\big(v(t')\big)\d {t'}\\
    &-\big\langle u(t),\partial_x\phi_c(\cdot-\xi)\big\rangle \xi_d^0\big(v(t)\big)-\int_0^t  \big\langle u(t'),e^{\partial_x^3(t-t')}\partial_x^3\partial_x\phi_c(\cdot-\xi)\big\rangle \xi_d^0\big(v(t')\big)\d {t'},
 \end{align*}
where we have moved the semigroup to the other side of the inner products via the adjoint relation $(e^{-\partial_x^3t})^*=e^{\partial_x^3t}$. Now we recognize the mild formula
\begin{align*}
  \big\langle u(t),\partial_x^3\phi_{c}(\cdot-\xi)\big\rangle=&\big\langle e^{-\partial_x^3t}\phi_{c_*},\partial_x^3\phi_{c_*}\big\rangle-\int_0^t \big\langle \partial_x\big(u^2(t')\big),e^{\partial_x^3 (t-t')}\partial_x^3\phi_c(\cdot-\xi)\big\rangle\d {t'}\\
    &+\int_0^t \big\langle u(t'),e^{\partial_x^3(t-t')}\partial_x^3\partial_c\phi_c(\cdot-\xi)\big\rangle c_d^0\big(v(t')\big)\d {t'}\\
    &-\int_0^t  \big\langle u(t'),e^{\partial_x^3(t-t')}\partial_x^3\partial_x\phi_c(\cdot-\xi)\big\rangle \xi_d^0\big(v(t')\big)\d {t'}.
\end{align*}
We thus we arrive at the strong form
 \begin{align*}
    \d \big\langle v(t),\phi_{c(t)}\big\rangle=& \big\langle u(t),\partial_x^3\phi_{c}(\cdot-\xi)\big\rangle \ \d t- \big\langle \partial_x\big(u^2(t)\big),\phi_c(\cdot-\xi)\big\rangle\ \d t\\
    &+\big(\big\langle u(t),\partial_c\phi_c(\cdot-\xi)\big\rangle-9c^{1/2}\big)c_d^0\big(v(t)\big)\ \d t-\big\langle u(t),\partial_x\phi_c(\cdot-\xi)\big\rangle \xi_d^0\big(v(t)\big)\ \d t.
 \end{align*}
 After substituting $u(t,\cdot+\xi)=\phi_{c}+v(t)$:
 \begin{align*}
    \d \big\langle v(t),\phi_{c(t)}\big\rangle=& \big\langle v(t),\partial_x^3\phi_{c}\big\rangle \ \d t- \big\langle \partial_x(\phi_{c}+v(t))^2,\phi_c\big\rangle\ \d t\\
    &+\big(\big\langle \phi_{c}+v(t),\partial_c\phi_c\big\rangle-9c^{1/2}\big)c_d^0\big(v(t)\big)\ \d t-\big\langle v(t),\partial_x\phi_c\big\rangle \xi_d^0\big(v(t)\big)\ \d t
 \end{align*}
where rewriting
\[\big\langle \partial_x(\phi_{c}+v(t))^2,\phi_c\big\rangle=-2\big\langle v(t),\phi_c\partial_x\phi_c\big\rangle -\big\langle N\big(v(t)\big),\phi_c\big\rangle\]
leads to
 \begin{align*}
    \d \big\langle v(t),\phi_{c(t)}\big\rangle=& \big\langle v(t),\partial_x^3\phi_{c}+2\phi_c\partial_x\phi_c-c\partial_x\phi_c\big\rangle \ \d t+ \big\langle N\big(v(t)\big),\phi_c\big\rangle\ \d t\\
    &+\big(\langle \phi_{c}+v(t),\partial_c\phi_c\rangle-9c^{1/2}\big)c_d^0\big(v(t)\big)\ \d t-\big\langle v(t),\partial_x\phi_c\big\rangle \Omega_d^0\big(v(t)\big)\ \d t.
 \end{align*}
Using the traveling wave identity
$\partial_x^3\phi_{c}+2\phi_c\partial_x\phi_c-c\partial_x\phi_c=0$, we arrive at the result
 \begin{align*}
    \d \big\langle v(t),\phi_{c(t)}\big\rangle=&  \big\langle N\big(v(t)\big),\phi_c\big\rangle\ \d t+\big(\langle \phi_{c}+v(t),\partial_c\phi_c\rangle-9c^{1/2}\big)c_d^0\big(v(t)\big)\ \d t-\big\langle v(t),\partial_x\phi_c\big\rangle \Omega_d^0\big(v(t)\big)\ \d t.
 \end{align*}
 The $\sigma$- and $\epsilon$-dependent terms in \eqref{eqn:Fmild} can be treated analogously, which completes the proof.
\end{proof}

\begin{proof}[Proof of \Cref{prop:mild}]
We compute the mild form of
\[v^T(s,x)=u\big(T+s,x+\xi(T)+c(T)s\big)-\Phi^T(\mathbf{m}^T(s),s,x).\]
Recalling that 
\[\Phi^T(\mathbf{m}^T(s),s,x):=\phi_{c^T(s)}\big(x+\xi(T)+c(T)s-\xi^T(s)\big),\]
a straightforward application of It\^{o}'s lemma yields
\begin{align*}
    \d \Phi^T=& \Big[\partial_c\Phi^Tc_d^{\sigma,\epsilon,T}+\partial_x\Phi^T(c(T)-c^T-\Omega_d^{\sigma,\epsilon,T})\Big]\d s\\
    &+\frac{\sigma^2}{2}\Big[\partial^2_c\Phi^T\|Q^{1/2}c_s^{T}\|_{L^2}^2+\partial^2_x\Phi^T\|Q^{1/2}\xi_s^{T}\|_{L^2}^2\Big] \d s\\
    &+\sigma^2\partial^2_{xc}\Phi^T\langle Q^{1/2}c_s^{T},Q^{1/2}\xi_s^T\rangle\d s\\
    &+\sigma  \partial_c\Phi^T\langle c_s^{T},T_{\xi(T)+c(T)s}\d W^Q_{T+s}\rangle+\sigma \partial_x\Phi^T\langle\xi_s^{T},T_{\xi(T)+c(T)s}\d W^Q_{T+s}\rangle
\end{align*}
where we have suppressed the dependence of $\Phi^T,c_d^{\sigma,\epsilon,T}, \xi_d^{\sigma,\epsilon,T},c_s^T$ and $\xi_s^T$ on $(\mathbf{m}^T(s),s)$.
Using the traveling wave identity
\begin{align*}0=&-\partial_x^3\Phi^T-\partial_x(\Phi^T)^2+c^T\partial_x\Phi^T\\
&=\mathcal{L}_{c(T)}\Phi^T+\big(c^T-c(T)\big)\partial_x\Phi^T-\partial_x\big((\Phi^T)^2\big)-2\partial_x\big(\phi_{c(T)}\Phi^T\big)
\end{align*}
we may pass to the mild form
\begin{align}\label{eqn:mildPhi}
     \Phi^T(s)=&e^{\mathcal{L}_{c(T)}s}\phi_{c(T)}+\int_0^s e^{\mathcal{L}_{c(T)}(s-s')}\Big[-\partial_x((\Phi^T)^2)-2\partial_x(\phi_{c(T)}\Phi^T)\Big]\d s' \\
     &+\int_0^se^{\mathcal{L}_{c(T)}(s-s')}\Big[\partial_c\Phi^Tc_d^{\sigma,\epsilon,T}-\partial_x\Phi^T\Omega_d^{\sigma,\epsilon,T}\Big]\d s'\nonumber\\
    &+\frac{\sigma^2}{2}\int_0^se^{\mathcal{L}_{c(T)}(s-s')}\Big[\partial^2_c\Phi^T\|Q^{1/2}c_s^{T}\|_{L^2}^2+\partial^2_x\Phi^T\|Q^{1/2}\xi_s^{T}\|_{L^2}^2\Big] \d s'\nonumber\\
    &+\sigma^2\int_0^se^{\mathcal{L}_{c(T)}(s-s')}\partial^2_{xc}\Phi^T\langle Q^{1/2}c_s^{T},Q^{1/2}\xi_s^T\rangle\d s'\nonumber\\
    &+\sigma  \int_0^se^{\mathcal{L}_{c(T)}(s-s')}\partial_c\Phi^T\langle c_s^{T},T_{\xi(T)+c(T)s'}\d W^Q_{T+s'}\rangle\nonumber\\
    &+\sigma \int_0^se^{\mathcal{L}_{c(T)}(s-s')}\partial_x\Phi^T\langle\xi_s^{T},T_{\xi(T)+c(T)s'}\d W^Q_{T+s'}\rangle.\nonumber
\end{align}
Next, we derive a mild formula for $u(T+s,x+\xi(T)+c(T)s)$ based on
the identity
\begin{align*}
    u(T+s)=&e^{-\partial_x^3 s}u(T)-\int_0^s e^{-\partial_x^3 (s-s')}\partial_x\big(u^2(T+s')\big)\d s'+\epsilon\int_0^s f(T+s')e^{-\partial_x^3 (s-s')}u(T+s')\d s'\\
    &+\sigma\int_0^s e^{-\partial_x^3 (s-s')}u(T+s')\d W_{T+s'}^Q.
\end{align*}
The mild It\^{o} formula \cite{mild} now yields
\begin{align*}
    u\big(T+s,x+\xi(T)+c(T)s\big)=&e^{-\partial_x^3s}u\big(T,x+\xi(T)\big)\\
    &-\int_0^s e^{-\partial_x^3(s-s')}\partial_x\big(u^2(T+s',\cdot+\xi(T)+c(T)s')\big)\d s'\\
    &+\epsilon\int_0^s f(T+s')e^{-\partial_x^3(s-s')}u\big(T+s',\cdot+\xi(T)+c(T)s'\big)\d s'\\
    &+c(T)\int_0^s e^{-\partial_x^3(s-s')}u_x\big(T+s',\cdot+\xi(T)+c(T)s'\big)\d s'\\
    &+\sigma \int_0^s e^{-\partial_x^3(s-s')}u\big(T+s',\cdot+\xi(T)+c(T)s'\big) T_{\xi(T)+c(T)s'}\d W^Q_{T+s'}.
\end{align*}
Using $-\partial_x^3+c(T)\partial_x=\mathcal{L}_{c(T)}+2\partial_x(\phi_{c(T)}\cdot)$, we rephrase the formula above in terms of the semigroup  $\{e^{\mathcal{L}_{c(T)}s}\}_{s\geq0}$ to find:
\begin{align}\label{eqn:mildU}
u\big(T+s,x+\xi(T)+c(T)s\big)=&e^{\mathcal{L}_{c(T)}s}u\big(T,x+\xi(T)\big)\\
&+2\int_0^se^{\mathcal{L}_{c(T)}(s-s')}\partial_x\big(\phi_{c(T)}u(T+s',\cdot+\xi(T)+c(T)s')\big)\d s'\nonumber\\
    &-\int_0^s e^{\mathcal{L}_{c(T)}(s-s')}\partial_x\big(u^2(T+s',\cdot+\xi(T)+c(T)s')\big)\d s'\nonumber\\
    &+\epsilon\int_0^s f(T+s')e^{\mathcal{L}_{c(T)}(s-s')}u\big(T+s',\cdot+\xi(T)+c(T)s'\big)\d s'\nonumber\\
    &+\sigma \int_0^s e^{\mathcal{L}_{c(T)}(s-s')}u\big(T+s',\cdot+\xi(T)+c(T)s'\big) T_{\xi(T)+c(T)s'}\d W^Q_{T+s'}. \nonumber
\end{align}
Subtracting the mild formula \eqref{eqn:mildPhi} from \eqref{eqn:mildU}, we arrive at
\begin{align*}
    v^T(s)=&e^{\mathcal{L}_{c(T)}s}v(T)-\int_0^s e^{\mathcal{L}_{c(T)}(s-s')}\big[\partial_x((v^T)^2)+2\partial_x\big((\phi_{c(T)}-\Phi^T)v^T\big)\big]\d s'\\
    &+\epsilon\int_0^s f(T+s')e^{\mathcal{L}_{c(T)}(s-s')}\big[\Phi^T+v^T\big]\d s'\\ 
     &-\int_0^se^{\mathcal{L}_{c(T)}(s-s')}\Big[\partial_c\Phi^Tc_d^{\sigma,\epsilon,T}-\partial_x\Phi^T\Omega_d^{\sigma,\epsilon,T}\Big]\d s'\\
    &-\frac{\sigma^2}{2}\int_0^se^{\mathcal{L}_{c(T)}(s-s')}\Big[\partial^2_c\Phi^T\|Q^{1/2}c_s^{T}\|_{L^2}^2+\partial^2_x\Phi^T\|Q^{1/2}\xi_s^{T}\|_{L^2}^2\Big] \d s'\\
    &-\sigma^2\int_0^se^{\mathcal{L}_{c(T)}(s-s')}\partial^2_{xc}\Phi^T\langle Q^{1/2}c_s^{T},Q^{1/2}\xi_s^T\rangle\d s'\\
    &-\sigma  \int_0^se^{\mathcal{L}_{c(T)}(s-s')}\partial_c\Phi^T\langle c_s^{T},T_{\xi(T)+c(T)s'}\d W^Q_{T+s'}\rangle\\
    &-\sigma \int_0^se^{\mathcal{L}_{c(T)}(s-s')}\partial_x\Phi^T\langle\xi_s^{T},T_{\xi(T)+c(T)s'}\d W^Q_{T+s'}\rangle\\
      &+\sigma \int_0^s e^{\mathcal{L}_{c(T)}(s-s')}[\Phi^T+v^T] T_{\xi(T)+c(T)s'}\d W^Q_{T+s'},
\end{align*}
as desired.
    % Substituting $u=v^T+\phi_{c^T}$ and using $\dot{\xi}^T=c^T+\dot{\Omega}^T$ we get
    % \begin{align*}
    %     \partial_tv^T=&-\partial_x^3 v^T-\partial_x^3\phi_{c^T} - \partial_x(v^T)^2-\partial_x (\phi_{c^T})^2-2\partial_x(\phi_{c^T}v^T)+c(T)\partial_xv^T-\dot{c}^T \partial_c\phi_{c^T}+(\dot{\Omega}^T+c^T)\partial_x\phi_{c^T},
    % \end{align*}
    % and using $\partial_x^3\phi_{c^T}+\partial_x(\phi_{c^T})^2=c^T\partial_x\phi_{c^T}$  we have
    %  \begin{align*}
    %     \partial_tv^T=&-\partial_x^3 v^T +N(v^T)-2\partial_x(\phi_{c^T}v^T)+c(T)\partial_xv^T-\dot{c}^T \partial_c\phi_{c^T}+\dot{\Omega}^T\partial_x\phi_{c^T}\\
    %     =&\mathcal{L}_{c(T)}v^T +2\partial_x((\phi_{c(T)}-\phi_{c^T})v^T)+N(v^T)+(-\dot{c}^T \partial_c+\dot{\Omega}^T\partial_x)\phi_{c^T}.
    % \end{align*}
\end{proof}

\begin{proof}[Proof of \Cref{thm:gaussian}]
Applying \cite[Theorem 4.5]{veraar} gives
\begin{align*}
    \mathbb{E}\sup_{t\in[0,T]} \big\|\int_0^t S(t-s)g(s)\d W^Q_s \big\|_{\mathcal{H}}^p\leq (K_p M\sqrt{p})^p T^{\frac{p}{2}-1}\int_0^T\mathbb{E}\Big[ \|g(t)\|^p_{\mathrm{HS}(L_Q^2,\mathcal{H})}\Big]\d t , 
\end{align*}
for $p\in(2,\infty)$, where $\limsup_{p\to \infty}K_p<\infty$. Assume without loss of generality that $K_p\leq K$ for $p>2$. By our assumption, 
\begin{align*}
    \mathbb{E}\sup_{t\in[0,T]} \big\|\int_0^t S(t-s)g(s)\d W^Q_s \big\|_{\mathcal{H}}^p\leq (BK M\sqrt{p}\sqrt{T})^p , 
\end{align*}
for $p>2$. Markov's inequality then gives
\begin{align*}
    \mathbb{P}\Big[ \sup_{t\in[0,T]}\big\|\int_0^t S(t-s)g(s)\d W^Q_s \big\|_{\mathcal{H}}^p \geq\lambda^p \Big] \leq (\lambda^{-1}BK M\sqrt{p}\sqrt{T})^p .
\end{align*}
For $\lambda > eBKM\sqrt{ T}$
 we may choose $p=(eBK M)^{-2}\lambda^2/T$ to conclude 
\[\mathbb{P}\Big[ \sup_{t\in[0,T]}\big\|\int_0^t S(t-s)g(s)\d W^Q_s \big\|_{\mathcal{H}} \geq\lambda \Big]=\mathbb{P}\Big[ \sup_{t\in[0,T]}\big\|\int_0^t S(t-s)g(s)\d W^Q_s \big\|_{\mathcal{H}}^p \geq\lambda^p \Big] \leq e^{-(eBKM)^{-2}\lambda^2/T}.\]
\end{proof}
\bibliographystyle{plain}
\bibliography{references}
\end{document}